\documentclass[letterpaper,reqno]{article}
\usepackage{amsmath}
\usepackage{amsfonts}
\usepackage{amssymb}
\usepackage{amstext}
\usepackage{amsbsy}
\usepackage{amsthm}
\usepackage{amscd}
\usepackage[all]{xy}

\newtheorem{theorem}{Theorem}[section]

\newtheorem{lem}[theorem]{Lemma}
\newtheorem{prop}[theorem]{Proposition}
\newtheorem{dfn}[theorem]{Definition}
\newtheorem{rem}[theorem]{Remark}
\newtheorem{cor}[theorem]{Corollary}

\newcommand{\Tt}{{\rm T}}
\newcommand{\Bb}{{\rm B}}
\newcommand{\Aa}{{\rm A}}
\newcommand{\Cc}{{\rm C}}
\newcommand{\Dd}{{\rm D}}

\newcommand{\Zee}{\mathbb{Z}}

\newcommand{\Cee}{\mathbb{C}}

\begin{document}

\title{Lattice Polarized K3 Surfaces and Siegel Modular Forms}
\author{Adrian Clingher
\thanks{
Department of Mathematics and Computer Science, University of Missouri - St. Louis, St. Louis  \ MO 63121. {\bf e-mail:} {\it clinghera@umsl.edu}
}
\and
Charles F. Doran\thanks{
Department of Mathematical and Statistical Sciences, University of Alberta, Edmonton AB T6G 2G1. {\bf e-mail:} {\it doran@math.ualberta.ca}
}
}
\date{}
\maketitle
\begin{center}
\abstract{
\noindent The goal of the present paper is two-fold. First, we present a classification of algebraic K3 surfaces polarized 
by the lattice ${\rm H} \oplus {\rm E}_8 \oplus {\rm E}_7$. Key ingredients for this classification are: a normal form for 
these lattice polarized K3 surfaces, a coarse moduli space and an explicit description of the inverse period map in terms 
of Siegel modular forms. Second, we give explicit formulas for a Hodge correspondence that relates these K3 surfaces to 
principally polarized abelian surfaces. The Hodge correspondence in question underlies a geometric two-isogeny of K3 
surfaces, the details of which are described in \cite{clingher5}.    
}
\end{center}

\section{Introduction}
Let ${\rm X}$ be an algebraic K3 surface defined over the field of complex numbers. Denote by ${\rm NS}({\rm
X})$ the N\'{e}ron-Severi lattice of ${\rm X}$. This is an even lattice of signature $(1, {\rm p}_{{\rm X}}-1)$
where ${\rm p}_{{\rm X}}$ is the Picard rank. By definition (see \cite{dolga1}), a {\it lattice polarization}
on the surface ${\rm X}$ is given by a primitive lattice embedding
$$ i \colon {\rm N} \ \hookrightarrow \ {\rm NS}({\rm X}) $$
whose image contains a pseudo-ample class. Here ${\rm N}$ is a choice of even lattice of signature 
$(1,r)$ with $ 0 \leq r \leq 19$. Two ${\rm N}$-polarized
K3 surfaces $({\rm X},i)$ and $({\rm X}',i')$ are said to be isomorphic
if there exists an analytic isomorphism $ \alpha \colon {\rm X} \rightarrow {\rm X}'$
such that $ \alpha^* \circ i' = i$, where $\alpha^*$ is the appropriate cohomology morphism.
\par The present paper concerns the special class of K3 surfaces polarized by the even lattice of rank seventeen 
$$ {\rm N} = {\rm H} \oplus {\rm E}_8 \oplus {\rm E}_7 . $$
Here ${\rm H}$ stands for the standard hyperbolic lattice of
rank two and ${\rm E}_8$, ${\rm E}_7$ are negative definite lattices associated with the corresponding exceptional root systems. 
Surfaces in this class have Picard ranks taking four possible values: $17$, $18$, $19$ or $20$.
\par This special class of algebraic K3 surfaces is of interest because of a remarkable Hodge-theoretic feature. 
Any given ${\rm N}$-polarized K3 surface 
$({\rm X},i)$ is associated uniquely with a well-defined principally polarized complex abelian surface 
$({\rm A}, \Pi)$. This feature 
appears due to the fact that both types of surfaces mentioned above are classified, via appropriate versions of Torelli Theorem, 
by a Hodge structure of weight two on ${\rm T} \otimes \mathbb{Q}$ where ${\rm T}$ is the rank-five lattice 
${\rm H} \oplus {\rm H} \oplus (-2) $. This fact determines a bijective map:
\begin{equation}
\label{basiccor}
({\rm X},i) \ \leftrightarrow \ ({\rm A}, \Pi)
\end{equation}
which is a Hodge correspondence. In fact, the map $(\ref{basiccor})$ can be regarded as a particular case of a more general 
Hodge-theoretic construction due to Kuga and Satake \cite{kugasatake}. In particular, map $(\ref{basiccor})$ realizes an analytic 
identification between the moduli spaces of periods associated with the two types of surfaces, both of which could be seen 
as the classical Siegel modular threefold 
$  \mathcal{F}_2 =  {\rm Sp}_4(\mathbb{Z}) \backslash {\mathbb H}_2. $
\par The correspondence given by $(\ref{basiccor})$ can be further refined. The set of all isomorphism classes of ${\rm N}$-polarized K3 surfaces divides naturally
into two disjoint subclasses. The first subclass consists of those surfaces $({\rm X},i)$ for which the
lattice polarization $i$ extends canonically to a polarization by the unimodular rank-eighteen lattice 
$ {\rm M} = {\rm H} \oplus {\rm E}_8 \oplus {\rm E}_8$. 
In terms of the Siegel modular threefold $\mathcal{F}_2$, this
subclass is associated with the Humbert surface usually denoted by $ {\mathcal H}_1 $. Under $(\ref{basiccor})$, 
the principally polarized abelian surface $({\rm A}, \Pi)$ associated to a ${\rm
M}$-polarized K3 surfaces $({\rm X},i)$ is of the form:
$$ \left ( \ {\rm E}_1 \times {\rm E}_2, \ \mathcal{O}_{{\rm E}_1 \times {\rm E}_2}(\,({\rm E}_1 \times \{p_2\}) +
(\{p_1\} \times {\rm E}_2)\,) \   \right ) $$ where $({\rm E}_1,p_1)$ and $({\rm E}_2,p_2)$ are complex elliptic
curves, uniquely determined up to permutation.
\par The second subclass is given by those ${\rm N}$-polarized K3 surfaces $({\rm X}, i)$ for which the lattice 
polarization cannot be extended from ${\rm N}$ to ${\rm M}$. These surfaces correspond in the Siegel threefold 
to the open region $ \mathcal{F}_2 \setminus {\mathcal H}_1 $. Their associated principally polarized abelian surfaces 
$({\rm A}, \Pi)$ are of the form:
$$ \left ( \ {\rm Jac}({\rm C}), \ \mathcal{O}_{{\rm Jac}({\rm C})}(\Theta) \   \right ) $$
where ${\rm C}$ is a non-singular complex genus-two curve and $\Theta$ is the theta-divisor, the image of ${\rm C}$ under the
Abel-Jacobi  embedding. The genus-two curve ${\rm C}$ is uniquely determined
by the pair $({\rm X},i)$ and $(\ref{basiccor})$ provides an analytic identification between
$\mathcal{F}_2 \setminus {\mathcal H}_1$ and the moduli space $ \mathcal{M}_2$ of complex genus-two curves.
\par The goal of the present paper is two-fold. First, we present a full classification theory for ${\rm N}$-polarized K3 surfaces 
along the lines of the classical theory of elliptic curves defined over the field of complex numbers. Second, we give explicit 
formulas for the correspondence $(\ref{basiccor})$ in terms of Siegel modular forms. 
\par A key ingredient for the results of this paper is the introduction of a {\it normal form} associated to K3 surfaces with 
${\rm N}$-polarizations. It will be instructive to first recall the classical Weierstrass normal form for complex elliptic curves 
and to trace our results in parallel
with that case. 
\begin{theorem}
Let $(g_2, g_3)$ be a pair of complex numbers.  Denote by ${\rm E}(g_2, g_3)$ the curve in $\mathbb{P}^2(x,y,z)$ cut out by the degree-three
homogeneous equation:
\begin{equation}
y^2z - 4 x^3 + g_2 xz^2 + g_3 z^3 \ = \ 0 \ .
\end{equation}
\begin{itemize}
\item [(a)] If $\Delta := g_2^3 - 27 g_3^2$ is nonzero, then ${\rm E}(g_2,g_3)$ is an elliptic curve.
\item [(b)] Given any elliptic curve ${\rm E}$, there exists $(g_2, g_3) \in \mathbb{C}^2 $, with $\Delta \neq 0$,
such that the curves ${\rm E}$ and ${\rm E}(g_2, g_3)$ are isomorphic as elliptic curves.
\end{itemize}
\end{theorem}
\noindent Our first result in this paper is analogus to the above.
\begin{theorem}
\label{theonormalform}
Let $(\alpha, \beta, \gamma, \delta)$ be a quadruple of complex numbers. Denote by 
${\rm X}(\alpha,\beta, \gamma, \delta)$ the minimal resolution of the surface in $\mathbb{P}^3(x,y,z,w)$ 
cut out by the degree-four homogeneous equation:
\begin{equation}
\label{nooform}
y^2zw - 4 x^3z + 3 \alpha xzw^2 +\beta zw^3 + \gamma xz^2w
- \frac{1}{2} ( \delta z^2 w^2 + w^4 ) \ = \ 0 \ .
\end{equation}
\begin{itemize}
\item [(a)] If $\gamma \neq 0$ or $\delta \neq 0$, then ${\rm X}(\alpha,\beta,\gamma, \delta)$ is a K3 surface
endowed with a canonical ${\rm N}$-polarization. 
\item [(b)] Given any ${\rm N}$-polarized K3 surface ${\rm X}$, there exists $(\alpha,\beta, \gamma, \delta) \in \mathbb{C}^4 $,
with $\gamma \neq 0$ or $\delta \neq 0$, such that surfaces ${\rm X}$ and ${\rm X}(\alpha,\beta, \gamma, \delta)$ are isomorphic as 
${\rm N}$-polarized K3 surfaces.
\end{itemize}
\end{theorem}
\noindent The quartic $(\ref{nooform})$ extends a two-parameter family of K3 surfaces given by Inose in \cite{inose1}. In the context of $(\ref{nooform})$, 
the special case $\gamma=0$ corresponds to the situation when the polarization extends to the lattice ${\rm H} \oplus {\rm E}_8 \oplus {\rm E}_8$, 
whereas the ${\rm N}$-polarizations of K3 surfaces ${\rm X}(\alpha,\beta, \gamma, \delta)$ with $\gamma \neq 0$ cannot be extended to 
${\rm H} \oplus {\rm E}_8 \oplus {\rm E}_8$. 
\par As it turns out, the normal forms $(\ref{nooform})$ are also ideal objects for establishing a moduli space for isomorphism classes of 
${\rm N}$-polarizations of K3 surfaces. Again let us first recall the classical case of Weierstrass elliptic curves.
\begin{theorem}
Two curves ${\rm E}(g_2,g_3)$ and ${\rm E}(g_2', g_3')$ are isomorphic as elliptic curves if and only if there exists $t \in \mathbb{C}^*$ such that: 
$$
\left ( \ g_2', \ g_3'\ \right) \ = \
\left ( \ t^{2} g_2, \ t^{3} g_3 \ \right ) . $$
The open variety:
$$ \mathcal{M}_{{\rm E}} \ = \ \left \{ \ [ \ g_2,g_3 \ ] \in {\mathbb W}{\mathbb P}^2(2,3) \ \middle \vert \ \Delta \neq 0 \ \right\} $$
forms a coarse moduli space for elliptic curves.
\end{theorem}
\noindent In the above context, the {\em $j$-invariant}
$${\rm j}({\rm E}) := \frac{g_2^3}{\Delta}$$
identifies $\mathcal{M}_{{\rm E}}$ and $\mathbb{C}$ (the ``$j$-line'').  The period map to the classifying space of Hodge structures is the isomorphism
of quasi-projective varieties:
\begin{equation}
{\rm per} \colon \mathcal{M}_{{\rm E}} \ \rightarrow \ \mathcal{F}_1 = {\rm PSL}_2(\mathbb{Z}) \backslash {\mathbb H}
\end{equation}
whose inverse is given by:
$$ {\rm per}^{-1} \ = \ \left [ \ 60 {\rm E}_4, \ 140 {\rm E}_6 \ \right ]$$
where ${\rm E}_4, {\rm E}_6 \ : \ {\mathbb H} \rightarrow {\mathbb C}$ are the classical Eisenstein series of weights four and six, respectively.
\par In our ${\rm N}$-polarized K3 surface setting, we have following analogous result:
\begin{theorem}
\label{theomoduli}
Two K3 surfaces ${\rm X}(\alpha_1,\beta_1, \gamma_1, \delta_1)$ and ${\rm X}(\alpha_2,\beta_2, \gamma_2, \delta_2)$ 
are isomorphic as ${\rm N}$-polarized K3 surfaces if and only if there exists $t \in \mathbb{C}^*$ such that: 
$$ 
\left (  \ \alpha_2, \ \beta_2, \ \gamma_2, \ \delta_2\ \right ) \ = \ 
\left ( \ t^2 \alpha_1, \ t^3 \beta_1, \ t^5 \gamma_1, \ t^6 \delta_1 \ \right ) . $$
The open variety:
$$ \mathcal{M}^{{\rm N}}_{{\rm K3}} \ = \ \left \{ \
[ \ \alpha,\beta,\gamma, \delta \ ] \in  {\mathbb W}{\mathbb P}^3(2,3,5,6) \ \middle \vert \ \gamma \neq 0 \ {\rm or} \ \delta \neq 0
\ \right \} $$
forms a coarse moduli space for ${\rm N}$-polarized K3 surfaces.
\end{theorem}
\noindent In the context of Theorem $\ref{theomoduli}$, the period map to the associated classifying space of Hodge structures appears as 
a morphism of quasi-projective varieties:
\begin{equation}
\label{pperiodmap}
{\rm per} \colon \mathcal{M}^{{\rm N}}_{{\rm K3}}  \ \rightarrow \ \mathcal{F}_2  =  {\rm Sp}_4(\mathbb{Z}) \backslash {\mathbb H}_2.
\end{equation}
By the appropriate version of Global Torelli Theorem for lattice polarized K3 surfaces (see for instance \cite{dolga1}), one knows that 
$(\ref{pperiodmap})$ is in fact an isomorphism. We prove that the inverse of $(\ref{pperiodmap})$ has a simple description in terms of 
Siegel modular forms.   
\begin{theorem}
\label{theoinverseperiod}
The inverse period map $ {\rm per}^{-1} \colon \mathcal{F}_2 \rightarrow \mathcal{M}_{{\rm K3}} $ 
is given by:
$$ {\rm per}^{-1} \ = \ 
\left [ \ \mathcal{E}_4, \  \mathcal{E}_6, \ 2^{12}  3^{5}  \mathcal{C}_{10}, \ 2^{12} 3^{6}  \mathcal{C}_{12}  \ \right ]  $$ 
where $\mathcal{E}_4$, $\mathcal{E}_6$ are genus-two Eisenstein series of weight four and six, and $\mathcal{C}_{10}$ and $\mathcal{C}_{12}$ 
are Igusa's cusp forms of weight $10$ and $12$, respectively.
\end{theorem} 
\noindent Theorems $\ref{theonormalform}$, $\ref{theomoduli}$ and $\ref{theoinverseperiod}$ allow one to give an explicit description of the 
dual principally polarized abelian surfaces associated by the Hodge-theoretic correspondence $(\ref{basiccor})$. In the case of K3 surfaces 
polarized by the lattice ${\rm H} \oplus {\rm E}_8 \oplus {\rm E}_8$, explicit formulas were given previously by the authors \cite{clingher3} 
as well as Shioda \cite{shioda06}. 
\begin{theorem}
\label{tthe8e8}
Under the duality correspondence $(\ref{basiccor})$, the principally polarized abelian surface ${\rm A}$ associated to $ {\rm X}(\alpha,\beta, 0, \delta)$ is given by 
$$ \left ( \ {\rm E}_1 \times {\rm E}_2, \ \mathcal{O}_{{\rm E}_1 \times {\rm E}_2} \left ( {\rm E}_1 + {\rm E}_2 \right ) \   \right ) $$ 
where ${\rm E}_1$ and ${\rm E}_2$ are complex elliptic curves with j-invariants satisfying:
$$ {\rm j}({\rm E}_1) + {\rm j}({\rm E}_2) \ = \ \frac{\alpha^3-\beta^2}{\delta} + 1, \ \ \ 
{\rm j}({\rm E}_1) \cdot {\rm j}({\rm E}_2) \ = \ \frac{\alpha^3}{\delta}.
$$
\end{theorem}  
\noindent In this paper, we use the formulas of Theorem $\ref{theoinverseperiod}$ in order to explicitly identify the 
genus-two curves ${\rm C}$ corresponding to the remaining case, by computing the Igusa-Clebsch 
invariants $ \left [ \ \mathcal{A}, \ \mathcal{B}, \ \mathcal{C} , \ \mathcal{D} \ \right ] \in {\mathbb W}{\mathbb P}^3(2,4,6,10) $ associated 
with these curves.
\begin{theorem}
Assume $\gamma \neq 0$. Under the duality correspondence $(\ref{basiccor})$, the principally polarized abelian surface ${\rm A}$ associated to 
$ {\rm X}(\alpha,\beta, \gamma, \delta)$ is given by 
$$
\left ( \ {\rm Jac}({\rm C}), \ \mathcal{O}_{{\rm Jac}({\rm C})}(\Theta) \   \right ) 
$$
where ${\rm C}$ is a smooth genus-two curve of Igusa-Clebsch invariants:
$$ \left [ \ \mathcal{A}, \ \mathcal{B}, \ \mathcal{C} , \ \mathcal{D} \ \right ] \ = \ 
\left [ \ 2^3 3 \delta, \ 2^2  3^2 \alpha \gamma^2, \ 2^3  3^2 (4\alpha \delta + \beta \gamma) \gamma^2, \ 2^2 \gamma^6    \ \right ]. 
$$
\end{theorem}
\noindent The present paper should be considered in connection with the companion paper \cite{clingher5}. This is because the proofs of the 
theorems mentioned above do not involve period computations. They rather rely on a very specific observation:  
the Hodge-theoretic correspondence $(\ref{basiccor})$
is a consequence of a purely geometric phenomenon - the existence of a pair of dual geometric two-isogenies of K3 surfaces
between the ${\rm N}$-polarized surface ${\rm X}$ and the Kummer surface ${\rm Y}$ associated to the abelian 
surface ${\rm A}$ corresponding to $({\rm X}, i)$ under $(\ref{basiccor})$. The precise meaning of this isogeny concept is 
explained in \cite{clingher5}. In short, the observation consists of the existence of two special 
Nikulin involutions $\Phi_{{\rm X}}$ and $\Phi_{{\rm Y}}$ ,
acting on the surfaces ${\rm X}$ and ${\rm Y}$, respectively, which lead to degree-two rational maps ${\rm
p}_{{\rm X}}$ and ${\rm p}_{{\rm Y}}$. The involutions $\Phi_{{\rm X}}$ and $\Phi_{{\rm Y}}$ are
associated naturally with two canonical elliptic fibrations $\varphi_{{\rm
X}}$ and $\varphi_{{\rm Y}}$ on ${\rm X}$ and ${\rm Y}$ over a base rational
curve ${\rm B}$. The involutions are fiberwise two-isogenies in the sense that they correspond to translations 
by sections of order-two within the smooth fibers of the fibrations $\varphi_{{\rm
X}}$ and $\varphi_{{\rm Y}}$.
\begin{equation}
\xymatrix 
{ {\rm Y} \ar @(dl,ul) _{\Phi_{{\rm Y}}} \ar [dr] _{\varphi_{{\rm Y}}}
\ar @/_0.5pc/ @{-->} _{{\rm p}_{{\rm Y}}} [rr] &
& {\rm X} \ar @(dr,ur) ^{\Phi_{{\rm X}}}
\ar [dl] ^{\varphi_{{\rm X}}}
\ar @/_0.5pc/ @{-->} _{{\rm p}_{{\rm X}}} [ll] \\
& {\rm B} & \\
} \label{introdiag}
\end{equation}
\noindent The above geometric phenomenon allows one to make the duality map explicit, without involving an analysis of 
Hodge structures or period computations.  
\par The present work focuses on the case of ${\rm N}$-polarized K3 surfaces for which the lattice polarization does not extend to 
a polarization by the lattice ${\rm M}={\rm H}\oplus {\rm E}_8 \oplus {\rm E}_8$. The case involving ${\rm M}$-polarizations has 
been presented in \cite{clingher4}, work on which 
the present paper builds. 
\par Various partial ingredients pertaining to this construction have been discussed by the authors and others in 
earlier works. In his 1977 work \cite{inose1}, Inose presented a normal form for K3 surfaces and constructed the 
Nikulin involution $\Phi_{{\rm Y}}$ on the Kummer surface associated with the product of two elliptic curves. 
The construction of $\Phi_{{\rm Y}}$ in Inose's context uses a different elliptic fibration with respect to which 
the Nikulin involution is not a fiberwise isogeny. In the paper \cite{clingher4}, the authors constructed each piece of diagram
(\ref{introdiag}) in the case of ${\rm M}$-polarized K3 surfaces, including explicit equations for both elliptic
fibrations $\varphi_{{\rm X}}$ and $\varphi_{{\rm Y}}$.  This case was also treated in \cite{shioda06} by Shioda. 
One particular sub-family of ${\rm M}$-polarized K3 surfaces, with generic Picard lattice enhanced to ${\rm H} \oplus
{\rm E}_8 \oplus {\rm E}_8 \oplus \langle -4 \rangle$ was considered by Van Geemen and Top in \cite{geementop}. 
The Van Geemen-Top family corresponds, in terms of the duality $(\ref{basiccor})$, to pairs of two-isogeneous 
elliptic curves.  
\par An indication that the construction can be extended from ${\rm M}$-polarized to ${\rm N}$-polarized K3 surfaces
was given by Dolgachev in his appendix to the paper \cite{galluzzi} by Galluzzi and Lombardo, where, 
based on an analysis of Fourier-Mukai partners, they observe that K3 surfaces with N\'{e}ron-Severi lattice
exactly ${\rm N}$ are in correspondence with Jacobians of genus two curves. The present paper has its origin in 
Dolgachev's observation. The authors extend the geometric arguments from \cite{clingher4} to the full 
${\rm N}$-polarized case by constructing in detail the two-isogenies between the ${\rm N}$-polarized K3 surfaces and 
their dual Kummer surfaces of principally polarized abelian surfaces. An explicit computation based on 
parts of this construction was made by Kumar \cite{kumar}. \\

\section{A Four-Parameter Quartic Family}
\label{extendedinosesection}
\begin{dfn}
Consider $(\alpha, \beta, \gamma, \delta) \in \mathbb{C}^4 $. Let
${\rm Q}(\alpha,\beta, \gamma, \delta)$ be the projective quartic surface in $\mathbb{P}^3(x,y,z,w)$ given by:
\begin{equation}
\label{extendedinose}
y^2zw - 4 x^3z + 3 \alpha xzw^2 +\beta zw^3 + \gamma xz^2w
- \frac{1}{2} ( \delta z^2 w^2 + w^4 ) \ = \ 0 \ .
\end{equation}
Denote by
${\rm X}(\alpha,\beta,\gamma, \delta)$ the non-singular complex surface obtained as the minimal resolution of ${\rm Q}(\alpha,\beta, \gamma, \delta)$.
\end{dfn}
\noindent The four-parameter quartic family ${\rm Q}(\alpha,\beta, \gamma, \delta)$ generalizes a special two-parameter family of K3 surfaces introduced 
by Inose in \cite{inose1}. 
\begin{theorem}
\label{tth1}
If $\gamma \neq 0$ or $\delta \neq 0$, then ${\rm X}(\alpha,\beta,\gamma, \delta)$ is a K3 surface
endowed with a canonical ${\rm N}$-polarization.
\end{theorem}
\begin{proof}
The hypothesis $\gamma \neq 0$ or $\delta \neq 0$ ensures that the singular locus of the quartic surface 
${\rm Q}(\alpha,\beta, \gamma, \delta)$ consists of a finite collection of rational double points. This 
fact implies, in turn, that ${\rm X}(\alpha,\beta,\gamma, \delta)$ is a K3 surface.
\par Let us present the ${\rm N}$-polarization on ${\rm X}(\alpha,\beta,\gamma, \delta)$. Note that 
${\rm Q}(\alpha,\beta, \gamma, \delta)$ has two special singular points
$$ {\rm P}_1 = [0,1,0,0], \ \ \ \ {\rm P}_2 = [0,0,1,0].  $$
For a generic choice of quadruple $(\alpha,\beta, \gamma, \delta)$, the singular locus of ${\rm Q}(\alpha,\beta, \gamma, \delta)$ is
precisely $\{ {\rm P}_1,  {\rm P}_2 \} $. Under the condition $\gamma \neq 0$ or $ \delta \neq 0$, both ${\rm P}_1$ and ${\rm P}_2$ are rational
double point singularities. The point ${\rm P}_1$ is always a rational double
point of type ${\rm A}_{11}$. The type of the rational double point ${\rm P}_2$ is covered by two situations. If $\gamma \neq 0$ then
${\rm P}_2$ has type ${\rm A}_5$. When $\gamma=0$, the singularity at ${\rm P}_2$ is of type ${\rm E}_6$.
\par The intersection locus of the quartic ${\rm Q}(\alpha,\beta, \gamma, \delta)$ with the plane of equation $w=0$ consists of
two distinct lines, ${\rm L}_1$ and ${\rm L}_2$, defined by $x=w=0$ and $z=w=0$, respectively. In addition, when $\gamma \neq 0$
one has an additional special curve on ${\rm Q}(\alpha, \beta, \gamma, \delta)$ obtained as the intersection of the plane of equation
$2 \gamma x = \delta w$ with the cubic surface
$$ 2\gamma^3 y^2 z + (-\delta^3 + 3 \alpha \gamma^2 \delta + 2 \beta \gamma^3) z w^2 - \gamma^3 w^3 = 0  \ .$$
This curve resolves to a rational curve in ${\rm X}(\alpha,\beta, \gamma, \delta)$ which we denote by $d$.
\par After resolving the singularities at ${\rm P}_1$ and ${\rm P}_2$ one obtains a special configuration
of rational curves on ${\rm X}(\alpha, \beta, \gamma, \delta)$. The dual diagram of this configuration, in the two cases in question, 
is presented below.
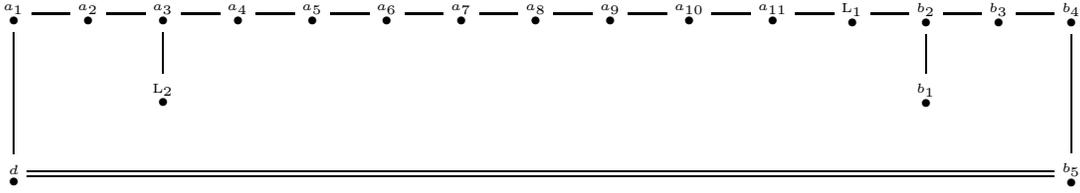
\begin{figure}[h]
\begin{center}
\leavevmode
\def\objectstyle{\scriptstyle}
\def\labelstyle{\scriptstyle}
\xymatrix @-0.8pc  {
\stackrel{a_1}{\bullet} \ar @{-} [r] \ar @{-}[dd]
& \stackrel{a_2}{\bullet} \ar @{-} [r]&
\stackrel{a_3}{\bullet} \ar @{-} [r] \ar @{-} [d] &
\stackrel{a_4}{\bullet} \ar @{-} [r] &
\stackrel{a_5}{\bullet} \ar @{-} [r] &
\stackrel{a_6}{\bullet} \ar @{-} [r] &
\stackrel{a_7}{\bullet} \ar @{-} [r] &
\stackrel{a_8}{\bullet} \ar @{-} [r] &
\stackrel{a_9}{\bullet} \ar @{-} [r] &
\stackrel{a_{10}}{\bullet} \ar @{-} [r] &
\stackrel{a_{11}}{\bullet} \ar @{-} [r] &
\stackrel{{\rm L}_{1}}{\bullet} \ar @{-} [r] &
\stackrel{b_{2}}{\bullet} \ar @{-} [d] & \stackrel{b_{3}}{\bullet} \ar @{-} [l]
 & \stackrel{b_{4}}{\bullet} \ar @{-} [l] \ar @{-}[dd]\\
 & & \stackrel{{\rm L}_{2}}{\bullet} &  & & & & &   &    & & & \stackrel{b_{1}}{\bullet} & \\
 \stackrel{d}{\bullet} \ar @{=}[rrrrrrrrrrrrrr]& & &   & & &  & & &    & & & & & \stackrel{b_{5}}{\bullet}  \\
} \end{center} \caption{$ {\rm Case} \ \gamma \neq 0 $} \label{Gneq0Config}
\end{figure}
\vspace{.2in}
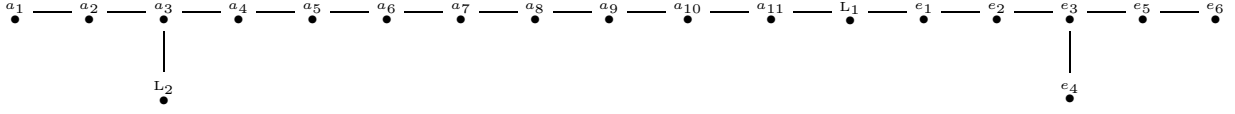
\begin{figure}[h]
\begin{center}
\leavevmode
\def\objectstyle{\scriptstyle}
\def\labelstyle{\scriptstyle}
\xymatrix @-0.81pc  {
\stackrel{a_1}{\bullet} \ar @{-} [r]
& \stackrel{a_2}{\bullet} \ar @{-} [r]& \stackrel{a_3}{\bullet} \ar @{-} [r] \ar @{-} [d] &
\stackrel{a_4}{\bullet} \ar @{-} [r] &
\stackrel{a_5}{\bullet} \ar @{-} [r] &
\stackrel{a_6}{\bullet} \ar @{-} [r] &
\stackrel{a_7}{\bullet} \ar @{-} [r] &
\stackrel{a_8}{\bullet} \ar @{-} [r] &
\stackrel{a_9}{\bullet} \ar @{-} [r] &
\stackrel{a_{10}}{\bullet} \ar @{-} [r] &
\stackrel{a_{11}}{\bullet} \ar @{-} [r] &
\stackrel{{\rm L}_{1}}{\bullet} \ar @{-} [r] &
\stackrel{e_{1}}{\bullet} \ar @{-} [r] &
\stackrel{e_{2}}{\bullet} \ar @{-} [r] &
\stackrel{e_{3}}{\bullet} \ar @{-} [d] & \stackrel{e_{5}}{\bullet} \ar @{-} [l]
 & \stackrel{e_{6}}{\bullet} \ar @{-} [l] \\
 & & \stackrel{{\rm L}_{2}}{\bullet} &  & & & & &  & & &    & & & \stackrel{e_{4}}{\bullet} & \\
} \end{center} \caption{$ {\rm Case} \ \gamma = 0 $} \label{Geq0Config}
\end{figure}

\noindent Note that when $\gamma \neq 0$ one has the following ${\rm E}_8 \oplus {\rm E}_7$ configuration.
$$
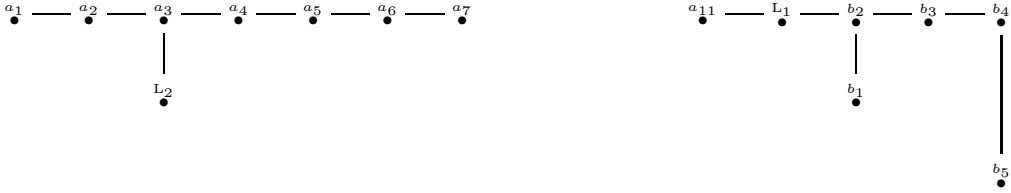
\def\objectstyle{\scriptstyle}
\def\labelstyle{\scriptstyle}
\xymatrix @-0.8pc  {
\stackrel{a_1}{\bullet} \ar @{-} [r] 
& \stackrel{a_2}{\bullet} \ar @{-} [r]&
\stackrel{a_3}{\bullet} \ar @{-} [r] \ar @{-} [d] &
\stackrel{a_4}{\bullet} \ar @{-} [r] &
\stackrel{a_5}{\bullet} \ar @{-} [r] &
\stackrel{a_6}{\bullet} \ar @{-} [r] &
\stackrel{a_7}{\bullet} 
&
&
&
&
\stackrel{a_{11}}{\bullet} \ar @{-} [r] &
\stackrel{{\rm L}_{1}}{\bullet} \ar @{-} [r] &
\stackrel{b_{2}}{\bullet} \ar @{-} [d] & \stackrel{b_{3}}{\bullet} \ar @{-} [l]
 & \stackrel{b_{4}}{\bullet} \ar @{-} [l] \ar @{-}[dd]\\
 & & \stackrel{{\rm L}_{2}}{\bullet} &  & & & & &   &    & & & \stackrel{b_{1}}{\bullet} & \\
& & &   & & &  & & &    & & & & & \stackrel{b_{5}}{\bullet}  \\
} $$

\noindent When $\gamma=0$, one has a similar ${\rm E}_8 \oplus {\rm E}_8$ configuration of curves.
$$
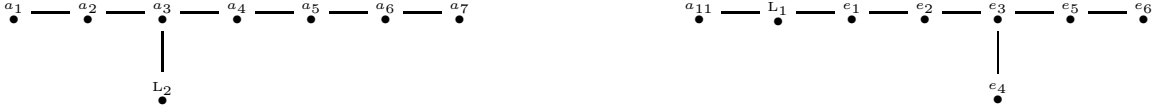
\def\objectstyle{\scriptstyle}
\def\labelstyle{\scriptstyle}
\xymatrix @-0.81pc  {
\stackrel{a_1}{\bullet} \ar @{-} [r]
& \stackrel{a_2}{\bullet} \ar @{-} [r]& \stackrel{a_3}{\bullet} \ar @{-} [r] \ar @{-} [d] &
\stackrel{a_4}{\bullet} \ar @{-} [r] &
\stackrel{a_5}{\bullet} \ar @{-} [r] &
\stackrel{a_6}{\bullet} \ar @{-} [r] &
\stackrel{a_7}{\bullet} 
&
&
&
&
\stackrel{a_{11}}{\bullet} \ar @{-} [r] &
\stackrel{{\rm L}_{1}}{\bullet} \ar @{-} [r] &
\stackrel{e_{1}}{\bullet} \ar @{-} [r] &
\stackrel{e_{2}}{\bullet} \ar @{-} [r] &
\stackrel{e_{3}}{\bullet} \ar @{-} [d] & \stackrel{e_{5}}{\bullet} \ar @{-} [l]
 & \stackrel{e_{6}}{\bullet} \ar @{-} [l] \\
 & & \stackrel{{\rm L}_{2}}{\bullet} &  & & & & &  & & &    & & & \stackrel{e_{4}}{\bullet} & \\
}
$$

\noindent The remaining orthogonal hyperbolic lattice ${\rm H}$ is spanned by the two classes associated to divisors $a_9$
and $f$ where:
\begin{equation}
\label{ttr1}
f \ = \ a_8 + 2 a_ 7 + 3 a_6 + 4a_5 + 5 a_4 + 6 a_3 + 3 {\rm L}_2 + 4 a_2 + 2 a_1 \ = \ a_{10} + 2 a_{11} + 3 {\rm L}_1 + 4 b_2 + 2b_1 + 3b_3 + 2b_2 + b_1
\end{equation}
when $\gamma \neq 0 $ and
\begin{equation}
\label{ttr2}
f \ = \ a_8 + 2 a_ 7 + 3 a_6 + 4a_5 + 5 a_4 + 6 a_3 + 3 {\rm L}_2 + 4 a_2 + 2 a_1 \ = \ a_{10} + 2 a_{11} + 3 {\rm L}_1 + 4 e_1 + 5 e_2 + 6 e_3 + 3e_4 + 2e_5 + e_6
\end{equation}
when $\gamma=0$.
\end{proof}
\begin{rem}
The surfaces ${\rm X}(\alpha, \beta, 0, 0)$ are rational surfaces. On the projective quartic
surface ${\rm Q}(\alpha, \beta, 0, 0)$ the singularity at ${\rm P}_2$ is no longer a rational double point, but an elliptic singularity.
\end{rem}
\noindent Let us briefly discuss the discriminant locus of the family of quartics ${\rm Q}(\alpha,\beta,
\gamma,\delta)$. One particular discriminant component, given by the vanishing of
\begin{equation}
\label{discrr1}
\mathcal{D}_1(\alpha, \beta, \gamma, \delta) = \gamma
\end{equation}
was already mentioned during the
proof of Theorem $\ref{tth1}$. This component corresponds (on its $\delta \neq 0$ region) to the situation when the ${\rm N}$-polarization of
the K3 surface ${\rm X}(\alpha,\beta,\gamma,\delta)$ extends canonically to a lattice polarization by $ {\rm H} \oplus {\rm E}_8 \oplus {\rm E}_8 $. 
\par A second important discriminant component corresponds to the situation when,
in addition to the singular points ${\rm P}_1$ and ${\rm P}_2$, extra singularities occur on the quartic surface ${\rm
Q}(\alpha,\beta,\gamma,\delta)$. One can check that this situation is determined by the vanishing of the polynomial:
$$ \mathcal{D}_4(\alpha, \beta, \gamma, \delta) \ = \ -2^53^6 \alpha^6 \beta \gamma^3 \ + \ 2^6 3^6 \alpha^3 \beta^3 \gamma^3 \ - \ 2^5 3^6 \beta^5 \gamma^3
\ - \ 2^4 3^5 \alpha^5 \gamma^4 \ + \ 2^4 3^5 5^2  \alpha^2 \beta^2 \gamma^4 \ + \ 2 \cdot 3^3 5^4 \alpha \beta \gamma^5 \ +  $$
$$ \ + \ 5^5 \gamma^6 \ - \  2^4 3^7 \alpha^7 \gamma^2 \delta \
\ + \ 2^5 3^7 \alpha^4 \beta^2 \gamma^2 \delta \ - \
2^4 3^7  \alpha \beta^4 \gamma^2 \delta \ + \ 2^33^5 5 \cdot 19  \alpha^3 \beta \gamma^3 \delta \ + \
2^3 3^5 5^2 \beta^3 \gamma^3 \delta \ + $$
$$  +  \ 3^3 5^3 11 \alpha^2 \gamma^4 \delta \ + \ 2^3 3^5 37 \alpha^4 \gamma^2 \delta^2 + \
2^3 3^5 5 \cdot 7  \alpha \beta^2 \gamma^2 \delta^2 \ - \ 2^3 3^3 5^3  \beta \gamma^3 \delta^2 \ + \ 2^4 3^6  \alpha^6 \delta^3
 \ - \ 2^5 3^6  \alpha^3 \beta^2 \delta^3 \ +  $$
\begin{equation}
\label{hum4}
   + \ 2^4 3^6  \beta^4 \delta^3 \ - \ 2^6 3^6  \alpha^2 \beta \gamma \delta^3 \ - \
2^3 3^5 5^2  \alpha \gamma^2 \delta^3 \ - \ 2^5 3^6  \alpha^3 \delta^4 \ - \ 2^5 3^6 \beta^2 \delta^4 \ + \ 2^4 3^6 \delta^5
\end{equation}
We shall see later (Remark $\ref{remarkondiscr}$), a more precise interpretation of the polynomial in $ (\ref{hum4})$. 
At this point, we note that, in terms of the ${\rm N}$-polarized K3 surfaces ${\rm X}(\alpha, \beta, \gamma, \delta) $, this discriminant 
component corresponds to the case in 
which the ${\rm N}$-polarization extends canonically to a polarization by the lattice $ {\rm H} \oplus {\rm E}_8 \oplus {\rm E}_7 \oplus {\rm A}_1$.
\par The overlap of the two discriminant components from above consists of the quartic surfaces  ${\rm Q}(\alpha,\beta, 0, \delta)$ with
\begin{equation}
\label{discrr3}
\alpha^6  + \beta^4 + \delta^2 -  2 \alpha^3 \beta^2  -  2 \alpha^3 \delta  -  2 \beta^2 \delta     \  = \  0.
\end{equation}
The K3 surfaces ${\rm X}(\alpha,\beta, \gamma, \delta)$ associated with the above condition are precisely those for which
the canonical ${\rm N}$-polarization extends to a polarization by $ {\rm H} \oplus {\rm E}_8 \oplus {\rm E}_8 \oplus {\rm A}_1$.

\subsection{Special Features on ${\rm X}(\alpha,\beta, \gamma, \delta)$}
\label{specialfeat}
\noindent Let us outline a few special properties of the four-parameter K3 family introduced above. These properties will play an important 
role in subsequent considerations.  
\par Note that the isomorphism class 
of ${\rm X}(\alpha,\beta, \gamma, \delta)$ does not change under a certain weighted scaling of the parameters $(\alpha, \beta, \gamma, \delta)$.   
\begin{prop}
Let $(\alpha, \beta, \gamma, \delta) \in \mathbb{C}^4 $ with  $\gamma \neq 0$ or $\delta \neq 0$. For any $t \in \mathbb{C}^*$,
the two ${\rm N}$-polarized K3 surfaces
$$ {\rm X}(\alpha,\beta, \gamma, \delta) \ \ {\rm and} \ \ {\rm X}(t^2 \alpha, \ t^3 \beta, \ t^5 \gamma, \ t^6 \delta) $$
are isomorphic.
\end{prop}
\begin{proof}
Let $q$ be a square root of $t$. The proposition then follows from the fact that the projective automorphism:
$$ \Phi \colon \mathbb{P}^3 \ \longrightarrow \ \mathbb{P}^3, \ \ \ [x,y,z,w] \ \mapsto \ [\ q^8x, \ q^9y, z, \ q^6w \ ] $$
maps the quartic ${\rm Q}(\alpha,\beta, \gamma, \delta)$ to ${\rm Q}(t^2 \alpha, \ t^3 \beta, \ t^5 \gamma, \
t^6 \delta)$ while satisfying $\Phi({\rm P}_1)={\rm P}_1$,  $\Phi({\rm P}_2)={\rm P}_2$.
\end{proof}
\noindent The K3 family ${\rm X}(\alpha,\beta, \gamma, \delta)$ can therefore be regarded as being parametrized, up to an isomorphism, 
by the three-dimensional open analytic space:
\begin{equation}
\label{weightedspace}
\mathcal{P}_{{\rm N}} \ = \ \left \{ \
[ \ \alpha,\beta,\gamma, \delta \ ] \in  {\mathbb W}{\mathbb P}^3(2,3,5,6) \ \middle \vert \ \gamma \neq 0 \  {\rm or} \ \delta \neq 0
\ \right \}. 
\end{equation} 
One of the main results of this paper (to be justified by the subsequent sections) is that the space 
$\mathcal{P}_{{\rm N}}$ is a coarse moduli space for ${\rm N}$-polarized 
K3 surfaces. 
\par We also note that K3 surfaces $ {\rm X}(\alpha,\beta, \gamma, \delta) $ carry two special elliptic fibrations:
$$ \varphi^{{\rm s}}_{{\rm X}},  \varphi^{{\rm a}}_{{\rm X}} \colon  {\rm X}(\alpha,\beta, \gamma, \delta) \rightarrow \mathbb{P}^1,  $$
which we shall refer to as {\it standard} and {\it alternate}\footnote{The broader context of the elliptic fibrations $\varphi^{{\rm s}}_{{\rm X}}$, 
$\varphi^{{\rm a}}_{{\rm X}}$ is discussed in Section $\ref{stanalt}$.}. They are induced, respectively, from the pencils on 
$ {\rm Q}(\alpha,\beta, \gamma, \delta) $ associated with the rational projections  
$$ {\rm pr}_1, {\rm pr}_2 \colon \mathbb{P}^3 \dashrightarrow \mathbb{P}^1, \ \ \ \  
{\rm pr}_1 ([x,y,z,w]) = [z,w] , \ \ \ \  {\rm pr}_2 ([x,y,z,w]) = [x,w] . $$
\begin{equation}
\label{twofibrations2}
\def\objectstyle{\scriptstyle}
\def\labelstyle{\scriptstyle}
\xymatrix @-0.9pc
{
 & & & & \mathbb{P}^1  \\
  {\rm X}(\alpha, \beta, \gamma, \delta) \ar @{->} [rr] 
  \ar @/^0.5pc/ @{->} ^{\varphi_{{\rm X}}^{{\rm s}}} [urrrr] 
  \ar @/_0.5pc/ @{->} _{\varphi_{{\rm X}}^{{\rm a}}} [drrrr]
  & & \mathbb{P}^3 \ar @{-->} [drr] ^{{\rm pr}_2} \ar @{-->} [urr] _{{\rm pr}_1} &  \\
 & & & & \mathbb{P}^1 \\
}
\end{equation}
The above elliptic fibrations are completely explicit and one can easily write explicit Weierstrass forms for them. For instance:
$$ v^2 \ = \ u^2 + f_{{\rm s}}(\lambda) u + g_{{\rm s}}(\lambda) $$
with  
$$ f_{{\rm s}}(\lambda) \ = \ \lambda ^4 (\gamma \lambda + 3 \alpha ) , \ \ \ g_{{\rm s}}(\lambda) \ = \ - \lambda^5 (\delta \lambda ^2 - 2 \beta  \lambda +1)$$
describes the standard elliptic fibration $\varphi_{{\rm X}}^{{\rm s}}$ over the affine chart $ \{ \ [\lambda, 1] \ \vert \ \lambda \in \mathbb{C} \ \}  $ 
of its base. A simple computation determines the discriminant of this elliptic curve family as:
$$ 4f_{{\rm s}}^3(\lambda)+27g_{{\rm s}}^2(\lambda) \ = \ $$
$$ = \ \lambda^{10} \left ( \ 4 \gamma^3 \lambda^5 + 3 (16 \alpha \gamma^2+9 \delta^2) \lambda^4 + 
12 (16 \alpha^2 \gamma-9 \beta \delta) \lambda^3 + 2 (128 \alpha^3+54 \beta^2+27 \delta) \lambda^2 -108 \beta \lambda + 27 \ \right ). 
$$
For the alternate fibration $\varphi_{{\rm X}}^{{\rm a}}$, one can describe the fibers over the affine chart 
$ \{ \ [\mu, 1] \ \vert \ \mu \in \mathbb{C} \ \}  $ of the base as:
$$ v^2 \ = \ u^2 + f_{{\rm a}}(\mu) u + g_{{\rm a}}(\mu) $$
with:
$$ f_{{\rm a}}(\mu) \ = \ \frac{1}{12} \left ( \ 
-64 \mu^6 + 96 \alpha \mu^4 + 32 \beta \mu^3 - 36 \alpha^2 \mu^2 -6 (4 \alpha \beta + \gamma ) \mu -4 \beta^2 + 3 \delta  
\ \right ) 
$$
$$
 g_{{\rm a}}(\mu) \ = \ \frac{1}{108} \left ( \ 4 \mu^3 - 3 \alpha \mu - \beta \ \right )
 \left ( \ 
 128 \mu ^6 -192 \alpha \mu^4 -64 \beta \mu^3 + 72 \alpha^2 \mu^2 + 6 (8 \alpha \beta +3 \gamma ) \mu + 8 \beta ^2-9 \delta
 \ \right )   \ .
 $$
 The discriminant of this family is:
$$ 4f_{{\rm a}}^3(\mu)+27g_{{\rm a}}^2(\mu) \ = \ $$
$$
-\frac{1}{16} \left ( 2 \gamma \mu - \delta \right )^2 \left ( \ 
16 \mu^6 - 24 \alpha \mu^4 - 8 \beta \mu^3 + 9 \alpha^2 \mu^2 + 2 (3 \alpha \beta + \gamma ) \mu + \beta^2- \delta \
\ \right )
$$
An analysis based on Tate's algorithm \cite{tate}, applied in the context of the above formulas, allows one to conclude the following:  
\begin{prop}
Assume $ \gamma \neq 0 $ or $ \delta \neq 0$. The standard elliptic fibration 
$\varphi_{{\rm X}}^{{\rm s}} \colon  {\rm X}(\alpha,\beta, \gamma, \delta) \rightarrow \mathbb{P}^1 $ carries a section, 
given by the curve $a_9$ from Figures $\ref{Gneq0Config}$ or $\ref{Geq0Config}$. In addition, there are 
two special singular fibers over the base points $ [0,1] $ and $ [1,0] $. The fiber $\varphi^{{\rm s}}_{[0,1]}$ 
has Kodaira type ${\rm II}^*$ and is represented by the divisor
$$ 2 a_1 + 4 a_2 + 6 a_3 + 3 {\rm L}_2 + 5 a_4 + 4 a_5 + 3 a_6 + 2 a_7 + a_8. $$  
If $ \gamma \neq 0 $, then the fiber $\varphi^{{\rm s}}_{[1,0]}$ has type ${\rm III}^*$ and is 
represented by the divisor
$$ b_5 + 2 b_4 + 3 b_3 + 4 b_2 + 2 b_1 + 3 {\rm L}_1 + 2 a_{11} + a_{10} $$
from Figure $\ref{Gneq0Config}$. If $ \gamma = 0 $, then the fiber $\varphi^{{\rm s}}_{[1,0]}$ has type ${\rm II}^*$ and is 
represented by the divisor
$$  2 e_6 + 4 e_5 + 6 e_3 + 3 e_4 + 5 e_2 + 4 e_1 + 3 {\rm L}_1 + 2 a_{11} + a_{10}$$
from Figure $\ref{Geq0Config}$.
\end{prop}
\begin{prop}
\label{detaliialt}
Assume $ \gamma \neq 0 $ or $ \delta \neq 0$. The alternate elliptic fibration 
$\varphi_{{\rm X}}^{{\rm a}} \colon  {\rm X}(\alpha,\beta, \gamma, \delta) \rightarrow \mathbb{P}^1 $ carries two disjoint 
sections, given by the pairs of curves $ a_1$, $b_4$ or $ a_1$, $e_6$ from Figures $\ref{Gneq0Config}$ or 
$\ref{Geq0Config}$, respectively. There is a special singular fiber over the base point $ [1,0] $. If 
$ \gamma \neq 0$, then the fiber $\varphi^{{\rm a}}_{[1,0]}$ 
has Kodaira type ${\rm I}_{10}^*$ and is represented by the divisor
$$ a_2 + {\rm L}_2 + 2 \left ( \ a_3 + a_4 + a_5 + a_6 + a_7 + a_8 + a_9 + a_{10} + a_{11} + {\rm L}_1 + b_2 \ \right ) + b_3 + b_1  $$  
from Figure $\ref{Gneq0Config}$. In such a case, one also has a singular fiber over $[\delta, 2 \gamma ] $ given by the divisor:
$$ d+ b_5 $$ of Figure $\ref{Gneq0Config}$. The fiber $\varphi^{{\rm a}}_{[\delta, 2 \gamma ]}$ has type ${\rm I}_2$ if 
$$3 \alpha \gamma^2 \delta + 2 \beta \gamma^3 - \delta^3 \neq 0 $$ 
and type ${\rm III}$ if
$$3 \alpha \gamma^2 \delta + 2 \beta \gamma^3 - \delta^3 = 0 . $$
If $ \gamma = 0 $, then the fiber $\varphi^{{\rm a}}_{[1,0]}$ has type ${\rm I}^*_{12}$ and is  
represented by the divisor
$$ a_2 + {\rm L}_2 + 2 \left ( \ a_3 + a_4 + a_5 + a_6 + a_7 + a_8 + a_9 + a_{10} + a_{11} + {\rm L}_1 + e_1 + e_2 + e_3 \ \right ) + e_5 + e_4 $$
from Figure $\ref{Geq0Config}$. 
\end{prop}
\noindent Note that the standard fibration $\varphi_{{\rm X}}^{{\rm s}} $ offers an alternate way of defining the ${\rm N}$-polarization on the K3 
surface ${\rm X}(\alpha,\beta, \gamma, \delta)$. It is known (see \cite{clingher3, kondo, shapiro}) that a pseudo-ample 
${\rm N}$-polarization on a K3 surface is equivalent geometrically 
with the existence of a jacobian elliptic fibration with two distinct special fibers of Kodaira types ${\rm II}^*$ and ${\rm III}^*$ (or higher), 
respectively. 
\par However, it is the alternate elliptic fibration $\varphi_{{\rm X}}^{{\rm a}} $, that will play the major 
role in the consideration of this paper. Let us consider the case $\gamma \neq 0$.  Then, the alternate fibration 
has two disjoint sections given by the curves $a_1$ and $b_4$ and a singular fiber of type 
${\rm I}^*_{10}$ occurs over the base point $[1,0]$ of $\varphi_{{\rm X}}^{{\rm a}} $. Consider the affine 
chart $[\mu,1]$ as in Proposition $\ref{detaliialt}$. The elliptic fiber of $\varphi_{{\rm X}}^{{\rm a}} $ over 
$ [\mu,1]$ has then the cubic form:
\begin{equation}
\label{cubicfam1}
\left \{ \ 
y^2z - \left ( 4 \mu^3  - 3 \alpha \mu   - \beta \right ) zw^2 + \gamma \mu z^2w
- \frac{1}{2} ( \delta z^2w  + w^3) \ = \ 0 \ 
\right \} \ \subset \ \mathbb{P}^2(y,z,w) \ , \ 
\end{equation}
with two special points $[1,0,0]$ and $[0,1,0]$ associated with the two sections. 
The affine version of the cubic equation in $(\ref{cubicfam1})$, in the base chart $w=1$, is:
\begin{equation}
\label{affineccc}
y^2z  \ = \ 
z^2  \left ( 
\frac{1}{2} \delta  - \gamma \mu   
\right )  
+ z \left ( 
4 \mu ^3  - 3 \alpha \mu - \beta  
\right ) + 
\frac{1}{2}   \ ,
\end{equation}   
and one can easily verify that this affine cubic curve carries a special involution:
\begin{equation}
\label{invo}
(y,z) \ \mapsto \ \left ( -y, \ \frac{1}{(\delta -2\gamma \mu) z}\right ) \ . 
\end{equation}
The map $(\ref{invo})$ extends to an involution of $ (\ref{cubicfam1}) $ which exchanges the section points $[1,0,0]$ and $[0,1,0]$. 
For the smooth elliptic curves in $ (\ref{cubicfam1}) $, the point $[0,1,0]$ can be seen as a point of order two in the elliptic curve 
group with origin at $ [1,0,0]$. The involution determined by $(\ref{invo})$ amounts then to a fiber-wise translation by $[0,1,0]$. 
\par Note that, after multiplying $(\ref{affineccc})$  by $ z \left ( \frac{1}{2} \delta - \gamma \lambda \right ) ^2 $, one gets:
$$
\left [ yz \left ( \frac{1}{2} \delta - \gamma \mu \right ) \right ]^2 \ = \ 
\left [ z \left ( \frac{1}{2} \delta - \gamma \mu \right ) \right ]^3 
+ \left [ z \left ( \frac{1}{2} \delta - \gamma \mu \right ) \right ]^2 \left ( 
4 \mu ^3  - 3 \alpha \mu - \beta  
\right ) + 
\left [ z \left ( \frac{1}{2} \delta - \gamma \mu \right ) \right ]
\frac{1}{2} \left ( \frac{1}{2} \delta - \gamma \mu \right )  \ .
$$ 
With the coordinate change:
$$ y_1 = yz \left ( \frac{1}{2} \delta - \gamma \mu \right ), \ \ \ z_1 = z \left ( \frac{1}{2} \delta - \gamma \mu \right ), $$
one obtains:
\begin{equation}
\label{firstrightsideww}
 y_1^2 \ = \ z_1^3 + \mathcal{P}(\mu) \cdot z_1^2 + \mathcal{Q}(\mu) \cdot z_1, 
 \end{equation}
where 
$$  
\mathcal{P}(\mu)=4 \mu ^3  - 3 \alpha \mu - \beta  , \ \  
\mathcal{Q}(\mu) = \frac{1}{2} \left ( \frac{1}{2} \delta - \gamma \mu \right ) \ . $$ 
One can recognize in $ (\ref{firstrightsideww})$ the classical equation for a jacobian elliptic fibration with a special section of order-two 
(see, for instance, Section 4 of the work of Van Geemen and Sarti \cite{sarti1}). The involution of $(\ref{invo})$ can be described in this new coordinate 
context as:
$$ 
\left ( z_1, \ y_1 \right ) \ \mapsto \ 
\left (
\frac{\mathcal{Q}(\mu)}{z_1} , \ -\frac{\mathcal{Q}(\mu) \cdot y_1 }{z_1^2} 
\right ). $$
One obtains the following result:
\begin{prop}
\label{propcuinv}
Let $(\alpha, \beta, \gamma, \delta) \in \mathbb{C}^4 $ with  $\gamma \neq 0$ or $\delta \neq 0$. The birational projective involution:
\begin{equation}
\label{vgsbirat}
\Psi \colon \mathbb{P}^3 \ \dashrightarrow \ \mathbb{P}^3, \ \ \ \  
 \Psi \left ( [x,y,z,w] \right )  \ = \ [ \ xz(\delta w-2\gamma x), \ -yz(\delta w-2\gamma x), \ w^3, \ zw(\delta w-2 \gamma x) \ ]
 \end{equation}
restricts to a birational involution of the quartic surface ${\rm Q}(\alpha, \beta, \gamma, \delta) $. Moreover $\Psi$ lift to 
a non-trivial involution $\Phi_{{\rm X}}$ of the ${\rm N}$-polarized K3 surface $ {\rm X}(\alpha,\beta, \gamma, \delta) $. 
\begin{equation}
\label{diagru}
\def\objectstyle{\scriptstyle}
\def\labelstyle{\scriptstyle}
\xymatrix 
{
{\rm X}(\alpha,\beta, \gamma, \delta) \ar @{->} [r] ^{\Phi_{{\rm X}}} \ar @{->} [d] & {\rm X}(\alpha,\beta, \gamma, \delta) \ar @{->} [d] \\
\mathbb{P}^3 \ar @{-->} ^{\Psi} [r] & \mathbb{P}^3 \\
}
\end{equation}
The involution $\Phi_{{\rm X}}$ exchanges the two disjoint sections of the alternate fibration $ \varphi^{{\rm a}} $ and, on the smooth fibers of this fibration, 
amounts to a fiber-wise translation by a section of order two.   
\end{prop}
\noindent Using the terminology of Definition 1.1 in \cite{clingher5}, 
$\Phi_{{\rm X}} \colon {\rm X}(\alpha,\beta, \gamma, \delta) \rightarrow {\rm X}(\alpha,\beta, \gamma, \delta) $ is a Van Geemen-Sarti involution. In the context of the dual diagrams of Figures $\ref{Gneq0Config}$ and $\ref{Geq0Config}$, the involution $\Phi_{{\rm X}}$ acts as a horizontal left-right flip.

\section{Hodge Theory and Siegel Modular Forms}
\label{hodgesiegel}
\par A coarse moduli space for the isomorphism classes of ${\rm N}$-polarized K3 surfaces can be constructed by 
gluing together spaces of local deformations. We refer the reader to \cite{asterix, dolga1} for a detailed description 
of the method. The moduli space 
$\mathcal{M}_{{\rm K3}}^{{\rm N}} $ so obtained is a quasi-projective analytic space of complex 
dimension three. Hodge theory, by the period map and the appropriate version of the Global Torelli Theorem provides 
one with an effective method of analyzing the structure of this space.
\subsection{The Period Isomorphism}
\label{modk3}
Recall that, up to an overall isometry, there exists a unique primitive embedding of ${\rm N}$ into the K3 lattice
$$ {\rm L}={\rm H} \oplus {\rm H} \oplus {\rm H} \oplus {\rm E}_8 \oplus {\rm E}_8 . $$
Fix such a lattice embedding and denote by $\Tt$ the orthogonal complement of its image. The classical period domain associated
to the lattice $\Tt$ is then:
$$ \Omega \ = \
\{ \
\omega \in \mathbb{P}^1 \left ( \Tt \otimes \Cee \right ) \ \vert \ (\omega, \omega)=0, \ \ (\omega, \bar{\omega})>0
\ \}.$$
One also has the following group isomorphism:
$$ \{ \
\sigma \in \mathcal{O}({\rm L}) \ \vert \ \sigma(\gamma) = \gamma \ {\rm for \ every} \ \gamma \in {\rm N}
\ \}  \ \stackrel{\simeq}{\longrightarrow} \ \mathcal{O}(\Tt).
$$
Via the classical Hodge decomposition, one associates to each ${\rm N}$-polarized K3 surface $({\rm X},i)$ a well-defined
point in the classifying space of ${\rm N}$-polarized Hodge structures
$$ \mathcal{O}(\Tt) \backslash \Omega . $$
Moreover, by the Global Torelli Theorem \cite{dolga1} for lattice polarized K3 surfaces, one has that the period map so constructed:
\begin{equation}
\label{permap11}
{\rm per} \colon \mathcal{M}_{{\rm K3}}^{{\rm N}}  \ \longrightarrow \  \mathcal{O}(\Tt) \backslash \Omega
\end{equation}
is an isomorphism of analytic spaces.
\par Let us analyze in detail the period domain $\Omega$. Note that the rank-five lattice $ \Tt$ is naturally
isomorphic to the orthogonal direct sum ${\rm H} \oplus {\rm H} \oplus (2) $. We select an integral basis 
$\{ p_1,p_2,q_1,q_2,r \} $ for $\Tt$ with intersection matrix:
$$
\left (
\begin{array}{ccccc}
0 & 0 & 1 & 0 & 0 \\
0 & 0 & 0 & 1 & 0 \\
1 & 0 & 0 & 0 & 0 \\
0 & 1 & 0 & 0 & 0 \\
0 & 0 & 0 & 0 & -2
\end{array}
\right )  \ .
$$
Since $p_1,p_2,q_1,q_2$ are all isotropic vectors, their intersection pairing with any given period line in
$\Omega$ is non-zero. The Hodge-Riemann bilinear relations imply then that every period in $\Omega$ can be
uniquely realized in this basis as:
$$ \omega(\tau,u,z) \ = \ ( \ \tau, \ 1, \ u , \ z^2 -\tau u,\  z \ )  $$
with $\tau,u,z \in \Cee$ satisfying $ \tau_2 u_2 > z_2^2 $. The 2-indices represent the fact that the imaginary part has been taken.
\par The period domain $\Omega$ has two connected components $ \Omega_o$ and $\overline{\Omega}_o$ which get interchanged
by the complex conjugation. Moreover, the map
\begin{equation}
\label{ident}
\kappa = \left (
\begin{array}{cc}
\tau & z \\
z & u
\end{array}
\right ) \ \ \rightarrow \ \ \omega(\tau,u,z)
\end{equation}
provides an analytic identification between the classical Siegel upper-half space of degree two:
\begin{equation}
\label{upperhalfspace}
\mathbb{H}_2 \ = \ \left \{ \ \kappa =
\left (
\begin{array}{cc}
\tau & z \\
z & u
\end{array}
\right ) \ \vert \ \
\tau_2 u_2 > z_2^2 , \ \
\tau_2 > 0
\ \right \}
\end{equation}
and the connected component $\Omega_o$. The action of the discrete group $\mathcal{O}({\rm T})$ admits a nice reinterpretion under this identification.
Note that
the real orthogonal group $\mathcal{O}(\Tt, \mathbb{R})$ has four connected components and the kernel of its action on
$\Omega$ is given by $\pm {\rm id}$. Let $\mathcal{O}^+(\Tt, \mathbb{R})$ be the (index-two) subgroup of
$\mathcal{O}(\Tt, \mathbb{R})$ that fixes the connected component of $\Omega_o$. This group can also be seen as:
$$ \mathcal{O}^+(\Tt, \mathbb{R}) \ = \ \{ \pm {\rm id} \}  \cdot \mathcal{S}\mathcal{O}^+(\Tt, \mathbb{R})$$
where $\mathcal{S}\mathcal{O}^+(\Tt, \mathbb{R})$ is the subgroup of $\mathcal{O}^+(\Tt, \mathbb{R})$ corresponding
to isometries of positive spinor-norm. Finally, set $\mathcal{O}^+(\Tt)=\mathcal{O}^+(\Tt, \mathbb{R}) \cap \mathcal{O}(\Tt)$. One has then the
following isomorphism of groups:
\begin{equation}
\label{gritsenkoizo}
{\rm Sp}_4(\mathbb{Z}) / \{ \pm {\rm I}_4 \} \ \longrightarrow \ \mathcal{O}^+(\Tt) / \{ \pm {\rm id} \} \ \simeq \   \mathcal{S}\mathcal{O}^+(\Tt).
\end{equation}
The details of $(\ref{gritsenkoizo})$ are given in \cite{gritsenko} (see Lemma 1.1 therein). In
an explicit form, the map $(\ref{gritsenkoizo})$ associates to a matrix:
$$
\left(
\begin{array}{llll}
 a_1 & a_2 & b_1 & b_2 \\
 a_3 & a_4 & b_3 & b_4 \\
 c_1 & c_2 & d_1 & d_2 \\
 c_3 & c_4 & d_3 & d_4
\end{array}
\right ) \ \in {\rm Sp}_4(\mathbb{Z})
$$
the isometry in $\mathcal{O}^+(\Tt)$ whose matrix, in  the basis $\{ p_1,p_2,q_1,q_2,r \} $, is:
$$
\left(
\begin{array}{lllll}
 -b_4 c_1+a_3 d_2 & -b_4 d_1+b_3 d_2 & b_3 c_2-a_4 d_1 & a_4 c_1-a_3
c_2 & b_3 c_1-b_4 c_2-a_3 d_1+a_4 d_2 \\
 c_3 d_2-c_1 d_4 & d_2 d_3-d_1 d_4 & -c_4 d_1+c_2 d_3 & -c_2 c_3+c_1
c_4 & -c_3 d_1+c_4 d_2+c_1 d_3-c_2 d_4 \\
 -b_4 c_1+a_3 d_2 & -b_4 d_1+b_3 d_2 & b_3 c_2-a_4 d_1 & a_4 c_1-a_3 c_2 & b_3 c_1-b_4 c_2-a_3 d_1+a_4 d_2 \\
 a_3 b_2-a_1 b_4 & b_2 b_3-b_1 b_4 & -a_4 b_1+a_2 b_3 & -a_2 a_3+a_1
a_4 & -a_3 b_1+a_4 b_2+a_1 b_3-a_2 b_4 \\
 -b_4 c_3+a_3 d_4 & -b_4 d_3+b_3 d_4 & b_3 c_4-a_4 d_3 & a_4 c_3-a_3
c_4 & b_3 c_3-b_4 c_4-a_3 d_3+a_4 d_4
\end{array}
\right).
$$
Under $(\ref{gritsenkoizo})$ and in connection with the classical action of the group
$ \Gamma_2 = {\rm Sp}_4(\mathbb{Z}) $ on
$\mathbb{H}_2$ , the identification $(\ref{ident})$ becomes equivariant. The following sequence of isomorphisms holds:
$$  \Gamma_2 \backslash \mathbb{H}_2 \ \simeq \ \mathcal{O}^+(\Tt) \backslash \Omega_o \ \simeq \ \mathcal{O}(\Tt) \backslash \Omega \ .  $$
One obtains:
\begin{prop}
\label{thmperisom} The period isomorphism $(\ref{permap11})$ identifies the moduli space
$\mathcal{M}_{{\rm K3}}^{{\rm N}}$ with the standard {\it Siegel modular threefold} 
\begin{equation}
 \mathcal{F}_2 = \Gamma_2
\backslash \mathbb{H}_2 \ . 
\end{equation}
\end{prop}
\noindent As it is well-known (see, for instance, Chapter 8 of \cite{birkenhake}), complex abelian surfaces $({\rm A}, \Pi)$ endowed with principal 
polarizations are also classified by Hodge structures of weight two associated with the lattice  ${\rm T}$. Moreover, via an appropriate 
version of Global Torelli Theorem, one has that the corresponding period map establishes an analytic identification between the coarse moduli space 
$\mathcal{A}_2$ of isomorphism classes of principally polarized complex abelian surfaces and the Siegel modular threefold $\mathcal{F}_2$. 
In connection with the above considerations, one obtains then the following result: 
\begin{prop}
\label{hodge1}
There exists a Hodge theoretic correspondence:
\begin{equation}
\label{ddduality} ({\rm A}, \Pi) \ \longleftrightarrow \ ({\rm X}, i)
\end{equation}
associating bijectively to every ${\rm N}$-polarized K3 surface $({\rm X}, i)$ a unique principally polarized
abelian surface $({\rm A}, \Pi)$. The correspondence $(\ref{ddduality})$ underlies an analytic
identification
\begin{equation}
\mathcal{A}_2 \ \cong \ \mathcal{M}_{{\rm K3}}^{{\rm N}}.
\end{equation}
between the corresponding coarse moduli spaces.
\end{prop}
\noindent One can further refine the correspondence $(\ref{ddduality})$. Recall (see, for instance, Chapter 4 of \cite{gonzalez}) that 
a principal polarization $\Pi$ on a complex abelian surface ${\rm A}$ can be of two types:
\begin{itemize}
\item [(i)] $\Pi = \mathcal{O}_{{\rm A}}({\rm E}_1 + {\rm E}_2)$ where ${\rm E}_1$ and ${\rm E}_2$ are smooth complex elliptic curves.
In this case, the abelian surface ${\rm A}$ splits canonically as a cartesian product ${\rm E}_1 \times {\rm E}_2$.
\item [(ii)] $\Pi = \mathcal{O}_{{\rm A}}({\rm C})$ where ${\rm C}$ is a smooth complex genus-two curve. In this case
one can identify ${\rm A}$ canonically with the Jacobian variety ${\rm Jac}({\rm C})$, with the divisor ${\rm C}$
being given by the image of the Abel-Jacobi map.
\end{itemize}
\noindent Case $(i)$ corresponds with the situation when the abelian surface ${\rm A}$ admits an ${\rm H}$-polarization. Under $(\ref{ddduality})$, 
one obtains then that the polarization ${\rm N}$-polarization $i$ of the corresponding K3 surface ${\rm X}$ can be extended to a polarization 
by the rank-eighteen lattice ${\rm H} \oplus {\rm E}_8 \oplus {\rm E}_8$. the case $(ii)$ corresponds with the situation when the principal 
polarization given by $\Pi$ cannot be extended to an ${\rm H}$-polarization of ${\rm A}$. Therefore, via $(\ref{ddduality})$ one obtains 
${\rm N}$-polarized K3 surfaces $({\rm X},i)$ for which the polarization $i$ cannot be extended to a 
${\rm H} \oplus {\rm E}_8 \oplus {\rm E}_8$-polarization. 
\par The considerations of this section lead then to the following conclusion. The bijective correspondence $(\ref{ddduality})$ breaks into 
two parts. first, one has a bijective correspondence:
\begin{equation}
\label{ddduality1} 
\left ( \ {\rm E}_1, \ {\rm E}_2 \right )  \ \longleftrightarrow \ ({\rm X}, i)
\end{equation} 
between un-ordered pairs of complex elliptic curves and ${\rm H} \oplus {\rm E}_8 \oplus {\rm E}_8$-polarized K3 surfaces ({\rm X}, i). 
Secondly, one has a bijective correspondence:
\begin{equation}
\label{ddduality2} 
{\rm C}  \ \longleftrightarrow \ ({\rm X}, i)
\end{equation}
between smooth complex genus-two curves ${\rm C}$ and ${\rm N}$-polarized K3 surfaces  ({\rm X}, i) with the property that polarization $i$ 
does not extend to a ${\rm H} \oplus {\rm E}_8 \oplus {\rm E}_8$ polarization.   
\par The correspondence $(\ref{ddduality1})$ was the central topic of the previous work \cite{clingher3} of the authors. The present paper 
gives an explicit description for $(\ref{ddduality2})$ 
\subsection{Siegel Modular Forms in Genus Two}
An effective way to understand the geometry of $\mathcal{F}_2$ is to use the Siegel modular forms of genus two.
Let us enumerate here the main such forms that will be relevant to the present paper. For detailed references, we refer
the reader to the classical papers of Igusa \cite{igusa62, igusa67, igusa79} and Hammond \cite{hammond} as well
as the more recent monographs of Van der Geer \cite{vdgeer06} and Klingen \cite{klingen}.
\par The simplest Siegel modular forms of genus two are those derived from Eisenstein series. These are modular forms of
even weight and are defined through the classical series:
\begin{equation}
\label{eisenstein}
\mathcal{E}_{2t}(\kappa) \ = \  \sum_{({\rm C},{\rm D})} \ {\rm det}({\rm C} \kappa + {\rm D})^{-2t}, \ \ \ \ \ t >1 .
\end{equation}
The group $\Gamma_1 = {\rm SL}(2,{\mathbb Z})$ acts by simultaneous left-multiplication on the pairs $({\rm C}, {\rm D})$
of symmetric $2 \times 2$ integral matrices, and the sum in $(\ref{eisenstein})$ is taken over the orbits of
this action. The Eisenstein forms $\mathcal{E}_{2t}$ are also integral, in the sense that their Fourier
coefficients are integers.
\par A second special class of Siegel modular forms of degree two are the Siegel cusp forms, which lie in the
kernel of the Siegel operator.
The most important cusp forms are $\mathcal{C}_{10}$, $\mathcal{C}_{12}$ and $\mathcal{C}_{35}$, of weights $10$, $12$ and 
$35$, respectively. One has\footnote{
Note that in Igusa's original notation \cite{igusa62, igusa67, igusa79}, the modular forms 
$\mathcal{E}_{4}$, $\mathcal{E}_{6}$, $ \mathcal{C}_{10}$, $\mathcal{C}_{12}$, $\mathcal{C}_{35}$ appear as 
$\psi_{4}$, $\psi_{6}$, $ \chi_{10}$, $\chi_{12}$ and $\chi_{35}$.
}:  
\begin{equation}
\label{ccusp10}
\mathcal{C}_{10} \ = \ -43867 \cdot 2^{-12} \cdot 3^{-5} \cdot 5^{-2} \cdot 7^{-1} \cdot 53^{-1}  
\left ( \ \mathcal{E}_4 \mathcal{E}_6 - \mathcal{E}_{10} \ \right )   
\end{equation}
\begin{equation}
\label{ccusp12}
\mathcal{C}_{12} \ = \ 131 \cdot 593 \cdot 2^{-13} \cdot 3^{-7} \cdot 5^{-3} \cdot 7^{-2} \cdot 337^{-1} 
\left ( \ 3^2 \cdot 7^2 \mathcal{E}_4^3 + 2 \cdot 5^3 \mathcal{E}_6^2 - 691  \mathcal{E}_{12} \ \right )   
\end{equation}
while $ \mathcal{C}_{35} $ satisfies a polynomial equation 
$\mathcal{C}_{35} ^2 = {\rm P}(\mathcal{E}_{4}, \mathcal{E}_{6}, \mathcal{C}_{10}, \mathcal{C}_{12})$ 
where ${\rm P}$ is a specific polynomial with all monomials of weighted degree $70$. The exact form of 
${\rm P}(\mathcal{E}_{4}, \mathcal{E}_{6}, \mathcal{C}_{10}, \mathcal{C}_{12})$ can be found 
in \cite[page 849]{igusa67}. 
\par The structure of the ring of Siegel modular forms of genus two is given by Igusa's Theorem:
\begin{theorem} (Igusa \cite{igusa79}) \
The graded ring ${\rm A}(\Gamma_2, \mathbb{C})$ of Siegel modular forms of degree two is generated by
$\mathcal{E}_{4}$, $\mathcal{E}_{6}$, $ \mathcal{C}_{10}$, $\mathcal{C}_{12}$ and $\mathcal{C}_{35}$ and is
isomorphic to:
$$ \mathbb{C} \left [ \mathcal{E}_{4}, \mathcal{E}_{6}, \mathcal{C}_{10}, \mathcal{C}_{12}, \mathcal{C}_{35} \right ] /
\left ( \mathcal{C}_{35} ^2 = {\rm P}(\mathcal{E}_{4}, \mathcal{E}_{6}, \mathcal{C}_{10}, \mathcal{C}_{12})\right ) \ . $$
\end{theorem}
\noindent Note that, by Igusa's work \cite{igusa62}, the cusp forms $\mathcal{C}_{10}$, $\mathcal{C}_{12}$ and $\mathcal{C}_{35}$ can also be introduced in 
terms of theta constants of even characteristics as follows: 
\begin{equation}
\label{cusp10}
\mathcal{C}_{10}(\kappa) \ = \ \ - 2^{-14} \cdot \prod_{m \ {\rm even}} \ \theta_m(\kappa)^2
\end{equation}
\begin{equation}
\label{cusp12}
\mathcal{C}_{12}(\kappa) \ = \ \ 2^{-17} \cdot 3^{-1} \cdot  \sum_{(m_1m_2m_3m_4m_5m_6)} \ \left ( \ 
\theta_{m_1}(\kappa) \theta_{m_2}(\kappa) \theta_{m_3}(\kappa)  \theta_{m_4}(\kappa) \theta_{m_5}(\kappa) \theta_{m_6}(\kappa) \ 
\right )^4
\end{equation}
\begin{equation}
\label{cusp35} \mathcal{C}_{35}(\kappa) \ = \ \ - i \cdot 2^{-39} \cdot 5^{-3} \cdot \left ( \prod_{m \ {\rm even}} \
\theta_m(\kappa) \right ) \cdot \left ( \sum_{\stackrel{(m_1m_2m_3)}{{\rm asyzygous}}} \  \pm
(\theta_{m_1}(\kappa)\theta_{m_2}(\kappa) \theta_{m_3}(\kappa) ) ^{20}  \right )
\end{equation}
The products in $(\ref{cusp10})$ and
$(\ref{cusp35})$ are taken over the ten even characteristics. The sum in $(\ref{cusp12})$ is taken over the
complements of the fifteen syzygous (G\"{o}pel) quadruples of even characteristics. The sum in
$(\ref{cusp35})$ is taken over the sixty asyzygous triples of even characteristics. According to Igusa's terminology, a 
triple of even characteristics is called syzygous if the sum of the three
characteristics is even.  Otherwise, the triple is called asyzygous.  A set of even characteristics is called syzygous
(repectively asyzygous) if all triples of the set are syzygous (respectively asyzygous).
%
%
%
\par The factors:
\begin{equation}
\label{cusp5}
 \mathcal{C}_5(\kappa) \ = \  2^{-7} \cdot \prod_{m \ {\rm even}} \ \theta_m (\kappa)
 \end{equation}
\begin{equation}
\label{cusp30} \mathcal{C}_{30}(\kappa) \ = \ \ - i \cdot 2^{-32} \cdot 5^{-3} \cdot \left ( \sum_{\stackrel{(m_1m_2m_3)}{{\rm
asyzygous}}} \  \pm (\theta_{m_1}(\kappa)\theta_{m_2}(\kappa) \theta_{m_3}(\kappa) ) ^{20}  \right ).
\end{equation}
are not Siegel modular forms in the traditional sense, as they carry non-trivial characters $\Gamma_2 \rightarrow \Zee/ 2 \Zee$. 
The forms $ \mathcal{C}_{5}$, $ \mathcal{C}_{30}$ however satisfy the relations:
$$ \mathcal{C}_5 ^2 \ = \ - \mathcal{C}_{10} , \ \ \ \ \mathcal{C}_{5} \mathcal{C}_{30} \ = \ \mathcal{C}_{35}. $$
\noindent When computing with modular forms in practice, one can employ standard methods of \cite{mumford1, mumford2, igusa60} that reduce expressions 
involving the ten theta constants with even characteristics to four fundamental theta constants (as given in Section $\ref{therealkummerside}$). Using 
Igusa's formulas in Section 4 of \cite{igusa60} and Section 3 of \cite{igusa79}, one obtains explicit (and far from complicated) expressions: 
\begin{align}
\label{tesst1}
\mathcal{E}_4 \ =& \ 2^4 P_8 \\ 
\mathcal{E}_6 \ =& \ 2^6 P_{12} \notag \\
\mathcal{C}_{10} \ =& \ -2^2 Q_{20} \notag \\
\mathcal{C}_{12} \ =& \ 2^4 \cdot 3^{-1} Q_{24} \notag
\end{align}
where $P_8$, $P_{12}$, $Q_{20}$ and $Q_{24}$ are homogeneous polynomials in the fundamental theta constants $a$, $b$, $c$, $d$. 
The precise formulas for $P_8$, $P_{12}$, $Q_{20}$ and $Q_{24}$, are given in Appendix \ref{specialpol}.
\subsection{The Singular Locus of $\mathcal{F}_2$}
\noindent The Siegel modular threefold $ \mathcal{F}_2 = \Gamma_2 \backslash \mathbb{H}_2$ is non-compact and highly singular. 
The singular locus of $ \mathcal{F}_2$ consists of the images under the projection
\begin{equation}
\label{pproj1}
\mathbb{H}_2 \ \rightarrow \ \Gamma_2 \backslash \mathbb{H}_2
\end{equation}
of the points in $\mathbb{H}_2$ whose associated periods $\omega(\tau,u,z)$ are orthogonal to roots\footnote{A
root of ${\rm T}$ is an element $r \in {\rm T}$ such that $(r,r)=-2$.} of the rank-five lattice
${\rm T}$. As ${\rm T}$ is isomorphic to the orthogonal direct sum ${\rm H} \oplus {\rm H} \oplus {\rm A}_1$,
the set of roots of ${\rm T}$ forms two distinct orbits under the natural action of $\mathcal{O}({\rm T})$. The
two orbits are distinguished by the lattice type of the orthogonal complement $\{r\}^{\perp} \subset {\rm
T}$ of a particular root $r$. For roots $r$ in one orbit, the orthogonal complements
$\{r\}^{\perp} $ are isomorphic to ${\rm H} \oplus {\rm H}$. For roots $r$ belonging to the second orbit,
$\{r\}^{\perp}$ are isomorphic to ${\rm H} \oplus \left ( 2 \right )  \oplus (-2) $. These
facts can be shown either directly, or deduced from more general results such as the ones in \cite{nikulin1}.
\par The singular locus of $ \mathcal{F}_2 $ has therefore two connected components, which turn out
to be the two Humbert surfaces $\mathcal{H}_1$ and $\mathcal{H}_4$. These surfaces are the images under the
projection $(\ref{pproj1})$ of the two divisors in $\mathbb{H}_2$ associated to $z=0$ and $\tau=u$,
respectively. As analytic spaces, both these surfaces are isomorphic (see for instance Chapter IX of
\cite{vdgeer}) to the Hilbert modular surface:
\begin{equation}
\label{hilbertmods}
 \left ( \Gamma_1 \times \Gamma_1 \right ) \rtimes \mathbb{Z} / 2 \mathbb{Z} \ \  \backslash  \ \left ( \mathbb{H}_1 \times \mathbb{H}_1 \right ).
 \end{equation}
\noindent The Humbert surfaces $\mathcal{H}_1$, $\mathcal{H}_4$ are the vanishing locus of the cusp forms $\mathcal{C}_{5}$ and $\mathcal{C}_{30}$, respectively. 
The formal sum $\mathcal{H}_1 + \mathcal{H}_4$ is then the vanishing divisor of the Siegel cusp form $\mathcal{C}_{35}$.
\par We also note that, under the period isomorphism of Proposition \ref{thmperisom}, the Humbert 
surfaces $\mathcal{H}_1$ and $\mathcal{H}_4$ correspond to ${\rm N}$-polarized K3 surfaces $({\rm X},i)$ 
for which the lattice polarizations $ i $ extends to a ${\rm H} \oplus {\rm E}_8 \oplus {\rm
E}_8$-polarization or ${\rm H} \oplus {\rm E}_8 \oplus {\rm E}_7 \oplus {\rm A}_1$-polarization,
respectively.
\begin{rem}
\label{remarkgen2}
Via the periods of the polarized Jacobian varieties ${\rm Jac}({\rm C})$, one gets a natural identification between 
the open subset $ \mathcal{F}_2 \setminus \mathcal{H}_1$ and the moduli space $\mathcal{M}_2$ of isomorphism classes of 
non-singular complex genus-two curves. The Igusa-Clebsch invariants \cite{bolza, clebsch, igusa60}
\begin{equation}
\left [ \ \mathcal{A}, \ \mathcal{B}, \ \mathcal{C}, \ \mathcal{D} \  \right ] \ \in \mathbb{WP}(2,4,6,10)
\end{equation}
classify the isomorphism class of a genus-two curve, as well as realize explicit coordinates on $ \mathcal{F}_2 \setminus \mathcal{H}_1$.
The invariants can be defined \cite{igusa67}, in terms of Siegel modular forms of genus two, as:
\begin{equation}
\label{igusaclebschinsiegel}
\left [ \ \mathcal{A}, \ \mathcal{B}, \ \mathcal{C}, \ \mathcal{D} \  \right ] \ \ = \left [ \  
2^3 3  \frac{ \ \mathcal{C}_{12}\ }{ \ \mathcal{C}_{10} \ } , \ 
2^2  \mathcal{E}_{4} , \ 
2^5  \frac{ \ \mathcal{E}_{4}\mathcal{C}_{12} \ }{ \ \mathcal{C}_{10} \ } + 2^3  3^{-1}  \mathcal{E}_{6} , \ 
2^{14}  \mathcal{C}_{10}  
 \ \right ].
\end{equation}
The above expression makes sense, as for period classes 
$ [\kappa] \in  \mathcal{F}_2 \setminus \mathcal{H}_1$, one has $\mathcal{C}_{10}(\kappa) \neq 0$. 
\par In particular, the Igusa-Clebsch 
invariants realize an explicit identification between $ \mathcal{F}_2 \setminus \mathcal{H}_1$, the moduli space $\mathcal{M}_2$ of genus-two 
curves and 
the open variety:
\begin{equation}
\left \{ \ 
\left [ \ \mathcal{A}, \ \mathcal{B}, \ \mathcal{C}, \ \mathcal{D} \  \right ] \ \in \mathbb{WP}(2,4,6,10) \ \vert \ 
\mathcal{D} \neq 0 \ 
\right \} \ .
\end{equation}
\end{rem}
\par 
\subsection{The Main Theorem}
\noindent The main theorem of this paper asserts the following:
\begin{theorem}
\label{maincomp}
Let ${\rm X}(\alpha, \beta, \gamma, \delta)$ be the four-parameter family of ${\rm N}$-polarized K3 surfaces introduced in 
Section $\ref{extendedinosesection}$. For $ \gamma \neq 0$ or $ \delta \neq 0$, let $ \kappa \in \mathbb{H}_2$ 
be a period point associated with ${\rm X}(\alpha, \beta, \gamma, \delta)$. Then, one has the following identity involving 
weighted projective 
points in $\mathbb{WP}(2,3,5,6)$:
\begin{equation}
\label{mainrelation}
\left [ \ 
\alpha, \ \beta , \ \gamma, \ \delta 
\ \right ] \ = \ \left [ \ \mathcal{E}_4, \  \mathcal{E}_6, \ 2^{12}  3^{5}  \mathcal{C}_{10}, \ 2^{12} 3^{6}  \mathcal{C}_{12}  \ \right ] \ .
\end{equation}    
\end{theorem}
\noindent The computation leading to the above result is based on a special geometric two-isogeny of K3 surfaces, 
the details of which are presented in the companion paper \cite{clingher5}. An ouline of this transformation is provided 
here in Section $\ref{twoisogeny}$. The proof of Theorem $\ref{maincomp}$ is given in Section 
$\ref{explicitcomp}$.
\par In the light of Theorem $\ref{maincomp}$ and based on the classical considerations of Igusa \cite{igusa60, igusa67}, one obtains that the period map:
$$ {\rm per} \colon \mathcal{P}_{{\rm N}} \rightarrow \mathcal{F}_2  $$
is an isomorphism and $(\ref{mainrelation})$ gives an explicit description of its inverse map. In particular, one obtains:
\begin{cor}
The open analytic space:
 \begin{equation}
\label{weightedspace1}
\mathcal{P}_{{\rm N}} \ = \ \left \{ \
[ \ \alpha,\beta,\gamma, \delta \ ] \in  {\mathbb W}{\mathbb P}^3(2,3,5,6) \ \middle \vert \ \gamma \neq 0 \  {\rm or} \ \delta \neq 0
\ \right \} 
\end{equation}
forms a coarse moduli space for isomorphism classes of ${\rm N}$-polarized K3 surfaces.
\end{cor}
\noindent We note that for $\gamma=0$, case in which the K3 surface ${\rm X}(\alpha, \beta, 0, \delta)$ carries a canonical 
$ {\rm H} \oplus {\rm E}_8 \oplus {\rm E}_8$-polarization, an identity equivalent with $(\ref{mainrelation})$ has been established by 
the authors in \cite{clingher3}. In this work we shall therefore focus on the $\gamma \neq 0$ case. 
\par For $\gamma \neq 0$, Theorem $\ref{maincomp}$ in connection with Remark $\ref{remarkgen2}$, provides 
an explicit formula, in terms of Igusa-Clebsch invariants, for the geometric transformation underlying the Hodge theoretic correspondence     
$(\ref{ddduality2})$.
\begin{cor}
Let ${\rm X}(\alpha, \beta, \gamma, \delta)$ be a ${\rm N}$-polarized K3 surface with $ \gamma \neq 0$. Then, the genus-two curve ${\rm C}$ 
associated to ${\rm X}(\alpha, \beta, \gamma, \delta)$ by the correspondence $(\ref{ddduality2})$ has Igusa-Clebsch invariants given by:
$$ \left [ \ \mathcal{A}, \ \mathcal{B}, \ \mathcal{C} , \ \mathcal{D} \ \right ] \ = \ 
\left [ \ 2^3 3 \delta, \ 2^2  3^2 \alpha \gamma^2, \ 2^3  3^2 (4\alpha \delta + \beta \gamma) \gamma^2, \ 2^2 \gamma^6    \ \right ]. $$ 
\end{cor}
\noindent The formula given by the above corrolary can be seen to agree with the computation done by Kumar \cite{kumar}. 
\begin{rem}
\label{remarkondiscr}
As a special remark, note that, under the formulas in $(\ref{mainrelation})$, one obtains the expected period interpretation 
for the discriminats $ (\ref{discrr1}) $ and $ (\ref{hum4}) $ of the quartic family ${\rm X}(\alpha, \beta, \gamma, \delta)$. Up to scaling 
by a constant, one has: 
$$  \ \mathcal{D}_1(\alpha, \beta, \gamma, \delta) \cdot \mathcal{D}_4(\alpha, \beta, \gamma, \delta) \  
\ = \ 
{\rm P}(\mathcal{E}_{4}, \mathcal{E}_{6}, \mathcal{C}_{10}, \mathcal{C}_{12}) 
\ = \ 
\mathcal{C}^2_{35} \ .$$
\end{rem}

\section{A Geometric Two-Isogeny of K3 Surfaces}
\label{twoisogeny}
This section outlines a purely geometric transformation upon which the main computation of this paper is based. For 
details regarding the transformation, as well as proofs, we refer the reader to the companion paper \cite{clingher5}. 
Various parts of the construction have also been discussed by Dolgachev (the Appendix Section of \cite{galluzzi}) 
and Kumar \cite{kumar}. 
\subsection{Elliptic Fibrations on ${\rm N}$-polarized K3 Surfaces}  
\label{stanalt}
\noindent Let $({\rm X},i)$ be a ${\rm N}$-polarized K3 surface. Assume also that the lattice polarization 
$i$ cannot be extended to a polarization by the rank-eighteen lattice $ {\rm H} \oplus {\rm E}_8 \oplus {\rm E}_8$. 
We are therefore in the case associated, under the Hodge theoretic correspondence (\ref{basiccor}), to principally polarized 
abelian surfaces obtained as Jacobians of genus-two curves.
\par By standard results on jacobian elliptic fibrations 
on K3 surfaces \cite{clingher3, kondo, shapiro}, the lattice polarization $i$ determines a canonical elliptic fibration 
$$ \varphi^{{\rm s}}_{{\rm X}} \colon {\rm X} \rightarrow \mathbb{P}^1 $$ 
with a section ${\rm S}^{{\rm s}}$ and two singular fibers of Kodaira types ${\rm II}^*$ and ${\rm III}^*$, respectively. We shall refer to 
$ \varphi^{{\rm s}}_{{\rm X}} $ as the {\it standard} elliptic fibration of ${\rm X}$. In the context of $\varphi^{{\rm s}}_{{\rm X}}$, 
one has the following dual configuration of rational curves on the K3 surface ${\rm X}$.
\begin{equation}
\label{diagg55}
\def\objectstyle{\scriptstyle}
\def\labelstyle{\scriptstyle}
\xymatrix @-0.9pc  {
\stackrel{a_1}{\bullet} \ar @{-} [r] 
& \stackrel{a_2}{\bullet} \ar @{-} [r]&
\stackrel{a_3}{\bullet} \ar @{-} [r] \ar @{-} [d] &
\stackrel{a_5}{\bullet} \ar @{-} [r]  &
\stackrel{a_6}{\bullet} \ar @{-} [r] &
\stackrel{a_7}{\bullet} \ar @{-} [r] &
\stackrel{a_8}{\bullet} \ar @{-} [r] &
\stackrel{a_9}{\bullet} \ar @{-} [r] &
\stackrel{S^{{\rm s}}}{\bullet} \ar @{-} [r] &
\stackrel{b_8}{\bullet} \ar @{-} [r] &
\stackrel{b_7}{\bullet} \ar @{-} [r] &
\stackrel{b_6}{\bullet} \ar @{-} [r] &
\stackrel{b_4}{\bullet} \ar @{-} [r] \ar @{-} [d] &
\stackrel{b_3}{\bullet} & \stackrel{b_2}{\bullet} \ar @{-} [l]
 & \stackrel{b_1}{\bullet} \ar @{-} [l] \\
 &   & \stackrel{a_4}{\bullet} & & & & & & &  &    & & \stackrel{b_5}{\bullet}  & &   \\
} 
\end{equation}
The fiber ${\rm F}^{{\rm s}}$ of the elliptic fibration $\varphi^{{\rm s}}_{{\rm X}}$ is represented by the line bundle: 
$$ 
\mathcal{O}_{{\rm X}} \left ( 2a_1+4a_2+6a_3+3a_4+5a_5+4a_6+3a_7+2a_8+a_9  \right )  = 
\mathcal{O}_{{\rm X}} \left ( b_1+2b_2+3b_3+4b_4+2b_5+3b_6+2b_7+b_8 \right ). $$
The ${\rm N}$-polarization of ${\rm X}$ appears in this context as:
$$ \langle  {\rm F}^{{\rm s}},{\rm S}^{{\rm s}} \rangle \oplus \langle a_1, a_2, \cdots a_8 \rangle \oplus \langle b_1, b_2, \cdots b_7 \rangle. $$
A second {\it alternate} elliptic fibration 
$ \varphi^{{\rm a}}_{{\rm X}} \colon {\rm X} \rightarrow \mathbb{P}^1 $ is obtained via the classical theory of Kodaira \cite{kodaira}. This elliptic pencil is associated with the line bundle:
$$ 
\mathcal{O}_{{\rm X}} \left ( \   
a_2 + a_4 + 2 (a_3+a_5+a_6+a_7 + a_8+{\rm S}^{{\rm s}}+b_8+b_7+b_6+ b_4) + b_3 + b_5 \ 
\right ). 
$$
The alternate elliptic fibration $\varphi^{{\rm a}}_{{\rm X}}$ has two disjoint sections 
$$ {\rm S}^{{\rm a}}_1= a_1, \ \ {\rm S}^{{\rm a}}_2 = b_2. $$ 
The assumption that the polarization $i$ does not extend to a lattice polarization by 
$ {\rm H} \oplus {\rm E}_8 \oplus {\rm E}_8$ implies the existence of an additional rational 
curve $c$, such that $ b_1 + c$ forms a singular fiber of type ${\rm I}_2$ (or ${\rm III}$) in the elliptic fibration 
$\varphi^{{\rm a}}_{{\rm X}}$. The diagram $(\ref{diagg55})$ completes to the following nineteen-curve diagram on ${\rm X}$.    
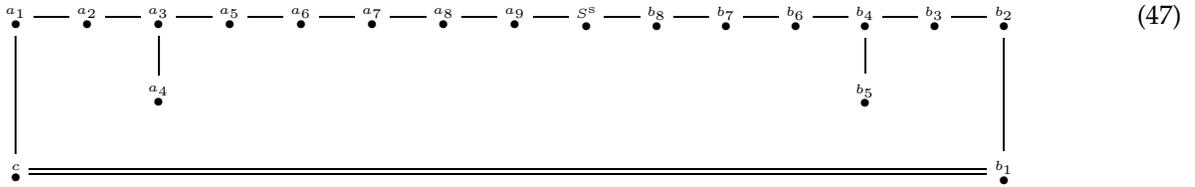
\begin{equation}
\label{diagg88}
\def\objectstyle{\scriptstyle}
\def\labelstyle{\scriptstyle}
\xymatrix @-0.9pc  {
& \stackrel{a_1}{\bullet} \ar @{-} [r] \ar @{-}[dd]  
& \stackrel{a_2}{\bullet} \ar @{-} [r]&
\stackrel{a_3}{\bullet} \ar @{-} [r] \ar @{-} [d] &
\stackrel{a_5}{\bullet} \ar @{-} [r]  &
\stackrel{a_6}{\bullet} \ar @{-} [r] &
\stackrel{a_7}{\bullet} \ar @{-} [r] &
\stackrel{a_8}{\bullet} \ar @{-} [r] &
\stackrel{a_9}{\bullet} \ar @{-} [r] &
\stackrel{S^{{\rm s}}}{\bullet} \ar @{-} [r] &
\stackrel{b_8}{\bullet} \ar @{-} [r] &
\stackrel{b_7}{\bullet} \ar @{-} [r] &
\stackrel{b_6}{\bullet} \ar @{-} [r] &
\stackrel{b_4}{\bullet} \ar @{-} [r] \ar @{-} [d] &
\stackrel{b_3}{\bullet} \ar @{-} [r] & \stackrel{b_2}{\bullet} \ar @{-}[dd]
 & \\
 & & &  \stackrel{a_4}{\bullet} &  & & & & &  &       &    & & \stackrel{b_5}{\bullet}  & & \\
  & \stackrel{c}{\bullet} \ar @{=}[rrrrrrrrrrrrrr]& &  & &  & & & & &     & & & & &   \stackrel{b_1}{\bullet}  & \\
} 
\end{equation}
\subsection{The Nikulin Construction}
As proved in \cite{clingher5}, the section $b_2$ has order two, as a member of the Mordell-Weil group ${\rm MW}(\varphi^{{\rm a}}_{{\rm X}}, a_1)$. 
Translations by $b_2$ in the smooth fibers of the elliptic fibration $\varphi^{{\rm a}}_{{\rm X}}$ extend then to a canonical 
Van Geemen-Sarti\footnote{For details regarding this concept, we refer the reader to Definition 1.1 of \cite{clingher5}.} 
involution 
\begin{equation}
\label{vgsonk3}
 \Phi_{{\rm X}} \colon {\rm X} \rightarrow {\rm X} \ .
 \end{equation} 
 The involution  $\Phi_{{\rm X}} $ acts on the curves of 
diagram $(\ref{diagg88})$ as a horizontal left-right flip. In particular, $\Phi_{{\rm X}}$ establishes a Shioda-Inose 
structure \cite{shiodainose, morrison1}, as it exchanges the two ${\rm E}_8$-configurations:
$$ \langle a_1, a_2,a_3, a_4.a_5, a_6, a_7, a_8 \rangle, \ \ \ \langle b_2,b_3, b_4, b_5, b_6, b_7, {\rm S}^{{\rm s}} \rangle \ .$$    
At this point one performs the Nikulin construction. Take the quotient of ${\rm X}$ by the involution $\Phi_{{\rm X}}$ which produces 
a singular surface with eight rational double points of type ${\rm A}_1$. Then take the minimal resolution of this quotient, hence obtaining 
a new K3 surface ${\rm Y}$. The construction exhibits a rational two-to-one map 
\begin{equation}
\label{projdouble}
{\rm p}_{\Phi_{{\rm X}}} \colon {\rm X} \dashrightarrow {\rm Y}. 
\end{equation} 
Moreover, as explained in \cite{clingher5}, the surface ${\rm Y}$ 
inherits an elliptic fibration 
\begin{equation}
\label{inducedfromalt} 
\varphi_{{\rm Y}} \colon {\rm Y} \rightarrow \mathbb{P}^1 
\end{equation} 
which is induced from the alternate fibration on ${\rm X}$. The elliptic fibration $\varphi_{{\rm Y}}$ carries a singular fiber of 
Kodaira type ${\rm I}_5^*$ and two 
disjoint sections $\widetilde{{\rm S}}_1$, $\widetilde{{\rm S}}_2$. As before, the section $\widetilde{{\rm S}}_2$ determines an element of order two in the Mordell-Weil group 
${\rm MW}(\varphi_{{\rm Y}},  \widetilde{{\rm S}}_1) $ 
and fiber-wise translations by $\widetilde{{\rm S}}_2$ extend to a dual Van Geemen-Sarti involution 
\begin{equation}
\label{dualvgs}
 \Phi_{{\rm Y}} \colon {\rm Y} \rightarrow {\rm Y} .
\end{equation} 
The Nikulin construction associated to $\Phi_{{\rm Y}}$ recovers back the 
K3 surface ${\rm X}$ as well as its alternate fibration. Hence, 
surfaces ${\rm X}$ and ${\rm Y}$ are naturally related by a geometric two-isogeny of K3 surfaces.           
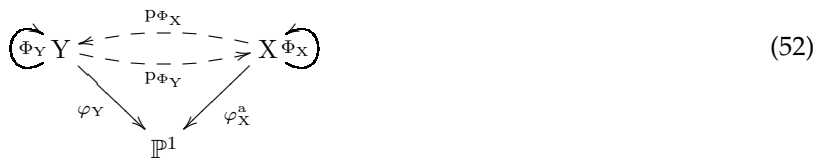
\begin{equation}
\label{chartt1}
\xymatrix 
{
{\rm Y} \ar @(dl,ul) _{\Phi_{{\rm Y}}} \ar [dr] _{\varphi_{{\rm Y}}} \ar @/_0.5pc/ @{-->} _{{\rm p}_{\Phi_{{\rm Y}}}} [rr]
&
& {\rm X} \ar @(dr,ur) ^{\Phi_{{\rm X}}} \ar [dl] ^{\varphi^{{\rm a}}_{{\rm X}}} \ar @/_0.5pc/ @{-->} _{{\rm p}_{\Phi_{{\rm X}}}} [ll] \\
& \mathbb{P}^1 & \\
}
\end{equation}
A key observation at this point is that the K3 surface ${\rm Y}$ carries a canonical Kummer structure. Let us summarize this fact. 
The Nikulin construction associated to the involution $ \Phi_{{\rm X}} $ induces a natural push-forward morphism at the cohomology 
level:
\begin{equation}
\label{morphhh}
( {\rm p}_{\Phi_{{\rm X}}} )_* \colon {\rm H}^2( {\rm X}, \mathbb{Z}) \rightarrow {\rm H}^2( {\rm Y}, \mathbb{Z}).  
\end{equation}
Denote by ${\rm U}_i$, with $1 \leq i \leq 8$, the exceptional rational curves on ${\rm Y}$ obtained from resolving the singularities 
associated with the eight fixed points of involution $ \Phi_{{\rm X}} $. The curves ${\rm U}_1, {\rm U}_2, \cdots {\rm U}_8$ form 
the even-eight configuration associated with 
the rational two-to-one map $(\ref{projdouble})$. The rank-eight lattice $\mathcal{N}$ defined as the minimal primitive sublattice of 
$ {\rm H}^2( {\rm X}, \mathbb{Z}) $ containing ${\rm U}_1, {\rm U}_2, \cdots {\rm U}_8$ is a Nikulin lattice. One has:
$$ \left \langle \ ( {\rm p}_{\Phi_{{\rm X}}} )_*(x), \ y \ \right \rangle_{{\rm Y}} \ = \ 0, \ $$
for any $ x \in {\rm H}^2( {\rm X}, \mathbb{Z})  $ and $ y \in \mathcal{N}$. 
\par Set then $ \mathcal{G} $ as the rank-seventeen sublattice 
of ${\rm NS}({\rm Y})$ given by the orthogonal direct product 
$$ ( {\rm p}_{\Phi_{{\rm X}}} )_*( i({\rm N})) \ \oplus \ \mathcal{N} \ . $$ 
Denote by $i({\rm N})^{\perp}$ and $\mathcal{G}^{\perp}$ the orthogonal complements in ${\rm H}^2( {\rm X}, \mathbb{Z})$ and 
${\rm H}^2( {\rm Y}, \mathbb{Z})$, respectively. In this context, one has:   
\begin{lem} \ 
\begin{itemize}
\item [(a)] The restriction of $(\ref{morphhh})$ induces a Hodge isometry 
\begin{equation}
\label{izo1}
 ( {\rm p}_{\Phi_{{\rm X}}} )_* \colon i({\rm N})^{\perp}(2) \ \stackrel{\simeq}{\longrightarrow} \  \mathcal{G}^{\perp} \ .  
 \end{equation}
\item [(b)] Let $\mathcal{K}$ be the rank-sixteen Kummer lattice\footnote{As defined in \cite{morrison1} or \cite{nikulin3}}. One has a canonical primitive lattice embedding:
\begin{equation}
\label{kummerstructure}
\mathcal{K} \oplus (4) \ \hookrightarrow \ \mathcal{P}. 
\end{equation}
\end{itemize} 
\end{lem}  
\noindent By Nikulin's criterion \cite{nikulin3}, the lattice embedding $(\ref{kummerstructure})$ determines a canonical Kummer structure on ${\rm Y}$,  
that is ${\rm Y}$ is a Kummer surface associated to a principally 
polarized abelian surface $({\rm A}, \Pi)$ and the sixteen exceptional curves determining the Kummer structure are explicitly determined. 
Let $ \pi \colon {\rm A} \dashrightarrow {\rm Y}$ be the rational two-to-one map 
associated to this Kummer structure. By restricting the map $\pi_*$ to the orthogonal complement of the principal 
polarization $\Pi$ in 
${\rm H}^2({\rm A}, \mathbb{Z})$, one obtains a classical Hodge isometry:
\begin{equation}
\label{izo2}
 \pi_* \colon \langle \Pi \rangle ^{\perp}(2) \ \stackrel{\simeq}{\longrightarrow} \  \mathcal{P}^{\perp} \ .  
\end{equation}
Connecting $(\ref{izo1})$ and $(\ref{izo2})$, one obtains an isometry of Hodge structures:
\begin{equation}
\label{izo3}
 (\pi_*)^{-1} \circ ( {\rm p}_{\Phi_{{\rm X}}} )_* \colon i({\rm N})^{\perp} \ \stackrel{\simeq}{\longrightarrow} \ \langle \Pi \rangle ^{\perp}  \ .
\end{equation}  
Both lattices $ \langle \Pi \rangle ^{\perp} $ and $i({\rm N})^{\perp}$ are isometric to ${\rm H} \oplus {\rm H} \oplus (-2) $. Hence, via 
the considerations of Section \ref{hodgesiegel}, one obtains that $({\rm A}, \Pi)$ is the abelian surface associated to $({\rm X},i)$ 
by the Hodge-theoretic correspondence $(\ref{basiccor})$. In particular $({\rm A}, \Pi)$ is isomorphic, as principally polarized 
abelian surface, to  
$$ \left ( \ {\rm Jac}({\rm C}), \ \mathcal{O}_{{\rm Jac}({\rm C})}(\Theta) \   \right ) $$ 
with ${\rm C}$ a well-defined complex non-singular genus-two curve.
\subsection{Elliptic Fibrations in the Context of the Kummer Structure}
\label{kummerside}
\noindent As it turns out, the elliptic fibration $ \varphi_{{\rm Y}}$, as well as the Van Geemen-Sarti involution $ \Phi_{{\rm Y}}$ can be 
explicitly described from classical features of the Kummer surface $ {\rm Y}={\rm Km}({\rm Jac}({\rm C}))$. In order to present this description, 
we shall first need to establish some notations.  
\subsubsection{Classical Facts on $ {\rm Km}({\rm C})$}
\noindent Let ${\rm C}$ be a complex non-singular genus-two curve. Assume a choice of labeling, $ a_1, a_2, \cdots a_6$, for the six ramification points of 
the canonical hyperelliptic structure. The Jacobian surface ${\rm Jac}({\rm C})$ parametrizes the degree-zero line bundles on
${\rm C}$. It comes equipped with a natural abelian group structure and contains sixteen
two-torsion points that form a subgroup
$${\rm Jac}({\rm C})_2 \ \simeq \ \left ( \Zee / 2 \Zee \right )^4. $$
The two-torsion points can be described as follows. Denote by ${\rm p}_{\varnothing}$ the neutral element of ${\rm Jac}({\rm C})$, i.e. the point
associated to the trivial line bundle of ${\rm C}$. The fifteen points of order two are then given by ${\rm p}_{ij}$ representing
the line bundles
$$ \mathcal{O}_{{\rm C}} \left ( a_i+a_j-2a_1 \right ),  \ \ \ 1 \leq i<j \leq 6.$$
The abelian group law on ${\rm Jac}({\rm C})_2$ can be seen as
$$ {\rm p}_{{\rm U}} + {\rm p}_{{\rm V}}  =  {\rm p}_{{\rm W}} $$
where ${\rm U}$, ${\rm V}$ and ${\rm W}$ are subsets of $\{ 1, 2, \cdots , 6 \}$, containing either zero or two elements, and:
\begin{equation}
\label{level2}
 {\rm W} \ = \
\begin{cases}
{\rm U} & {\rm if} \ \ {\rm V} = \varnothing \\
{\rm V} & {\rm if} \ \ {\rm U} = \varnothing \\
\varnothing & {\rm if} \ \ {\rm U} = {\rm V}  \\
\left ( {\rm U} \cup {\rm V} \right ) \setminus \left ( {\rm U} \cap {\rm V} \right )   & {\rm if} \ \ \vert {\rm U} \cap {\rm V} \vert = 1  \\
\{ 1, 2, \cdots , 6 \} \setminus \left ( {\rm U} \cup {\rm V} \right )  & {\rm if} \ \  {\rm U} \neq \varnothing, \ {\rm V} \neq \varnothing
\ {\rm and} \ {\rm U} \cap {\rm V}  = \varnothing
\end{cases}.
\end{equation}
The choice of labeling of the ramification points of ${\rm C}$ defines a level-two structure on ${\rm Jac}({\rm C})$.
\par Consider the Abel-Jacobi embedding associated to the Weierstrass point $a_0$, i.e:
$$ {\rm C} \hookrightarrow {\rm Jac}({\rm C}),  \ \ \
x \rightsquigarrow \mathcal{O}_{{\rm C}} \left ( x-a_1 \right ) $$ and denote by $\Theta_{\varnothing}$ the image of ${\rm C}$ under this map. This
is an irreducible curve on ${\rm Jac}({\rm C})$, canonically isomorphic to ${\rm C}$ and containing the six two-torsion points:
${\rm p}_{\varnothing}$, ${\rm p}_{12}$,   ${\rm p}_{13}$, ${\rm p}_{14}$, ${\rm p}_{15}$, ${\rm p}_{16}$. Let then $\Theta_{ij}$ be the image
of $\Theta_{\varnothing}$ under the translation by the order-two point ${\rm p}_{ij}$. Each of the resulting sixteen Theta divisors
$\Theta_{\varnothing}$, $\Theta_{ij}$ contains exactly six of the sixteen two-torsion points. For instance $ \Theta_{1j}$, for $2 \leq j \leq 6$, contains
$$ {\rm p}_{1j}, \ {\rm p}_{2j}, \  \cdots \  {\rm p}_{j-1 j} , \ {\rm p}_{\varnothing} , \ {\rm p}_{jj+1},  \ \cdots \   {\rm p}_{j6}, $$
while $\Theta_{ij}$, for $2 \leq i < j \leq 6$, contains:
$$ {\rm p}_{1i}, \ {\rm p}_{1j}, \ {\rm p}_{ij}, \ {\rm p}_{kl}, \ {\rm p}_{km}, \ {\rm p}_{lm} $$
where $\{ k,l,m \} = \{1,2, \cdots 6 \} \setminus \{ 0,i,j\}$. Each two-torsion point lies on precisely six of the sixteen Theta divisors.
\par The sixteen two-torsion points together with the sixteen Theta divisors on ${\rm Jac}({\rm C})$ yield, via the Kummer construction, 
a classical configuration of thirty-two smooth rational curves on ${\rm Km}({\rm C})$ - the $(16;6)$ configuration. Sixteen of the curves, 
denoted ${\rm E}_{\varnothing}$, ${\rm E}_{ij}$ are 
the exceptional curves associated to the two-torsion points ${\rm p}_{\varnothing}$, ${\rm p}_{ij}$ of ${\rm Jac}({\rm C})$, respectively.
The remaining sixteen rational curves, denoted $\Delta_{\varnothing}$, $\Delta_{ij}$ are the proper transforms of the
images of the Theta divisors $\Theta_{\varnothing}$, $\Theta_{ij}$, respectively. Following the classical terminology, we shall
refer to these latter sixteen curves as {\it tropes}.
\par On the Jacobian surface ${\rm Jac}({\rm C})$, one has $h^0({\rm Jac}({\rm C}), 2 \Theta_{\varnothing} ) = 4 $ and the linear system 
$\vert 2 \Theta_{\varnothing} \vert $ is base point free. The associated morphism:
$$ \varphi_{\vert 2 \Theta_{\varnothing} \vert} \colon {\rm Jac}({\rm C}) \rightarrow \mathbb{P}^3 $$
is generically two-to-one and its image
$${\rm S}({\rm C}) = \varphi_{\vert 2 \Theta_{\varnothing} \vert} \left ( {\rm Jac}({\rm C}) \right ) \ \subset \ \mathbb{P}^3 $$
is a quartic surface. One has a canonical identification
$$ {\rm S}({\rm C}) \ = \ {\rm Jac}({\rm C}) / \{ \pm {\rm id} \} $$
and the images of the sixteen two-torsion points of ${\rm Jac}({\rm C})$
are singularities on ${\rm S}({\rm C})$: rational double points of type ${\rm A}_1$. By convention, we shall also label these
sixteen singularities as ${\rm p}_{\varnothing}$, ${\rm p}_{ij}$. The minimal resolution of ${\rm S}({\rm C})$
is then isomorphic to the Kummer surfaces ${\rm Km}({\rm C}) $.
\begin{equation}
\label{diagc1}
\xymatrix {
{\rm Jac}({\rm C}) \ar @{->} _{\varphi_{\vert 2 \Theta_{\varnothing} \vert}} [dr] \ar @{-->} [r] & {\rm Km}({\rm C}) \ar @{->} ^{\sigma} [d] \\
& \mathbb{P}^3 & &  \\
}
\end{equation}
In this context, the sixteen curves ${\rm G}_{\varnothing}$, ${\rm G}_{ij}$ are resulting from resolving the sixteen singular 
points of ${\rm S}({\rm C})$. The tropes $\Delta_{\varnothing}$, $\Delta_{ij}$ are conics resulting
from intersecting the quartic surface ${\rm S}({\rm C})$ with sixteen special planes of $\mathbb{P}^3$. The linear system of hyperplane sections 
associated to the morphism $\sigma \colon {\rm Km}({\rm C}) \rightarrow  \mathbb{P}^3$ of diagram $(\ref{diagc1})$
is given by
$$ \vert 2\Delta_{\varnothing} + {\rm G}_{\varnothing} +  \sum_{2 \leq t \leq 6 } \ {\rm G}_{1t}   \vert
\ = \
\vert 2\Delta_{1j} + {\rm G}_{\varnothing} + \sum_{\stackrel{1 \leq t \leq 6 }{t \neq j}} \ {\rm G}_{1t}   \vert
\ = \
\vert 2\Delta_{ij} + {\rm G}_{1i} +{\rm G}_{1j}+{\rm G}_{ij} +  {\rm G}_{kl} +{\rm G}_{km}+{\rm G}_{lm}   \vert.
$$
Let ${\rm pr} \colon \mathbb{P}^3 \dashrightarrow \mathbb{P}^2 $ be the projection from the point ${\rm p}_{\varnothing}$. The 
images through this projection of the six planes associated with the tropes $ \Delta_{\varnothing}$, $\Delta_{1j}$, $ 2 \leq j \leq 6$ 
form a configuration of six distinct lines in $\mathbb{P}^2$:
\begin{equation}
\label{branchsix}
 \mathcal{L} = \{ \ {\rm L}_1 ,  \ {\rm L}_2, \ {\rm L}_3, \ \cdots  \ {\rm L}_6 \ \}. 
 \end{equation}
The six lines are tangent to a common smooth conic and meet at fifteen distinct points $ q_{ij} = {\rm pr}(p_{ij})$, $ 1 \leq i < j \leq 6$.
After blowing up the points $q_{ij}$, one obtains a rational surface ${\rm R}$ with fifteen exceptional curves ${\rm E}_{ij}$. 
Denote by ${\rm L}'_i$ with $ 1 \leq i \leq 6$, the rational 
curves on ${\rm R}$ obtained as proper transforms of the six lines ${\rm L}_i$. Then, one has a double cover morphism 
\begin{equation}
\label{nsympdouble}
 \pi \colon {\rm Km}({\rm C}) \rightarrow {\rm R} 
 \end{equation}
with branched locus given by the six disjoint curves ${\rm L}'_i$, $ 1 \leq i \leq 6$.      
\begin{equation}
\label{diaggcc11}
\xymatrix { 
{\rm Jac}({\rm C}) \ar @{->} _{\varphi_{\vert 2 \Theta_{\varnothing} \vert}} [dr] \ar @{-->} [r] & 
{\rm Km}({\rm C}) 
\ar @{->} ^{\sigma} [d] \ar @{->} ^{\pi} [r] & {\rm R} \ar @{->} ^{\rho} [d] \\
& \mathbb{P}^3 \ar @{-->} ^{{\rm pr}} [r] & \mathbb{P}^2  \\
}
\end{equation}
The deck transformation $ \beta \colon {\rm Km}({\rm C}) \rightarrow {\rm Km}({\rm C}) $ associated with the double cover $(\ref{nsympdouble})$ 
is a non-symplectic involution with fixed locus given by the union of six curves $ \Delta_{\varnothing}$, $ \Delta_{1j} $, $ 2 \leq j \leq 6 $. 
\subsection{Two Elliptic Fibrations on ${\rm Y}$}
\label{detailsfrompaper2}
There are two elliptic fibrations on the Kummer surface $ {\rm Y}$= ${\rm Km}({\rm C})$ that play an important role in our discussion. 
The first one is the elliptic fibration $\varphi_{{\rm Y}} \colon {\rm Y} \rightarrow \mathbb{P}^1 $ of $(\ref{inducedfromalt})$. The 
geometric features of this fibration are discussed in detail in Chapter 3 of \cite{clingher5}. Let us outline here the main properties. 
The elliptic pencil $ \varphi_{{\rm Y}}$ is associated with the line bundle: 
\begin{equation}
\label{divofi5star}
\mathcal{O}_{{\rm Y}} \left ( \  \Delta_{34} + \beta (\Delta_{34}) + 2 \left ( {\rm G}_{34}+ \Delta_{13} + {\rm G}_{23} + \Delta_{12}+{\rm G}_{12} + \Delta_{\varnothing} \right ) + {\rm G}_{15}+{\rm G}_{16} \ \right ) \ .
\end{equation} 
The fibration carries therefore a singular fiber of Kodaira type ${\rm I}_5^*$ 
\begin{equation}
\label{diagg2244}
\def\objectstyle{\scriptstyle}
\def\labelstyle{\scriptstyle}
\xymatrix @-0.9pc
{
& \stackrel{\Delta_{34}}{\bullet} \ar @{-} [dr] & & & &  & & & \stackrel{{\rm G}_{15}}{\bullet} \ar @{-} [dl]
 & 
 \\
& & \stackrel{{\rm G}_{34}}{\bullet} \ar @{-} [r] \ar @{-} [dl] &
\stackrel{\Delta_{13}}{\bullet} \ar @{-} [r] &
\stackrel{{\rm G}_{23}}{\bullet} \ar @{-} [r] &
\stackrel{\Delta_{12}}{\bullet} \ar @{-} [r] &
\stackrel{{\rm G}_{12}}{\bullet} \ar @{-} [r] &
\stackrel{\Delta_{\varnothing}}{\bullet} \ar @{-} [dr] & \\
& \stackrel{\beta(\Delta_{34})}{\bullet} & & & &   & & & \stackrel{{\rm G}_{16}}{\bullet} & 
\\
}
\end{equation}
In the generic situation, there are six additional singular fibers of type ${\rm I}_2$ and one of 
type ${\rm I}_1$. The tropes $ \Delta_{15} $ and $\Delta_{16}$ are disjoint sections in $\varphi_{{\rm Y}}$, whereas 
$\Delta_{14}$ is a bi-section.  
\begin{equation}
\label{diagg22}
\def\objectstyle{\scriptstyle}
\def\labelstyle{\scriptstyle}
\xymatrix @-0.9pc
{
& \stackrel{\Delta_{34}}{\bullet} \ar @{-} [dr] & & & &  & & & \stackrel{{\rm G}_{15}}{\bullet} \ar @{-} [dl]
 & \stackrel{\Delta_{15}}{\bullet} \ar @{-} [l] \\
\stackrel{\Delta_{14}}{\bullet} \ar @{-} [rr]  & & \stackrel{{\rm G}_{34}}{\bullet} \ar @{-} [r] \ar @{-} [dl] &
\stackrel{\Delta_{13}}{\bullet} \ar @{-} [r] &
\stackrel{{\rm G}_{23}}{\bullet} \ar @{-} [r] &
\stackrel{\Delta_{12}}{\bullet} \ar @{-} [r] &
\stackrel{{\rm G}_{12}}{\bullet} \ar @{-} [r] &
\stackrel{\Delta_{\varnothing}}{\bullet} \ar @{-} [dr] & \\
& \stackrel{\beta(\Delta_{34})}{\bullet} & & & &   & & & \stackrel{{\rm G}_{16}}{\bullet} & \stackrel{\Delta_{16}}{\bullet} \ar @{-} [l]\\
}
\end{equation}
As an element of the Mordell-Weil group ${\rm MW}(\varphi_{{\rm Y}}, \Delta_{15})$, the section $\Delta_{16}$ has order two. Hence, fiber-wise 
translations by $\Delta_{16}$ extend to define the Van Geemen-Sarti involution $\Phi_{{\rm Y}} \colon {\rm Y} \rightarrow {\rm Y} $ of $(\ref{dualvgs})$.  
\par A simple computation shows that, in the context of diagram $(\ref{diaggcc11})$, the ${\rm I}_5^*$ divisor in $(\ref{divofi5star})$ is the pull-back 
under the double cover $ \pi \colon {\rm Y} \rightarrow {\rm R}$ of:   
\begin{equation}
\label{ruling1}
5 \rho ^* ({\rm h}) - 3{\rm E}_{13}- 2 \left ( {\rm E}_{14} + {\rm E}_{25}+ {\rm E}_{26} \right ) - 
\left ( {\rm E}_{24}+ {\rm E}_{35}+ {\rm E}_{36}+{\rm E}_{56} \right ) 
\end{equation}
where is the hyperplane class in $ \mathbb{P}^2$. The fibers of $\varphi_{{\rm Y}}$ are therefore coming from pencil of 
projective quintic 
curves in $\mathbb{P}^2$, with a 
triple point at $q_{13}$, three double points at $q_{14}$, $q_{25}$, $q_{26}$ and also passing through the 
four points $q_{24}$, $q_{35}$, $q_{36}$, $q_{56}$. The divisor $ (\ref{ruling1})$ determines a ruling 
\begin{equation}
\label{ruling3}
\varphi_{{\rm R}} \colon {\rm R} \rightarrow \mathbb{P}^1. 
\end{equation}
The generic fiber of this ruling is a rational curve with four distinct special points: the intersection with $ {\rm L}'_5 $, ${\rm L}'_6$ 
(sections) and ${\rm L}'_4$ (bi-section). The associated elliptic fiber of $\varphi_{{\rm Y}}$ is the double cover of this rational curve 
branched at the four special points. The elliptic fibration $\varphi_{{\rm Y}}$ factors through the ruling $(\ref{ruling3})$.
$$ \varphi_{{\rm Y}} \colon {\rm Y} \ \stackrel{\pi}{\longrightarrow} \ {\rm R} \ \stackrel{\varphi_{{\rm R}}}{\longrightarrow} \ \mathbb{P}^1 $$
\par The second elliptic fibration on we consider on the K3 surface ${\rm Y}$ is associated, in a manner similar with the above 
description, with the pencil of conic curves in $\mathbb{P}^2$ passing through $q_{13}$, $q_{14}$, $q_{25}$, $q_{26}$. The line bundle: 
\begin{equation}
\label{ruling2}
\mathcal{O}_{{\rm R}} \left ( \ 2 \rho ^* ({\rm h}) - {\rm E}_{13}-  {\rm E}_{14} - {\rm E}_{25}  - {\rm E}_{26}  \ \right )  
\end{equation}   
determines a ruling 
\begin{equation}
\label{ruling4}
\psi_{{\rm R}} \colon {\rm R} \rightarrow \mathbb{P}^1 
\end{equation}
 whose pull-back through the double cover 
$\pi \colon {\rm Y} \rightarrow {\rm R}$ gives an elliptic fibration $ \psi_{{\rm Y}} \colon {\rm Y} \rightarrow \mathbb{P}^1$. The elliptic 
fibration $ \psi_{{\rm Y}}$ carries two special singular fibers of Kodaira types ${\rm I}_3$ and $ {\rm I}^*_2$.        
\begin{equation}
\label{diagg99}
\def\objectstyle{\scriptstyle}
\def\labelstyle{\scriptstyle}
\xymatrix @-0.9pc
{
& \stackrel{\Delta_{34}}{\bullet} \ar @{-} [dr] & & & & & {\rm G}_{23} \ar @{-} [dr] &    & & & \stackrel{{\rm G}_{15}}{\bullet} \ar @{-} [dl]
 & 
  \\
\stackrel{{\rm G}_{56}}{\bullet} \ar @{-} [ur] \ar @{-} [dr] & & \stackrel{{\rm G}_{34}}{\bullet} 
\ar @{-} [dl] & & &
&
&
\stackrel{\Delta_{12}}{\bullet} \ar @{-} [r] &
\stackrel{{\rm G}_{12}}{\bullet} \ar @{-} [r] &
\stackrel{\Delta_{\varnothing}}{\bullet} \ar @{-} [dr] & \\
& \stackrel{\beta(\Delta_{34})}{\bullet} & & & &   &  {\rm G}_{24} \ar @{-} [ur] & & & & \stackrel{{\rm G}_{16}}{\bullet} & 
\\
}
\end{equation}
In the generic situation, $\psi_{{\rm Y}}$ has six additional fibers of type ${\rm I}_2$.
\par In the next section, we shall use the two elliptic fibrations $ \varphi_{{\rm Y}} $ and $\psi_{{\rm Y}}$ in the context of 
the following property.  
\begin{prop}
The product morphism $ \varphi_{{\rm Y}} \times \psi_{{\rm Y}} $ factors through the double cover map:
$$ \varphi_{{\rm Y}} \times \psi_{{\rm Y}} \colon {\rm Y} \ \stackrel{\pi}{\longrightarrow} \ {\rm R} \  
\stackrel{\varphi_{{\rm R}} \times \psi_{{\rm R}}}{\longrightarrow} \  \mathbb{P}^1 \times \mathbb{P}^1 \ .$$
Moreover  $ \varphi_{{\rm R}} \times \psi_{{\rm R}} \colon {\rm R} \rightarrow \mathbb{P}^1 \times \mathbb{P}^1 $ is 
a birational morphism.
\end{prop}
\section{An Explicit Computation: Proof of Theorem \ref{maincomp}}
\label{explicitcomp}
We shall prove the identity in Theorem \ref{maincomp} by explicitly describing the details of the geometric two-isogeny transformation outlined 
in Section \ref{twoisogeny}. We give explicit formulas for the elliptic fibration $\varphi_{{\rm Y}}$ on the 
Kummer surface ${\rm Y} = {\rm Km}({\rm C})$ from the points of view of the two contexts involved: the appearance of $\varphi_{{\rm Y}}$ from 
the four-parameter ${\rm N}$-polarized K3 family 
${\rm X}(\alpha, \beta, \gamma, \delta) $ 
(with $ \gamma \neq 0 $ ) via the Nikulin construction and 
the set-up of $\varphi_{{\rm Y}}$ in the context of the Kummer construction as described in Section $\ref{detailsfrompaper2}$. 
The first description will depend on the quadruple parameter $ (\alpha, \beta, \gamma, \delta) $, while in the latter context, 
we give a formula for $\varphi_{{\rm Y}}$ in terms of Siegel modular forms. Identity $(\ref{mainrelation})$ will follow from the 
matching of the explicit formulas on the two sides.  

\subsection{The Fibration $\varphi_{{\rm Y}}$ via the Nikulin Construction}
Recall from Section $\ref{specialfeat}$ that, in the context of the K3 surface ${\rm X}(\alpha, \beta, \gamma, \delta) $, 
the alternate fibration $ \varphi^{{\rm a}}_{{\rm X}} $ can be described by the affine equation: 
\begin{equation}
\label{firstrightsideww2}
 y_1^2 \ = \ z_1^3 + \mathcal{P}_{{\rm X}}(\mu) \cdot z_1^2 + \mathcal{Q}_{{\rm X}}(\mu) \cdot z_1, 
 \end{equation}
where 
$$  
\mathcal{P}_{{\rm X}}(\mu)=4 \mu ^3  - 3 \alpha \mu - \beta  , \ \  
\mathcal{Q}_{{\rm X}}(\mu) = \frac{1}{2} \left ( \frac{1}{2} \delta - \gamma \mu \right ) \ . $$ 
The Van Geemen-Sarti involution $\Phi_{{\rm X}}$ is described by Proposition $\ref{propcuinv}$ and, in the context of the affine coordinates 
$(z_1,y_1)$ of $(\ref{firstrightsideww2})$, acts as:
$$ 
\left ( z_1, \ y_1 \right ) \ \mapsto \ 
\left (
\frac{\mathcal{Q}_{{\rm X}}(\mu)}{z_1} , \ -\frac{\mathcal{Q}_{{\rm X}}(\mu) \cdot y_1 }{z_1^2} 
\right ). $$
Then, as explained, for instance, by Van Geemen and Sarti in Section 4 of \cite{sarti1}, one can write an affine form for the elliptic fibration $\varphi_{{\rm Y}}$ 
as follows: 
\begin{equation}
\label{secondrightside}
 y_2^2 \ = \ z_2^3  + \mathcal{P}_{{\rm Y}}(\mu) \cdot z_2^2 + \mathcal{Q}_{{\rm Y}}(\mu)  \cdot z_2, 
 \end{equation}
 where the affine coordinates $(z_2,y_2)$ are:
 $$
 z_2 \ = \ \frac{y_1^2}{z_1^2} , \ \ \ y_2 \ = \ \frac{\left ( \mathcal{Q}_{{\rm X}}(\mu) -  z_1^2  \right )y_1 }{z_1^2}
 $$
 and
\begin{equation}
 \mathcal{P}_{{\rm Y}}(\mu) \ = \ - 2 \mathcal{P}_{{\rm X}}(\mu) \ = \ - 8 \mu^3 + 6 \alpha \mu + 2 \beta   
 \end{equation}
\begin{equation}
\mathcal{Q}_{{\rm Y}}(\mu) \ = \ \mathcal{P}_{{\rm X}}^2(\mu) - 4 \mathcal{Q}_{{\rm X}}(\mu) \ = \ 
 16 \mu^6 - 24 \alpha \mu^4 - 8 \beta \mu^3 + 9 \alpha^2 \mu^2 + 2 (3 \alpha \beta +  \gamma) \mu + \beta^2 - \delta \ .
 \end{equation}
\subsection{The Fibration $\varphi_{{\rm Y}}$ via the Kummer Construction}
\label{therealkummerside}
The maps of diagram $(\ref{diaggcc11})$ can be described explicitly in terms of genus-two theta functions. Let 
$ \kappa \in \mathbb{H}_2 $ be a point of the Siegel upper half-space defined in $(\ref{upperhalfspace})$. Furthermore, assume that $\kappa$ 
is associated with a set of periods for the polarized Hodge structure of ${\rm Jac}({\rm C})$. By classical results (see \cite{mumford1, mumford2}), 
there are then sixteen theta functions 
$$ \theta_m(\kappa, \cdot) \colon \mathbb{C}^2 \rightarrow \mathbb{C}, $$ 
with characteristics $m=(u,v)$, $u,v \in \{ 0, 1/2 \} \times \{ 0, 1/2 \} $. The theta functions $\theta_m(\kappa, \cdot)$ descend 
to sections in line bundles over the Jacobian surface ${\rm Jac}({\rm C})$ determining the sixteen Theta divisors\footnote{
One can arrange that $ \theta_{m}(\kappa, \cdot) \in {\rm H}^0({\rm Jac}({\rm C}), \Theta_{\varnothing}) $ for $m=((0,0),(0,0))$ and 
for the level-two structure induced by characteristics to match $(\ref{level2})$. 
} $\Theta_{\varnothing}$, $ \Theta_{ij}$. 
\par Among the possible sixteen characteristics $m=(u,v)$, ten are even and six are odd. The ten even theta functions are related 
by six independent Riemann theta relations. Our computation will be based on the following four {\it fundamental theta functions}
\begin{equation}
\label{fundamtheta}
\theta_{m_1}(\kappa, \cdot), \ \ \theta_{m_2}(\kappa, \cdot), \ \ \theta_{m_3}(\kappa, \cdot), \ \ \theta_{m_4}(\kappa, \cdot)
\end{equation}
with:
$$ m_1= \left ( (0,0),(0,0) \right ), \ \  m_2= \left ( (0,0),(1/2,1/2) \right ) $$
$$  m_3= \left ( (0,0),(1/2,0) \right ), \ \ m_4= \left ( (0,0),(0,1/2) \right ). $$
In this context, one can describe the morphism $\varphi_{\vert 2 \Theta_{\varnothing} \vert}$ of diagram $(\ref{diaggcc11})$ as:
\begin{equation}
\xymatrix { 
\mathbb{C}^2 \ar @{->} [d] \ar @{->} ^{\Xi} [drr] & & \\
{\rm Jac}({\rm C}) \ar @{->} _{\varphi_{\vert 2 \Theta_{\varnothing} \vert}} [rr] & & \mathbb{P}^3 \\ 
}
\end{equation}
where $ \Xi \colon \mathbb{C}^2 \rightarrow \mathbb{P}^3 $ is defined as
$$\Xi(Z) \ = \ \left [ \   
\theta_{m_1}(\kappa, 2{\rm Z} ), \ \theta_{m_1}(\kappa, 2{\rm Z} ), \ \theta_{m_3}(\kappa, 2{\rm Z} ), \ \theta_{m_4}(\kappa, 2{\rm Z} )
\ \right ] . 
$$
Via Frobenius identities, one obtains then an explicit description for the quartic surface:
$$ {\rm S}_{{\rm C}}  =  \varphi_{\vert 2 \Theta_{\varnothing} \vert} \left ( {\rm Jac}({\rm C}) \right ) 
\ \subset \ \mathbb{P}^3(x,y,z,w). $$
This is the classical equation of Hudson \cite{hudson, gonzalez}:
\begin{equation}
\label{hudson}
 x^4+y^4+z^4+w^4 +2\Dd xyzw + \Aa(x^2w^2+y^2z^2)+ \Bb (y^2w^2+x^2z^2) + \Cc (x^2y^2+ z^2w^2) \ = \ 0. 
\end{equation}
The coefficients $\Aa$, $\Bb$, $\Cc$, $\Dd$ of the Hudson quartic are rational functions in the four {\it fundamental theta constants} :
$$ a = \theta_{m_1}(\kappa, 0), \ \ \   b = \theta_{m_2}(\kappa, 0), \ \ \ c = \theta_{m_3}(\kappa, 0), \ \ \ d = \theta_{m_4}(\kappa, 0), $$
and appear as follows:
\begin{equation}
\label{Hudsoncoef}
 \Aa \ = \ \frac{b^4+c^4-a^4-d^4}{a^2d^2-b^2c^2}, \ \ \ \Bb \ = \ \frac{c^4+a^4-b^4-d^4}{b^2d^2-c^2a^2},
\ \ \ \Cc \ = \ \frac{a^4+b^4-c^4-d^4}{c^2d^2-a^2b^2}, \  
\end{equation}
$$ \Dd \ = \ \frac{
abcd(d^2+a^2-b^2-c^2)(d^2+b^2-c^2-a^2)(d^2+c^2-a^2-b^2)(a^2+b^2+c^2+d^2)
}{
(a^2d^2-b^2c^2)(b^2d^2-c^2a^2)(c^2d^2-a^2b^2)
}.
$$
Note that, as function of $ \kappa \in \mathbb{H}_2$, the homogeneous polynomial 
\begin{equation}
\label{discr}
(ad-bc)(ad+bc)(ac-bd)(ac+bd)(ab-cd)(ab+cd)
(a^2+d^2- b^2-c^2)(a^2+c^2- b^2-d^2)
\end{equation}
$$
(a^2+b^2-c^2-d^2)(a^2+b^2+c^2+d^2) $$
represents $\mathcal{C}_{10}$ scaled by a non-zero constant. The zero-divisor of $\mathcal{C}_{10}$ is the Humbert surface 
$\mathcal{H}_1$, and hence the denominators in $(\ref{Hudsoncoef})$ are all non-zero.
\par In the Hudson quartic setting, the sixteen singularities $ p_{\varnothing}$, $p_{ij}$ of $ {\rm S}_{{\rm C}} $ 
are as follows: 
\begin{itemize}
\item $ 
p_{\varnothing} \ =  \ \ [a,b,c,d]$
\item $ 
p_{12} \ = \ \ [c,d,a,b]$
\item $ 
p_{13} \ =  \ \ [a,-b,-c,d]$
\item $
p_{14} \ =  \ \ [-b,a,d,-c]$
\item $ 
p_{15} \ =  \ \ [c,d,-a,-b]$
\item $ 
p_{16} \ = \ \ [-b,-a,d,c]$
\item $ 
p_{23} \ = \ \ [-c,d,a,-b]$
\item $ 
p_{24} \ = \ \ [d,-c,-b,a]$
\item $ 
p_{25} \ = \ \ [-a,-b,c,d]$
\item $ 
p_{26} \ = \ \  [d,c,-b,-a]$
\item $ 
p_{34} \ = \ \ [b,a,d,c]$
\item $ 
p_{35} \ = \ \ [-c,d,-a,b]$
\item $ 
p_{36} \ =  \ \ [b,-a,d,-c]$
\item $ 
p_{45} \ =   \ \ [d,-c,b,-a]$
\item $ 
p_{46} \ = \ \ [-a,b,-c,d]$
\item $ 
p_{56} \ = \ \ [d,c,b,a]$
\end{itemize}
The sixteen tropes $ \Delta_{\varnothing}$, $\Delta_{ij}$ correspond to the following sixteen hyperplanes:
\begin{itemize}
\item $ 
\Delta_{\varnothing} : \ \ \ dx - cy + bz- aw= 0, 
$
\item $ 
\Delta_{12} : \ \ \ bx - ay + dz- cw= 0, 
$
\item $ 
\Delta_{13}: \ \ \ dx + cy - bz+ aw= 0, 
$
\item $ 
\Delta_{14} : \ \ \ cx + dy - az- bw= 0, 
$
\item $ 
\Delta_{15} : \ \ \ -bx + ay + dz- cw= 0, 
$
\item $ 
\Delta_{16} : \ \ \ -cx + dy + az- bw= 0, 
$
\item $ 
\Delta_{23} : \ \ \ -bx - ay + dz+ cw= 0, 
$
\item $ 
\Delta_{24} : \ \ \ -ax -b y + cz +dw= 0, 
$
\item $ 
\Delta_{25} : \ \ \ dx - cy -bz+aw= 0, 
$
\item $ 
\Delta_{26} : \ \ \ ax - by -c z+ dw= 0, 
$
\item $ 
\Delta_{34} : \ \ \ -cx + dy -az+bw= 0, 
$
\item $ 
\Delta_{35} : \ \ \ bx + ay + dz+ cw= 0, 
$
\item $ 
\Delta_{36} : \ \ \ cx + dy + az + bw= 0, 
$
\item $ 
\Delta_{45} : \ \ \ ax + by + cz+ dw= 0, 
$
\item $ 
\Delta_{46}: \ \ \ dx + cy + bz + aw= 0, 
$
\item $ 
\Delta_{56} : \ \ \ -ax + by - cz+dw= 0, 
$
\end{itemize}
The rational projection ${\rm pr} \colon \mathbb{P}^3 \dashrightarrow \mathbb{P}^2 $ of diagram $(\ref{diaggcc11})$ has then the explicit form:
$$ {\rm pr} \left ( \ [ x, y, z, w ] \right )  \ = \ \left [ \ 
-bx + ay - dz + cw, \  cx - dy - az + bw, \ dx + cy - bz - aw \ 
  \right ] \ .$$
By a slight abuse of notation, we shall use homogeneous coordinates $[x,y,z]$ on the target space of the projection. In these coordinates, 
the six lines $ {\rm L}_n $, with $ 1 \leq n  \leq 6 $ forming the branch locus 
$(\ref{branchsix})$ can be described through the equations ${\rm L}_n(x,y,z) = 0$ where :
\begin{itemize}
\item $ {\rm L}_1(x,y,z) \ = \ 2 (a c + b d) x + 
   2 (a b - c d) y - (a^2 - b^2 - c^2 + d^2) z $
\item $ {\rm L}_2(x,y,z) \ = \ x $
\item $ {\rm L}_3(x,y,z) \ = \ z $
\item $ {\rm L}_4(x,y,z) \ = \ 2 (a d - bc)x + (a^2 - b^2 + c^2 - d^2)y + 2 (a b + c d)z $
\item $ {\rm L}_5(x,y,z) \ = \ (-a^2 - b^2 + c^2 + d^2)x + 2 (a d + b c)y - 2 (a c - b d)z $
\item $ {\rm L}_6(x,y,z) \ = \  y $
\end{itemize}   
The fifteen intersection points $q_{ij}$ of the six-line configuration are:  
\begin{itemize} 
\item $ 
q_{12} \ = \ 
[ \ 0, \ -a^2 + b^2 + c^2 - d^2, \ -2 a b + 2 c d \ ] $
\item $ 
q_{13} \ = \ 
[ \ -2 a b + 2 c d, \ 2 a c + 2 b d, \ 0 \ ] $
\item $ 
q_{14} \ = 
[ \ a^2 + b^2 - c^2 - d^2, \ -2 b c - 2 a d, \ 2 a c - 2 b d \ ] $
\item $ 
q_{15} \ = \ 
[ \ -2 b c + 2 a d, \ a^2 - b^2 + c^2 - d^2, \ 2 a b + 2 c d \ ] $
\item $ 
q_{16} \ = \ 
[ \ -a^2 + b^2 + c^2 - d^2, \ 0, \ -2 a c - 2 b d \ ] $
\item $ 
q_{23} \ = \ 
[ \ 0, \ -a^2 - b^2 - c^2 - d^2, \ 0 \ ] $
\item $ 
q_{24} \ = \ 
[ \ 0, \ 2 a b + 2 c d, \ -a^2 + b^2 - c^2 + d^2 \ ] $
\item $ 
q_{25} \ = \ 
[ \ 0, \ -2 a c + 2 b d, \ -2 b c - 2 a d \ ] $
\item $ 
q_{26} \ = 
[ \ 0, \ 0, \ a^2 + b^2 + c^2 + d^2 \ ] $
\item $ 
q_{34} \ = \ 
[ \ a^2 - b^2 + c^2 - d^2, \ 2 b c - 2 a d, \ 0 \ ] $
\item $ 
q_{35} \ = 
[ \ 2 b c + 2 a d, \ a^2 + b^2 - c^2 - d^2, \ 0 \ ] $
\item $ 
q_{36} \ = 
[ \ -a^2 - b^2 - c^2 - d^2, \ 0, \ 0 \ ] $
\item $ 
q_{45} \ = \ 
[ \ -2 a c - 2 b d, \ -2 a b + 2 c d, \ a^2 - b^2 - c^2 + d^2 \ ] $
\item $ 
q_{46} \ = \ 
[ \ 2 a b + 2 c d, \ 0, \ 2 b c - 2 a d \ ] $
\item $ 
q_{56} = \ 
[ \ 2 a c - 2 b d, \ 0, \ -a^2 - b^2 + c^2 + d^2 \ ] $
\end{itemize}
\subsection{The Quintic Pencil $\varphi_{{\rm R}}$}
\label{exquint}
As explained in Section $\ref{detailsfrompaper2}$, in order to describe explicitly the elliptic fibration 
$ \varphi_{{\rm Y}} \colon {\rm Y} \rightarrow \mathbb{P}^1$ of $(\ref{inducedfromalt})$, one needs to understand 
the ruling $ \varphi_{{\rm R}} \colon {\rm R} \rightarrow \mathbb{P}^1$ of $(\ref{ruling3})$. This ruling is associated 
with the pencil of quintic curves in $\mathbb{P}^2$, with a 
triple point at $q_{13}$, three double points at $q_{14}$, $q_{25}$, $q_{26}$ and passing through the 
four points $q_{24}$, $q_{35}$, $q_{36}$, $q_{56}$. 
\par This pencil can be described explicitly. Note that a first such quintic curve is given by:
\begin{equation}
\label{squinticdiv}
 {\rm L}_1 + {\rm L}_2+{\rm L}_3+{\rm C} 
\end{equation}
where ${\rm C}$ is the unique conic passing through $ q_{13}$, $q_{14}$, $q_{25}$, $q_{26}$, $q_{56}$.
The pull-back of the divisor $(\ref{squinticdiv})$ determines the ${\rm I}_5^*$ fiber of the elliptic 
fibration $ \varphi_{{\rm Y}}$, as described 
in $(\ref{diagg2244})$. The conic ${\rm C}$ is given by the following polynomial:
\begin{equation}
\label{conicc}
{\rm C}(x,y,z) \ = \ c_{200}x^2 + c_{020} y^2 + c_{002} z^2 + c_{110}xy + c_{101} xz + c_{011} yz 
\end{equation}
with coefficients set as follows:
\begin{itemize}
\item [] $ c_{200} \ = \ -2 (a d - b c) (b c + a d) (a c + b d) (a^2 + b^2 - c^2 - d^2) $
\item [] $ c_{020} \ = \ - (b c + a d) (a b - c d) (a^2 + b^2 - c^2 - d^2) (a^2 - b^2 + c^2 - d^2) $
\item [] $ c_{002} \ = \ 0 $
\item [] $ c_{110} \ = \ -(b c + a d) (a^2 + b^2 - c^2 - d^2) (a^3 c - 3 a b^2 c + a c^3 + 3 a^2 b d - b^3 d + 3 b c^2 d - 3 a c d^2 - b d^3) $
\item [] $ c_{101} \ = \ - 4 (a d - b c ) (b c + a d) (a c - b d) (a c + b d) $
\item [] $ c_{011} \ = \ (a c - b d) (a b - c d) (a^2 + b^2 - c^2 - d^2) (a^2 - b^2 + c^2 - d^2) \ . $
\end{itemize}
We have therefore a description for the divisor $(\ref{squinticdiv})$ as the zero-locus a special quintic:
\begin{equation}
\label{quintic1}
{\rm QIN}_1(x,y,z) \ = \ {\rm L}_1(x,y,z) \cdot {\rm L}_2(x,y,z) \cdot {\rm L}_3(x,y,z) \cdot {\rm C}(x,y,z). 
\end{equation}
In order to select a second quintic polynomial with the required properties, we choose to impose the extra condition that the quintic curve 
passes through $q_{45}$. In the generic situation, the pull-back of the strict transform of this quintic curve determines a singular fiber 
of Kodaira type ${\rm I}_2$ on the elliptic fibration $ \varphi_{{\rm Y}}$. A polynomial describing this curve can be given as follows:  
\begin{align*}
{\rm QIN}_2(x,y,z) \ =& \ k_{500} x^5 + k_{050} y^5 + k_{005} z^5 + k_{410} x^4 y + k_{401} x^4 z + k_{140} x y^4 + 
k_{041} y^4 z + k_{104} x z^4 + k_{014} y z^4 + \\
&
+ k_{320} x^3 y^2 + k_{302} x^3 z^2 + k_{230} x^2 y^3 + 
k_{032} y^3 z^2 + k_{203} x^2 z^3 + k_{023} y^2 z^3 + k_{311} x^3 y z + k_{131} x y^3 z + \\
&+k_{113} x y z^3 +  k_{122} x y^2 z^2 + k_{212} x^2 y z^2 + k_{221} x^2 y^2 z \ .
\end{align*}
The coefficients $k_{ijk}$ are homogeneous degree-sixteen polynomials the fundamental theta constants $a,b,c,d$. 
Their precise form is given in Appendix $\ref{coefquint2}$.
\par The full pencil of quintic curves can be then described by:
\begin{equation}
\label{fullquint}
 {\rm QIN}_{t_1,t_2} (x,y,z) \ = \ t_1  \cdot {\rm QIN}_1 (x,y,z) + t_2 \cdot {\rm QIN}_2 (x,y,z), \ \ \ (t_1,t_2) \in \mathbb{C}^2 \ . 
 \end{equation}
\subsection{The Conic Pencil $\psi_{{\rm R}}$}
\label{exconic}
As explained earlier, the quintic pencil in Section $\ref{exquint}$ determines a ruling $ \varphi_{{\rm R}} $ on the rational 
surface ${\rm R}$ obtained by blowing up the fifteen points $q_{ij}$ on $\mathbb{P}^2$. The proper transforms $ {\rm L}'_5$, ${\rm L}'_6$ are 
sections in this ruling, while  ${\rm L}'_4$ is a bi-section. On each smooth fiber of $ \varphi_{{\rm R}} $, these sections/bi-section 
determine four distinct points and the associated elliptic fiber of $ \varphi_{{\rm Y}}$ is the double cover of the rational curve branched 
at these four special points. 
\par Our strategy shall be to describe explicitly the location of the four branch points via a parametrization of the ruling. In order to 
accomplish this task, we shall use the second ruling $\psi_{{\rm R}} \colon {\rm R} \rightarrow \mathbb{P}^1 $, the ruling associated to 
the pencil of projective conics passing through 
the four points  $q_{13}$, $q_{14}$, $q_{25}$, $q_{26}$.  This pencil can be written explicitly as:
\begin{equation}
{\rm C}_{s_1,s_2}(x,y,z) \ = \ s_1 \cdot {\rm C}(x,y,z) + s_2 \cdot {\rm L}_1(x,y,z) \cdot {\rm L}_2(x,y,z) , \ \ \ (s_1,s_2) \in \mathbb{C}^2 \ \ .  
\end{equation} 
As explained in Section $\ref{detailsfrompaper2}$, the intersection between generic 
fibers of the rulings $\varphi_{{\rm R}}$ and $\psi_{{\rm R}}$,  respectively, consist of exactly one point and one obtains a birational morphism 
$ \varphi_{{\rm R}} \times \psi_{{\rm R}} \colon {\rm R} \rightarrow  \mathbb{P}^1 \times \mathbb{P}^1 $.   
\begin{equation}
\xymatrix { 
{\rm R} \ar @{->} _{\rho} [d] \ar @{->} ^{\varphi_{{\rm R}} \times \psi_{{\rm R}}} [drr] & & \\
\mathbb{P}^2 \ar @{-->}  [rr] & & \mathbb{P}^1 \times \mathbb{P}^1 \\ 
}
\end{equation}
\subsection{Explicit Description of the Elliptic Fibration $\varphi_{{\rm Y}}$}
Let $ t \in \mathbb{C}$. Consider then the quintic curve:
\begin{equation}
\label{quintt1}
{\rm QIN}_{t, 1}(x,y,z) = t  \cdot {\rm QIN}_1 (x,y,z) +   {\rm QIN}_2 (x,y,z)\ = \ 0. 
\end{equation}
From the point of view of this work, one has four important points on the curve $(\ref{quintt1})$. These points are 
given by the residual intersections with the lines ${\rm L}_5$, ${\rm L}_6$, and ${\rm L}_4$, 
respectively. The images of these four points through the rational map:
\begin{equation}
\label{rationalm}
\frac{{\rm C}(x,y,z)}{{\rm L}_1(x,y,z) \cdot {\rm L}_2(x,y,z)}
\end{equation}
can be described as follows. The image through $(\ref{rationalm})$ of the intersection with ${\rm L}_5$ is:  
\begin{equation}
\label{forma}
A(t) \ = \  \frac{A_0+A_1 t}{24 (a^2 + b^2 + c^2 + d^2)}
\end{equation}
where:
\begin{align*}
A_0 \ =& \ -48 (a d-b c) (a c+b d) (a b-c d) (a^2+b^2+c^2+d^2) \ , \\
 A_1 \ =& \ a^2-b^2+c^2-d^2 \ .
\end{align*}
The image through $(\ref{rationalm})$ of the intersection with ${\rm L}_6$ is:

\begin{equation}
\label{formb}
B(t) \ = \  \frac{B_0+B_1 t}{24 (a^2 + b^2 + c^2 + d^2)}
\end{equation}
where:
\begin{align*}
B_0 \ =& \ 6 (a^2+b^2+c^2+d^2) (a^6-3 a^4 b^2+3 a^4 c^2-3 a^4 d^2-8 a^3 b c d+3 a^2 b^4+2 a^2 b^2 c^2-2 a^2 b^2 d^2+3 a^2 c^4+ \\
&+2 a^2 c^2 d^2
+3 a^2 d^4+8 a b^3 c d-8 a b c^3 d+8 a b c d^3-b^6+3 b^4 c^2-3 b^4 d^2-3 b^2 c^4-2 b^2 c^2 d^2-3 b^2 d^4+ \\
&+c^6-3 c^4 d^2+3 c^2 d^4-d^6) \ , \\
B_1 \ =& \ a^2-b^2+c^2-d^2 \ .
\end{align*}
Finally, the two points of intersection with ${\rm L}_4$, map under $(\ref{rationalm})$, to the two roots of the quadratic equation:
\begin{equation}
\label{formcde}
C(t) \cdot u^2+D u + E \ = \ 0 \ , 
\end{equation} 
where:
\begin{align*}
C(t) \ =& \ C_0 + C_1 t  \\
   C_1 \ =& \ a^2-b^2+c^2-d^2 \ , \\
   C_0 \ =& \ 6 (a^2+b^2+c^2+d^2) (a^6+a^4 b^2-a^4 c^2+a^4 d^2-8 a^3 b c d-a^2 b^4-10 a^2 b^2 c^2+10 a^2 b^2 d^2-a^2 c^4- \\
   &-10 a^2 c^2 d^2-a^2 d^4+ 
   +8 a b^3 c d-8 a b c^3 d+8 a b c d^3-b^6-b^4 c^2+b^4 d^2+b^2 c^4+10 b^2 c^2 d^2+b^2 d^4+ \\
   &+c^6+c^4 d^2-c^2 d^4-d^6) \ , \\
   D \ =& \ -24 (a^2 + b^2 + c^2 + d^2) (b^2 c^2 a^8 - b^2 d^2 a^8 + c^2 d^2 a^8 -
    b^2 c^4 a^6 - b^2 d^4 a^6 + c^2 d^4 a^6 + b^4 c^2 a^6 - 
   b^4 d^2 a^6 - \\
   &-c^4 d^2 a^6 + 6 b^2 c^2 d^2 a^6 
   -b^2 c^6 a^4 + 
   b^2 d^6 a^4 - c^2 d^6 a^4 - 6 b^4 c^4 a^4 - 6 b^4 d^4 a^4 - 
   6 c^4 d^4 a^4 + 2 b^2 c^2 d^4 a^4 - \\
   &-b^6 c^2 a^4 + b^6 d^2 a^4 - 
   c^6 d^2 a^4 - 2 b^2 c^4 d^2 a^4 
   + 2 b^4 c^2 d^2 a^4 + b^2 c^8 a^2 +
    b^2 d^8 a^2 - c^2 d^8 a^2 + b^4 c^6 a^2 - b^4 d^6 a^2 - \\
    &-
   c^4 d^6 a^2 + 6 b^2 c^2 d^6 a^2 - b^6 c^4 a^2 - b^6 d^4 a^2 + 
   c^6 d^4 a^2 + 2 b^2 c^4 d^4 a^2 - 2 b^4 c^2 d^4 a^2 - b^8 c^2 a^2 +
    b^8 d^2 a^2 + \\
    &+c^8 d^2 a^2 + 6 b^2 c^6 d^2 a^2 + 
   2 b^4 c^4 d^2 a^2 + 6 b^6 c^2 d^2 a^2 + b^2 c^2 d^8 + b^2 c^4 d^6 - 
    b^4 c^2 d^6 - b^2 c^6 d^4 - 6 b^4 c^4 d^4 - \\
    &- b^6 c^2 d^4 - 
   b^2 c^8 d^2 - b^4 c^6 d^2 + b^6 c^4 d^2 + b^8 c^2 d^2) , \ \\
   E \ =& \ -24 (a d - b c) (a d + b c) (a c - b d) (a c + b d) (a b - 
     c d) (a b + c d) (a^2 + b^2 - c^2 - d^2) (a^2 - b^2 - c^2 + 
     d^2) \\
     &(a^2 + b^2 + c^2 + d^2)^2 \ . 
\end{align*}
One obtains then an explicit affine expression for the elliptic fibration $ \varphi_{{\rm Y}}$ as:
\begin{equation}
\label{firstexpr}
   v^2 \ = \ \left ( \ u-A(t) \ \right ) \left ( \ u-B(t) \ \right ) \left ( \ C(t)  \cdot u^2 + D  u+E \ \right )
   \end{equation} 
\subsection{Adjustments to Formula $(\ref{firstexpr})$}
Next, we shall perform a series of transformations on the formula in expression $(\ref{firstexpr})$ with the goal of making a 
comparison with $(\ref{secondrightside})$. First, we shall perform a change in the affine coordinates $(u,v)$ setting 
\begin{align*}
u_1  \ =& \ \frac{ \ u-B(t) \ }{ \ u-A(t) \ }  \\
v_1 \ =& \ \frac{ \ v  \left (  A(t)-B(t)  \right ) \ }{ \ \left ( u-A(t) \right )^2  \left ( A^2(t)C(t) + A(t)D + E \right ) \ } \ .
\end{align*}
Intuitively, this operation amounts to sending $ A(t)$ to infinity and $ B(t)$ to zero. One obtains:
\begin{equation}
\label{secondexpr}
v_1^2 \ = \ u_1^3 + M(t) u_1^2 + N(t) u_1 
\end{equation}
where:
\begin{align*}
M(t) \ =& \ - \frac{ \ 2 A(t)B(t)C(t) + A(t)D + B(t)D + 2E \ }{ \ A(t)^2C(t) + A(t)D + E \ } \ , \\
N(t) \ =& \ \frac{ \ B^2(t) C(t) + B(t)D + E \ }{ \ A^2(t)C(t) + A(t) D + E \ }  \ .
   \end{align*}
An explicit evaluation of these two rational functions gives:
\begin{align*}
   M(t) \ =& \  - \frac{2 ( - t^3 + M_2 t^2 + M_1 t + M_0) }{t(M_3- t)(M_4+t)} \ , \\
   N(t) \ =& \  - \frac{ (N_1+t)(N_2+t)(N_3+t)  }{t(M_3- t)(M_4+t)} \ ,
\end{align*}
   where the coefficients $M_0$, $M_1$, $M_2$, $M_3$, $N_1$, $N_2$, $N_3$ are as follows:
\begin{align*}  
   M_0 \ =& \ 1728 (a^2+b^2+c^2+d^2)^3 (a^9 b c d-4 a^6 b^2 c^2 d^2-2 a^5 b^5 c d-2 a^5 b c^5 d-2 a^5 b c d^5+4 a^4 b^4 c^4+4 a^4 b^4 d^4 + \\
   &+4 a^4 c^4 d^4+8 a^3 b^3 c^3 d^3-4 a^2 b^6 c^2 d^2-4 a^2 b^2 c^6 d^2-4 a^2 b^2 c^2 d^6+a b^9 c d-2 a b^5 c^5 d-2 a b^5 c d^5+a b c^9 d- \\ 
   &-2 a b c^5 d^5+a b c d^9+4 b^4 c^4 d^4) \\
   M_1 \ =& \ -36 (a^2+b^2+c^2+d^2)^2 (a^8-32 a^5 b c d-2 a^4 b^4-2 a^4 c^4-2 a^4 d^4+136 a^2 b^2 c^2 d^2-32 a b^5 c d- \\
   & -32 a b c^5 d-32 a b c d^5
   +b^8-2 b^4 c^4-2 b^4 d^4+c^8-2 c^4 d^4+d^8) \\
   M_2 \ =& \ -12 (a^2+b^2+c^2+d^2) (a^4-12 a b c d+b^4+c^4+d^4) \\
   M_3 =& \ 96 a b c d (a^2+b^2+c^2+d^2) \\
   M_4 \ =& \ 6 (a^2+b^2+c^2+d^2) (a^2-2 a c+b^2-2 b d+c^2+d^2) (a^2+2 a c+b^2+2 b d+c^2+d^2) \\
   N_1 \ =& \ 6 (a-b-c-d) (a+b+c-d) (a+b-c+d) (a-b+c+d) (a^2+b^2+c^2+d^2) \\
   N_2 \ =& \ 6 (a^2+b^2+c^2+d^2) (a^2-2 a b+b^2+c^2-2 c d+d^2) (a^2+2 a b+b^2+c^2+2 c d+d^2) \\
   N_3 \ =& \ 6 (a^2+b^2+c^2+d^2) (a^2-2 a d+b^2-2 b c+c^2+d^2) (a^2+2 a d+b^2+2 b c+c^2+d^2) 
\end{align*}
Next, we get rid of denominators in $(\ref{secondexpr})$, making a change:
   $$ \widetilde{M}(t) = q \cdot M(t) \ \ \ \  \widetilde{N}(t) = q^2\cdot  N(t) $$
   where 
   $$ q \ = \ t(M_3- t)(M_4+t) \ . $$
Under the appropriate change $ (u_1, v_1) \rightarrow (u_2, v_2) $ in affine coordinates, one obtains:
\begin{equation}
\label{thirdexpr}
 v_2^2 \ = \ u_2^3 - 2 ( - t^3 + M_2 t^2 + M_1 t + M_0) u_2^2 - t(M_3- t)(M_4+t)(N_1+t)(N_2+t)(N_3+t)u_2 \ .
\end{equation}
We then eliminate the degree-two term from the polynomial $\widetilde{M}(t)$ by changing the affine coordinate of the fibration: 
\begin{equation}
\varepsilon \ = \ t - \frac{M_2}{3} \ = \ t + 4 (a^2+b^2+c^2+d^2) (a^4-12 a b c d+b^4+c^4+d^4) \ .
\end{equation}
At this point, the formula starts to clear miraculously. One obtains: 
\begin{equation}
\label{fourthexpr}
v_2^2 \ = \ u_2^3 + \widetilde{M}(\varepsilon) \cdot u_2^2 + \widetilde{N}(\varepsilon) \cdot u_2 
\end{equation}
with:
\begin{align*}
\widetilde{M}(\varepsilon) \ =& \ 2\varepsilon^3-24P_2^2P_8 \varepsilon-32P_2^3P_{12} \\
\widetilde{N}(\varepsilon) \ =& \ 
\varepsilon^6-24P_2^2P_8\varepsilon^4-32P_2^3P_{12}\varepsilon^3+144P_2^4P_8^2\varepsilon^2+384P_2^5P_{20}\varepsilon+254P_2^6P_{24} \ ,
\end{align*}
where the terms $P_2$, $P_8$, $P_{12}$, $P_{20}$, $P_{24}$ are homogeneous polynomials in the fundamental theta constants $a,b,c,d$. 
The precise form of $P_2$, $P_8$, $P_{12}$, $P_{20}$, $P_{24}$ is given in the appendix Section $\ref{specialpol}$.    
A final modification of the affine parameter of the fibration: 
$$ \eta = \frac{\varepsilon}{P_2} $$
allows one to describe this fibration as elliptic fibration $ \varphi_{{\rm Y}}$ as follows:
\begin{equation}
\label{fifthexpr}
v_3^2 \ = \ u_3^3 + \hat{M}(\eta) \cdot u_3^2 + \hat{N}(\eta) \cdot u_3 
\end{equation} 
\begin{equation}
\label{poly33}
 \hat{M}(\eta) \ = \  \frac{(-4)}{P_2^3} \cdot \widetilde{M}(P_2 \eta)\ = \ -4 \left (  2\eta^3 -24 P_8 \eta - 32 P_{12} \right ) 
 \end{equation}
\begin{equation} 
\label{poly66}
\hat{N}(\eta) \ = \ \frac{(-4)^2}{P_2^6} \cdot \widetilde{N}(P_2 \eta) \ = \ 16 \left ( \ 
\eta^6-24P_8 \eta^4-32P_{12}\eta^3+144P_8^2\eta^2+384P_{20}\eta+254P_{24} \ \right ) \ .
\end{equation}
\subsection{Matching of the Two Interpretations}
Comparing $(\ref{secondrightside})$ and $\ref{fifthexpr}$, one obtains that the two affine forms describe isomorphic elliptic fibration 
if and only if the following identities hold, up to a common weighted scaling of type $(2,3,5,6)$:
\begin{align*}
\alpha \ =& \ 2^4 P_8 \\
\beta \ =& \ 2^6 P_{12} \\
\gamma \ =& \ 2^{10} \cdot 3 \cdot \left ( P_{20} - P_8 P_{12} \right ) \ \\ 
\delta \ =& \ 2^{12} \left ( P_{12}^2 - P_{24} \right )   \ . 
\end{align*}
Via identities $(\ref{eqforp20})$ and $(\ref{eqforp24})$ the above provides the following identity of weighted points in ${\rm WP}(2,3,5,6)$:
\begin{equation}
\label{almostthere}
\left [ \ \alpha, \beta, \gamma, \delta \ \right ] 
\ = \ \left [ \ 2^4 P_8, \ 2^6 P_{12}, \ - 2^{14} 3^5 Q_{20}, \ 2^{16} 3^5 Q_{24} \ \right ] \ .
\end{equation}
After taking into account formulas $(\ref{tesst1})$, identity $(\ref{almostthere})$ becomes:
\begin{equation}
\left [ \ \alpha, \beta, \gamma, \delta \ \right ] \ = \ \left [ \ \mathcal{E}_4, \ \mathcal{E}_6, \ 2^{12} 3^5 \mathcal{C}_{10} , 
\ 2^{12} 3^{6} \mathcal{C}_{12} \ \right ] \ . \ 
\end{equation}
This completes the proof of Theorem $\ref{maincomp}$.
\section{Appendix}
During the computation presented in this paper, a few special polynomials played an important role. 
We include their precise form in this appendix section. The homogeneous polynomials of this section 
have as parameters the four fundamental theta constants $a$, $b$, $c$, $d$ of Section $ \ref{therealkummerside} $. 
The polynomials presented below are available in electronic {\tt Mathematica} format at {\tt http://www.arch.umsl.edu/$\widetilde{ \ }$clingher/siegel-paper/Mathematica/}.
\subsection{Special Polynomials : $P_2$, $P_8$, $P_{12}$, $P_{20}$, $P_{24}$, $Q_{20}$, $Q_{24}$}
\label{specialpol}
\begin{align*}
P_2 \ =& \ a^2+b^2+c^2+d^2 \\ 
P_8 \ =& \ a^8+14 a^4 b^4+14 a^4 c^4+14 a^4 d^4+168 a^2 b^2 c^2 d^2+b^8+14 b^4 c^4+14 b^4 d^4+c^8+14 c^4 d^4+d^8   \\ 
P_{12} \ =& \ 
a^{12}-33 a^8 b^4-33 a^8 c^4-33 a^8 d^4+792 a^6 b^2 c^2 d^2-33 a^4 b^8+ 330 a^4 b^4 c^4+ 330 a^4 b^4 d^4-33 a^4 c^8+ \\
&+330 a^4 c^4 d^4 -33 a^4 d^8+792 a^2 b^6 c^2 d^2+792 a^2 b^2 c^6 d^2+792 a^2 b^2 c^2 d^6+b^{12}-33 b^8 c^4-33 b^8 d^4-33 b^4 c^8+\\
&+330 b^4 c^4 d^4 -33 b^4 d^8+c^{12}-33 c^8 d^4-33 c^4 d^8+d^{12} \\
P_{20} \ =& \ 
a^{20}-19 b^4 a^{16}-19 c^4 a^{16}-19 d^4 a^{16}-336 b^2 c^2 d^2 a^{14}-494 b^8 a^{12}-494 c^8 a^{12}-494 d^8 a^{12}+716 b^4 c^4 a^{12}+ \\
&
+716 b^4 d^4 a^{12} + 716 c^4 d^4 a^{12}+7632 b^2 c^2 d^6 a^{10}+7632 b^2 c^6 d^2 a^{10}+7632 b^6 c^2 d^2 a^{10}-494 b^{12} a^8-494 c^{12} a^8- \\
&
-494 d^{12} a^8+1038 b^4 c^8 a^8+1038 b^4 d^8 a^8+1038 c^4 d^8 a^8+1038 b^8 c^4 a^8+1038 b^8 d^4 a^8+1038 c^8 d^4 a^8+ \\
&
+129012 b^4 c^4 d^4 a^8+7632 b^2 c^2 d^{10} a^6+106848 b^2 c^6 d^6 a^6+106848 b^6 c^2 d^6 a^6+7632 b^2 c^{10} d^2 a^6+ \\
&
+106848 b^6 c^6 d^2 a^6+7632 b^{10} c^2 d^2 a^6-19 b^{16} a^4-19 c^{16} a^4-19 d^{16} a^4+716 b^4 c^{12} a^4+716 b^4 d^{12} a^4+ \\
&
+716 c^4 d^{12} a^4+1038 b^8 c^8 a^4+1038 b^8 d^8 a^4+1038 c^8 d^8 a^4+129012 b^4 c^4 d^8 a^4+716 b^{12} c^4 a^4+716 b^{12} d^4 a^4+ \\
&
+716 c^{12} d^4 a^4+129012 b^4 c^8 d^4 a^4+129012 b^8 c^4 d^4 a^4-336 b^2 c^2 d^{14} a^2+7632 b^2 c^6 d^{10} a^2+7632 b^6 c^2 d^{10} a^2+ \\
&
+7632 b^2 c^{10} d^6 a^2+106848 b^6 c^6 d^6 a^2+7632 b^{10} c^2 d^6 a^2-336 b^2 c^{14} d^2 a^2+7632 b^6 c^{10} d^2 a^2+7632 b^{10} c^6 d^2 a^2- \\
&
-336 b^{14} c^2 d^2 a^2+b^{20}+c^{20}+d^{20}-19 b^4 c^{16}-19 b^4 d^{16}-19 c^4 d^{16}-494 b^8 c^{12}-494 b^8 d^{12}-494 c^8 d^{12}+ \\
&
+716 b^4 c^4 d^{12}-494 b^{12} c^8-494 b^{12} d^8-494 c^{12} d^8+1038 b^4 c^8 d^8+1038 b^8 c^4 d^8-19 b^{16} c^4-19 b^{16} d^4- \\
&
-19 c^{16} d^4+716 b^4 c^{12} d^4+1038 b^8 c^8 d^4+716 b^{12} c^4 d^4 \\
Q_{20} \ =& \ (b c-a d) (a d+b c) (b d-a c) (a c+b d) (a b-c d) (a b+c d) (a^2+b^2-c^2-d^2) (-a^2+b^2+c^2-d^2) \\
&(-a^2+b^2-c^2+d^2) (a^2+b^2+c^2+d^2) \\
\end{align*}
\begin{align*}
P_{24} \ =& \ 
 (a^4 - 12 a b c d + b^4 + c^4 + d^4) (a^4 + 12 a b c d + b^4 + c^4 + d^4) \\
&  (a^4 - 6 a^2 b^2 - 6 a^2 c^2 - 6 a^2 d^2 + b^4 - 6 b^2 c^2 - 6 b^2 d^2 + c^4 - 6 c^2 d^2 + d^4) \\
& (a^4 - 6 a^2 b^2 + 6 a^2 c^2 + 6 a^2 d^2 + b^4 + 6 b^2 c^2 + 6 b^2 d^2 + c^4 - 6 c^2 d^2 + d^4) \\
& (a^4 + 6 a^2 b^2 - 6 a^2 c^2 + 6 a^2 d^2 + b^4 + 6 b^2 c^2 - 6 b^2 d^2 + c^4 + 6 c^2 d^2 + d^4) \\
& (a^4 + 6 a^2 b^2 + 6 a^2 c^2 - 6 a^2 d^2 + b^4 - 6 b^2 c^2 + 6 b^2 d^2 + c^4 + 6 c^2 d^2 + d^4) \\
%
%
%
Q_{24} \ =& \ 
b^2 c^2 d^2 a^{18}+2 b^4 c^4 a^{16}+2 b^4 d^4 a^{16}+2 c^4 d^4 a^{16}-12 b^2 c^2 d^6 a^{14}-12 b^2 c^6 d^2 a^{14}-12 b^6 c^2 d^2 a^{14}-2 b^4 c^8 a^{12}- \\
& -2 b^4 d^8 a^{12}-2 c^4 d^8 a^{12}-2 b^8 c^4 a^{12}-2 b^8 d^4 a^{12}-2 c^8 d^4 a^{12}+76 b^4 c^4 d^4 a^{12}+22 b^2 c^2 d^{10} a^{10}-52 b^2 c^6 d^6 a^{10}- \\
&-52 b^6 c^2 d^6 a^{10}+22 b^2 c^{10} d^2 a^{10}-52 b^6 c^6 d^2 a^{10}+22 b^{10} c^2 d^2 a^{10}-2 b^4 c^{12} a^8-2 b^4 d^{12} a^8-2 c^4 d^{12} a^8+\\
&+36 b^8 c^8 a^8+36 b^8 d^8 a^8+36 c^8 d^8 a^8+36 b^4 c^4 d^8 a^8-2 b^{12} c^4 a^8-2 b^{12} d^4 a^8-2 c^{12} d^4 a^8+36 b^4 c^8 d^4 a^8+\\
&+36 b^8 c^4 d^4 a^8-12 b^2 c^2 d^{14} a^6-52 b^2 c^6 d^{10} a^6-52 b^6 c^2 d^{10} a^6-52 b^2 c^{10} d^6 a^6-8 b^6 c^6 d^6 a^6-52 b^{10} c^2 d^6 a^6-\\
&-12 b^2 c^{14} d^2 a^6-52 b^6 c^{10} d^2 a^6-52 b^{10} c^6 d^2 a^6-12 b^{14} c^2 d^2 a^6+2 b^4 c^{16} a^4+2 b^4 d^{16} a^4+2 c^4 d^{16} a^4-2 b^8 c^{12} a^4-\\
&-2 b^8 d^{12} a^4-2 c^8 d^{12} a^4+76 b^4 c^4 d^{12} a^4-2 b^{12} c^8 a^4-2 b^{12} d^8 a^4-2 c^{12} d^8 a^4+36 b^4 c^8 d^8 a^4+36 b^8 c^4 d^8 a^4+\\
&+2 b^{16} c^4 a^4+2 b^{16} d^4a^4+2 c^{16} d^4 a^4+76 b^4 c^{12} d^4 a^4+36 b^8 c^8 d^4 a^4+76 b^{12} c^4 d^4 a^4+b^2 c^2 d^{18} a^2-12 b^2 c^6 d^{14} a^2-\\
&-12 b^6 c^2 d^{14} a^2+22 b^2 c^{10} d^{10} a^2-52 b^6 c^6 d^{10} a^2+22 b^{10} c^2 d^{10} a^2-12 b^2 c^{14} d^6 a^2-52 b^6 c^{10} d^6 a^2-52 b^{10} c^6 d^6 a^2-\\
&-12 b^{14} c^2 d^6 a^2+b^2 c^{18} d^2 a^2-12 b^6 c^{14} d^2 a^2+22 b^{10} c^{10} d^2 a^2-12 b^{14} c^6 d^2 a^2+b^{18} c^2 d^2 a^2+2 b^4 c^4 d^{16}-\\
&-2 b^4 c^8 d^{12}-2 b^8 c^4 d^{12}-2 b^4 c^{12} d^8+36 b^8 c^8 d^8-2 b^{12} c^4 d^8+2 b^4 c^{16} d^4-2 b^8 c^{12} d^4-2 b^{12} c^8 d^4+2 b^{16} c^4 d^4 
\end{align*}
The above polynomials satisfy the following relations:
\begin{equation}
\label{eqforp20}
 P_{20}-P_{8} \cdot P_{12} \ = \ - 2^4 3^4 Q_{20}  
\end{equation}
\begin{equation}
\label{eqforp24}
 P_{12}^2 - P_{24} \ = \ 2^4 \cdot 3^5 \cdot Q_{24} 
\end{equation}
\subsection{Coefficients of the Quintic ${\rm QIN}_2(x,y,z)$}
\label{coefquint2}
\begin{align*} 
k_{500} \ =& \  0 \\
k_{005} \ =&  \ 0 
\\
 k_{050} \ =& \ -8 (b c + a d)^2 (a b - c d)^3 (a^2 + b^2 - c^2 - d^2) (-a^2 + b^2 - c^2 + d^2)^2\\
 k_{410} \ =& \ 4 (b c + a d) (a c + b d)^3 (a^2 + b^2 - c^2 - d^2)^2 (-a^2 + b^2 - c^2 + d^2)^2 \\
 k_{401} \ =& \ 16 (b c - a d)^3 (b c + a d) (a c + b d)^2 (a^2 + b^2 - c^2 - d^2) (a^2 + b^2 + c^2 + d^2) \\
 k_{140} \ =& \ 4 (b c + a d) (a b - c d)^2 (a^2 + b^2 - c^2 - d^2) (-a^2 + b^2 - 
   c^2 + d^2)^2 (a^3 b + a b^3 - 7 a b c^2 - 7 a^2 c d - \\
   &
 - 7 b^2 c d + 
   c^3 d - 7 a b d^2 + c d^3) \\
 k_{041} \ =& \ -16 (b c + a d) (a b - c d)^3 (a^2 + b^2 - c^2 - d^2) (-a^2 + b^2 - 
   c^2 + d^2) (a^3 c - 2 a b^2 c + a c^3 - 2 a^2 b d + \\
   & + b^3 d -2 b c^2 d - 2 a c d^2 + b d^3) \\
k_{320} \ =& \ 4 (b c + a d) (a c + b d)^2 (a^2 + b^2 - c^2 - d^2) (-a^2 + b^2 - 
   c^2 + d^2)^2 (3 a^3 b + 3 a b^3 - 5 a b c^2 - 5 a^2 c d - \\
   &
 -  5 b^2 c d + 3 c^3 d - 5 a b d^2 + 3 c d^3) \\
 k_{302} \ =& \ -16 (-b c + a d)^2 (b c + a d) (a c + b d)^2 (a^2 + b^2 + c^2 + 
   d^2) (a^3 b + a b^3 - 3 a b c^2 + 3 a^2 c d + 3 b^2 c d - \\
   &-c^3 d - 3 a b d^2 - c d^3)  \\
k_{230} \ =& \ 12 (b c + a d) (a c + b d) (a b - c d) (a^2 + b^2 - c^2 - 
   d^2) (-a^2 + b^2 - c^2 + d^2)^2 (a^3 b + a b^3 - 3 a b c^2 - \\
   & -3 a^2 c d - 3 b^2 c d + c^3 d - 3 a b d^2 + c d^3) 
   \end{align*}
   \begin{align*}
k_{221} \ =& \ 8 (a b - c d) (2 a^{11} b c^2 - a^9 b^3 c^2 - 4 a^7 b^5 c^2 + 
   2 a^5 b^7 c^2 + 2 a^3 b^9 c^2 - a b^{11} c^2 - 3 a^9 b c^4 + 
   10 a^7 b^3 c^4 - \\
   &-4 a^5 b^5 c^4 - 22 a^3 b^7 c^4 - 5 a b^9 c^4 - 
   6 a^7 b c^6 - 8 a^5 b^3 c^6 + 34 a^3 b^5 c^6 + 4 a b^7 c^6 + 
   4 a^5 b c^8 - 22 a^3 b^3 c^8 + \\
   &+6 a b^5 c^8 + 4 a^3 b c^{10} - 
   3 a b^3 c^{10} - a b c^{12} + a^{12} c d - 3 a^8 b^4 c d + 
   3 a^4 b^8 c d - b^{12} c d - 4 a^{10} c^3 d + a^8 b^2 c^3 d + \\
   & + 
   8 a^6 b^4 c^3 d - 6 a^4 b^6 c^3 d - 12 a^2 b^8 c^3 d - 
   3 b^{10} c^3 d - 4 a^8 c^5 d - 18 a^6 b^2 c^5 d + 18 a^4 b^4 c^5 d + 
   38 a^2 b^6 c^5 d + \\
   &+ 6 b^8 c^5 d + 6 a^6 c^7 d -12 a^4 b^2 c^7 d - 
   26 a^2 b^4 c^7 d + 4 b^6 c^7 d + 3 a^4 c^9 d + 6 a^2 b^2 c^9 d - 
   5 b^4 c^9 d - 2 a^2 c^{11} d - \\
   &-b^2 c^{11} d + a^{11} b d^2 - 
   2 a^9 b^3 d^2 - 2 a^7 b^5 d^2 + 4 a^5 b^7 d^2 + a^3 b^9 d^2 - 
   2 a b^{11} d^2 - 6 a^9 b c^2 d^2 + 16 a^7 b^3 c^2 d^2 - \\
   &-16 a^3 b^7 c^2 d^2 + 6 a b^9 c^2 d^2 +  12 a^7 b c^4 d^2 - 
   34 a^5 b^3 c^4 d^2 + 24 a^3 b^5 c^4 d^2 - 26 a b^7 c^4 d^2 + 
   18 a^5 b c^6 d^2 - \\
   &-8 a^3 b^3 c^6 d^2 + 38 a b^5 c^6 d^2 - 
   a^3 b c^8 d^2 -12 a b^3 c^8 d^2 + 3 a^{10} c d^3 + 
   12 a^8 b^2 c d^3 + 6 a^6 b^4 c d^3 - 8 a^4 b^6 c d^3 - \\
   &- a^2 b^8 c d^3 + 4 b^{10} c d^3 +22 a^8 c^3 d^3 + 
   8 a^6 b^2 c^3 d^3 - 8 a^2 b^6 c^3 d^3 - 22 b^8 c^3 d^3 + 
   8 a^6 c^5 d^3 + 34 a^4 b^2 c^5 d^3 + \\
   &+24 a^2 b^4 c^5 d^3 + 
   34 b^6 c^5 d^3 - 10 a^4 c^7 d^3 - 16 a^2 b^2 c^7 d^3 - 
   22 b^4 c^7 d^3 + a^2 c^9 d^3 + 2 b^2 c^9 d^3 + 5 a^9 b d^4 + \\
   &+ 22 a^7 b^3 d^4 + 4 a^5 b^5 d^4 -10 a^3 b^7 d^4 + 3 a b^9 d^4 + 
   26 a^7 b c^2 d^4 - 24 a^5 b^3 c^2 d^4 + 34 a^3 b^5 c^2 d^4 - 
   12 a b^7 c^2 d^4 - \\
   &-18 a^5 b c^4 d^4 +18 a b^5 c^4 d^4 - 
   8 a^3 b c^6 d^4 - 6 a b^3 c^6 d^4 + 3 a b c^8 d^4 - 6 a^8 c d^5 - 
   38 a^6 b^2 c d^5 - 18 a^4 b^4 c d^5 + \\
   &+18 a^2 b^6 c d^5 + 
   4 b^8 c d^5 -34 a^6 c^3 d^5 - 24 a^4 b^2 c^3 d^5 - 
   34 a^2 b^4 c^3 d^5 - 8 b^6 c^3 d^5 + 4 a^4 c^5 d^5 - 
   4 b^4 c^5 d^5 + \\
   &+4 a^2 c^7 d^5 + 2 b^2 c^7 d^5 - 4 a^7 b d^6 -
   34 a^5 b^3 d^6 + 8 a^3 b^5 d^6 + 6 a b^7 d^6 - 38 a^5 b c^2 d^6 + 
   8 a^3 b^3 c^2 d^6 - 18 a b^5 c^2 d^6 + \\
   &+6 a^3 b c^4 d^6 + 
   8 a b^3 c^4 d^6 - 4 a^6 c d^7 +26 a^4 b^2 c d^7 + 
   12 a^2 b^4 c d^7 - 6 b^6 c d^7 + 22 a^4 c^3 d^7 + 
   16 a^2 b^2 c^3 d^7 + \\
   &+10 b^4 c^3 d^7 - 2 a^2 c^5 d^7 - 
   4 b^2 c^5 d^7 - 6 a^5 b d^8 + 22 a^3 b^3 d^8 - 4 a b^5 d^8 + 
   12 a^3 b c^2 d^8 + a b^3 c^2 d^8 - 3 a b c^4 d^8 + \\
   &+5 a^4 c d^9 - 
   6 a^2 b^2 c d^9 -3 b^4 c d^9 
   -2 a^2 c^3 d^9 - b^2 c^3 d^9 + 
   3 a^3 b d^{10} - 4 a b^3 d^{10} + a^2 c d^{11} + 2 b^2 c d^{11} + a b d^{12}) \\ & \\
   %
   %
   %
k_{212} \ =& \  (a c + b d) (5 a^{12} b c - 6 a^{10} b^3 c - 9 a^8 b^5 c + 
   12 a^6 b^7 c + 3 a^4 b^9 c - 6 a^2 b^{11} c + b^{13} c - 
   6 a^{10} b c^3 
   +18 a^8 b^3 c^3 - \\
   &-20 a^6 b^5 c^3 - 84 a^4 b^7 c^3 - 
   38 a^2 b^9 c^3 + 2 b^{11} c^3 - 9 a^8 b c^5 - 4 a^6 b^3 c^5 + 
   178 a^4 b^5 c^5 + 60 a^2 b^7 c^5 -b^9 c^5 + \\
   &+12 a^6 b c^7 - 
   52 a^4 b^3 c^7 + 76 a^2 b^5 c^7 - 4 b^7 c^7 + 3 a^4 b c^9 - 
   22 a^2 b^3 c^9 - b^5 c^9 - 6 a^2 b c^{11} +2 b^3 c^{11} + b c^{13} + \\
   &+
   a^{13} d - 6 a^{11} b^2 d + 3 a^9 b^4 d + 12 a^7 b^6 d - 9 a^5 b^8 d - 
   6 a^3 b^{10} d + 5 a b^{12} d - 6 a^{11} c^2 d - 6 a^9 b^2 c^2 d + \\
   &+
   12 a^7 b^4 c^2 d + 60 a^5 b^6 c^2 d + 58 a^3 b^8 c^2 d + 
   10 a b^{10} c^2 d + 3 a^9 c^4 d - 4 a^7 b^2 c^4 d + 
   10 a^5 b^4 c^4 d - \\
   &-132 a^3 b^6 c^4 d -37 a b^8 c^4 d + 
   12 a^7 c^6 d + 28 a^5 b^2 c^6 d - 148 a^3 b^4 c^6 d - 
   84 a b^6 c^6 d - 9 a^5 c^8 d + 42 a^3 b^2 c^8 d - \\
   &-37 a b^4 c^8 d - 
   6 a^3 c^{10} d + 10 a b^2 c^{10} d + 5 a c^{12} d + 10 a^{10} b c d^2 + 
   58 a^8 b^3 c d^2 + 60 a^6 b^5 c d^2 + 12 a^4 b^7 c d^2 - \\
   &-
   6 a^2 b^9 c d^2 -6 b^{11} c d^2 + 42 a^8 b c^3 d^2 - 
   56 a^6 b^3 c^3 d^2 + 60 a^4 b^5 c^3 d^2 - 88 a^2 b^7 c^3 d^2 - 
   22 b^9 c^3 d^2 + 28 a^6 b c^5 d^2 + \\
   &+ 60 a^4 b^3 c^5 d^2 + 
   476 a^2 b^5 c^5 d^2 + 76 b^7 c^5 d^2 - 4 a^4 b c^7 d^2 - 
   88 a^2 b^3 c^7 d^2 + 60 b^5 c^7 d^2 - 6 a^2 b c^9 d^2 - 
   38 b^3 c^9 d^2 - \\
   &-6 b c^{11} d^2 +2 a^{11} d^3 - 38 a^9 b^2 d^3 - 
   84 a^7 b^4 d^3 - 20 a^5 b^6 d^3 + 18 a^3 b^8 d^3 - 6 a b^{10} d^3 - 
   22 a^9 c^2 d^3 - \\
   &-88 a^7 b^2 c^2 d^3 + 60 a^5 b^4 c^2 d^3 -
   56 a^3 b^6 c^2 d^3 + 42 a b^8 c^2 d^3 - 52 a^7 c^4 d^3 + 
   60 a^5 b^2 c^4 d^3 - 340 a^3 b^4 c^4 d^3 - \\
   &-148 a b^6 c^4 d^3 - 
   4 a^5 c^6 d^3 - 56 a^3 b^2 c^6 d^3 - 132 a b^4 c^6 d^3 + 
   18 a^3 c^8 d^3 + 58 a b^2 c^8 d^3 - 6 a c^{10} d^3 - 
   37 a^8 b c d^4 - \\
   &-132 a^6 b^3 c d^4 + 10 a^4 b^5 c d^4 - 
   4 a^2 b^7 c d^4 +3 b^9 c d^4 - 148 a^6 b c^3 d^4 - 
   340 a^4 b^3 c^3 d^4 + 60 a^2 b^5 c^3 d^4 - 52 b^7 c^3 d^4 + \\
   &+
   10 a^4 b c^5 d^4 + 60 a^2 b^3 c^5 d^4 + 178 b^5 c^5 d^4 + 
   12 a^2 b c^7 d^4 - 84 b^3 c^7 d^4 + 3 b c^9 d^4 - a^9 d^5 + 
   60 a^7 b^2 d^5 + \\
   &+178 a^5 b^4 d^5 - 4 a^3 b^6 d^5 - 9 a b^8 d^5 + 
   76 a^7 c^2 d^5 + 476 a^5 b^2 c^2 d^5 +60 a^3 b^4 c^2 d^5 + 
   28 a b^6 c^2 d^5 + 178 a^5 c^4 d^5 + \\
   &+60 a^3 b^2 c^4 d^5 + 
   10 a b^4 c^4 d^5 - 20 a^3 c^6 d^5 + 60 a b^2 c^6 d^5 - 
   9 a c^8 d^5 - 84 a^6 b c d^6 - 148 a^4 b^3 c d^6 + 
   28 a^2 b^5 c d^6 + \\
   &+12 b^7 c d^6 - 132 a^4 b c^3 d^6 - 
   56 a^2 b^3 c^3 d^6 - 4 b^5 c^3 d^6 + 60 a^2 b c^5 d^6 -
   20 b^3 c^5 d^6 + 12 b c^7 d^6 - 4 a^7 d^7 + \\
   &+76 a^5 b^2 d^7 - 
   52 a^3 b^4 d^7 + 12 a b^6 d^7 + 60 a^5 c^2 d^7 - 
   88 a^3 b^2 c^2 d^7  4 a b^4 c^2 d^7 - 84 a^3 c^4 d^7 + 
   12 a b^2 c^4 d^7 + \\
   &+12 a c^6 d^7 - 37 a^4 b c d^8 + 
   42 a^2 b^3 c d^8 - 9 b^5 c d^8 + 58 a^2 b c^3 d^8 +
   18 b^3 c^3 d^8 - 9 b c^5 d^8 - a^5 d^9 - 22 a^3 b^2 d^9 + \\
   &+
   3 a b^4 d^9 - 38 a^3 c^2 d^9 - 6 a b^2 c^2 d^9 + 3 a c^4 d^9 + 
   10 a^2 b c d^{10} - 6 b^3 c d^{10} -6 b c^3 d^{10} + 2 a^3 d^{11} - \\
   &-
   6 a b^2 d^{11} - 6 a c^2 d^{11} + 5 b c d^{12} + a d^{13}) \\
     \end{align*}
  \begin{align*}
 k_{122} \ =& \  (a b - c d) (5 a^{12} b c - 6 a^{10} b^3 c - 9 a^8 b^5 c + 
   12 a^6 b^7 c + 3 a^4 b^9 c - 6 a^2 b^{11} c + b^{13} c - 
   14 a^{10} b c^3 + 58 a^8 b^3 c^3 + \\
   &+4 a^6 b^5 c^3 - 92 a^4 b^7 c^3 - 
   22 a^2 b^9 c^3 + 2 b^{11} c^3 - 17 a^8 b c^5 - 20 a^6 b^3 c^5 + 
   234 a^4 b^5 c^5 + 28 a^2 b^7 c^5 - b^9 c^5 + \\
   &+20 a^6 b c^7 - 
   108 a^4 b^3 c^7 + 28 a^2 b^5 c^7 - 4 b^7 c^7 + 11 a^4 b c^9 - 
   22 a^2 b^3 c^9 - b^5 c^9 - 6 a^2 b c^{11} + 2 b^3 c^{11} + b c^{13} + \\
   &+
   a^{13} d - 6 a^{11} b^2 d +3 a^9 b^4 d + 12 a^7 b^6 d - 9 a^5 b^8 d - 
   6 a^3 b^{10} d + 5 a b^{12} d - 6 a^{11} c^2 d + 34 a^9 b^2 c^2 d - \\
   &-
   12 a^7 b^4 c^2 d - 12 a^5 b^6 c^2 d +50 a^3 b^8 c^2 d + 
   10 a b^{10} c^2 d + 11 a^9 c^4 d - 4 a^7 b^2 c^4 d + 
   146 a^5 b^4 c^4 d - 52 a^3 b^6 c^4 d - \\
   &-5 a b^8 c^4 d + 
   20 a^7 c^6 d +4 a^5 b^2 c^6 d - 36 a^3 b^4 c^6 d - 
   20 a b^6 c^6 d - 17 a^5 c^8 d + 58 a^3 b^2 c^8 d - 5 a b^4 c^8 d - \\
   &-
   14 a^3 c^{10} d + 10 a b^2 c^{10} d + 5 a c^{12} d +10 a^{10} b c d^2 + 
   50 a^8 b^3 c d^2 - 12 a^6 b^5 c d^2 - 12 a^4 b^7 c d^2 + 
   34 a^2 b^9 c d^2 - \\
   &-6 b^{11} c d^2 + 58 a^8 b c^3 d^2 - 
   120 a^6 b^3 c^3 d^2 
   +12 a^4 b^5 c^3 d^2 - 248 a^2 b^7 c^3 d^2 - 
   22 b^9 c^3 d^2 + 4 a^6 b c^5 d^2 - \\
   &-12 a^4 b^3 c^5 d^2 + 
   332 a^2 b^5 c^5 d^2 + 28 b^7 c^5 d^2 - 4 a^4 b c^7 d^2 - 
   248 a^2 b^3 c^7 d^2 + 28 b^5 c^7 d^2 + 34 a^2 b c^9 d^2 - \\
   &-
   22 b^3 c^9 d^2 - 6 b c^{11} d^2 + 2 a^{11} d^3 - 22 a^9 b^2 d^3 - 
   92 a^7 b^4 d^3 + 4 a^5 b^6 d^3 + 58 a^3 b^8 d^3 -14 a b^{10} d^3 - \\
   &-
   22 a^9 c^2 d^3 - 248 a^7 b^2 c^2 d^3 + 12 a^5 b^4 c^2 d^3 - 
   120 a^3 b^6 c^2 d^3 + 58 a b^8 c^2 d^3 - 108 a^7 c^4 d^3 - 
   12 a^5 b^2 c^4 d^3 - \\
   &-100 a^3 b^4 c^4 d^3 - 36 a b^6 c^4 d^3 - 
   20 a^5 c^6 d^3 - 120 a^3 b^2 c^6 d^3 - 52 a b^4 c^6 d^3 + 
   58 a^3 c^8 d^3 + 50 a b^2 c^8 d^3 - \\
   &-6 a c^{10} d^3 - 5 a^8 b c d^4 - 
   52 a^6 b^3 c d^4 + 146 a^4 b^5 c d^4 - 4 a^2 b^7 c d^4 + 
   11 b^9 c d^4 - 36 a^6 b c^3 d^4 - 100 a^4 b^3 c^3 d^4 - \\
   &- 
   12 a^2 b^5 c^3 d^4 - 108 b^7 c^3 d^4 +146 a^4 b c^5 d^4 + 
   12 a^2 b^3 c^5 d^4 + 234 b^5 c^5 d^4 - 12 a^2 b c^7 d^4 - 
   92 b^3 c^7 d^4 + 3 b c^9 d^4 - \\
   &-a^9 d^5 + 28 a^7 b^2 d^5 + 
   234 a^5 b^4 d^5 -20 a^3 b^6 d^5 - 17 a b^8 d^5 + 28 a^7 c^2 d^5 + 
   332 a^5 b^2 c^2 d^5 - 12 a^3 b^4 c^2 d^5 + \\
   &+4 a b^6 c^2 d^5 + 
   234 a^5 c^4 d^5 + 12 a^3 b^2 c^4 d^5 +146 a b^4 c^4 d^5 + 
   4 a^3 c^6 d^5 - 12 a b^2 c^6 d^5 - 9 a c^8 d^5 - 20 a^6 b c d^6 - \\
   &- 
   36 a^4 b^3 c d^6 + 4 a^2 b^5 c d^6 + 20 b^7 c d^6 - 
   52 a^4 b c^3 d^6 
   -120 a^2 b^3 c^3 d^6 - 20 b^5 c^3 d^6 - 
   12 a^2 b c^5 d^6 + 4 b^3 c^5 d^6 + \\
   &+12 b c^7 d^6 - 4 a^7 d^7 + 
   28 a^5 b^2 d^7 - 108 a^3 b^4 d^7 + 20 a b^6 d^7 +28 a^5 c^2 d^7 - 
   248 a^3 b^2 c^2 d^7 - 4 a b^4 c^2 d^7 - \\
   &-92 a^3 c^4 d^7 - 
   12 a b^2 c^4 d^7 + 12 a c^6 d^7 - 5 a^4 b c d^8 + 
   58 a^2 b^3 c d^8 -17 b^5 c d^8 + 50 a^2 b c^3 d^8 + 
   58 b^3 c^3 d^8 - \\
   &-9 b c^5 d^8 - a^5 d^9 - 22 a^3 b^2 d^9 + 
   11 a b^4 d^9 - 22 a^3 c^2 d^9 + 34 a b^2 c^2 d^9 + 3 a c^4 d^9 
   + 
   10 a^2 b c d^{10} - 14 b^3 c d^{10} - \\
   &-6 b c^3 d^{10} + 2 a^3 d^{11} - 
   6 a b^2 d^{11} - 6 a c^2 d^{11} + 5 b c d^{12} + a d^{13}) \\
   %
   %
   %
   %
   %
   %
   %
k_{203} \ =& \ -32 (-b c + a d)^2 (b c + a d) (a c - b d) (a c + b d)^2 (a b + c d) (a^2 + b^2 + c^2 + d^2) \\
k_{032} \ =& \ -8 (-a c + b d) (a b - c d)^3 (a^2 + b^2 - c^2 - d^2) (-a^2 + b^2 - 
   c^2 + d^2) (a^3 c - 5 a b^2 c + a c^3 - 5 a^2 b d + \\
   &+b^3 d - 
   5 b c^2 d - 5 a c d^2 + b d^3) \\
k_{023} \ =& \ 16 (-a c + b d)^2 (a b - c d)^3 (a b + c d) (a^2 + b^2 - c^2 - d^2) (-a^2 + b^2 - c^2 + d^2) \\
k_{311} \ =& \ 
   4 (a c + b d) (-2 a^{11} b c^2 + a^9 b^3 c^2 + 4 a^7 b^5 c^2 - 
   2 a^5 b^7 c^2 - 2 a^3 b^9 c^2 + a b^{11} c^2 + a^9 b c^4 - 
   6 a^7 b^3 c^4 + \\
   &+12 a^5 b^5 c^4 
   +30 a^3 b^7 c^4 + 11 a b^9 c^4 + 
   4 a^7 b c^6 + 4 a^5 b^3 c^6 - 24 a^3 b^5 c^6 - 4 a b^7 c^6 - 
   2 a^5 b c^8 + 14 a^3 b^3 c^8 - \\
   &-12 a b^5 c^8 - 2 a^3 b c^{10} + 
   3 a b^3 c^{10} + a b c^{12} - a^{12} c d + 3 a^8 b^4 c d - 
   3 a^4 b^8 c d + b^{12} c d + 2 a^{10} c^3 d + 3 a^8 b^2 c^3 d - \\
   &-12 a^6 b^4 c^3 d - 22 a^4 b^6 c^3 d - 6 a^2 b^8 c^3 d +
   3 b^{10} c^3 d + 2 a^8 c^5 d + 14 a^6 b^2 c^5 d - 14 a^4 b^4 c^5 d - 
   38 a^2 b^6 c^5 d - \\
   &-12 b^8 c^5 d - 4 a^6 c^7 d + 4 a^4 b^2 c^7 d + 
   32 a^2 b^4 c^7 d - 4 b^6 c^7 d - a^4 c^9 d - 6 a^2 b^2 c^9 d + 
   11 b^4 c^9 d + 2 a^2 c^{11} d + \\
   &+b^2 c^{11} d - a^{11} b d^2 + 
   2 a^9 b^3 d^2 + 2 a^7 b^5 d^2 - 4 a^5 b^7 d^2 - a^3 b^9 d^2 + 
   2 a b^{11} d^2 + 6 a^9 b c^2 d^2 - 4 a^7 b^3 c^2 d^2 + \\
   &+
   4 a^3 b^7 c^2 d^2 - 6 a b^9 c^2 d^2 - 4 a^7 b c^4 d^2 + 
   44 a^5 b^3 c^4 d^2 + 20 a^3 b^5 c^4 d^2 + 32 a b^7 c^4 d^2 - 
   14 a^5 b c^6 d^2 + \\
   &+4 a^3 b^3 c^6 d^2 - 38 a b^5 c^6 d^2 - 
   3 a^3 b c^8 d^2 - 6 a b^3 c^8 d^2 - 3 a^{10} c d^3 + 
   6 a^8 b^2 c d^3 + 22 a^6 b^4 c d^3 +12 a^4 b^6 c d^3 - \\
   &-
   3 a^2 b^8 c d^3 - 2 b^{10} c d^3 - 14 a^8 c^3 d^3 - 
   4 a^6 b^2 c^3 d^3 + 4 a^2 b^6 c^3 d^3 + 14 b^8 c^3 d^3 - 
   4 a^6 c^5 d^3 -44 a^4 b^2 c^5 d^3 + \\
   &+20 a^2 b^4 c^5 d^3 - 
   24 b^6 c^5 d^3 + 6 a^4 c^7 d^3 + 4 a^2 b^2 c^7 d^3 + 
   30 b^4 c^7 d^3 - a^2 c^9 d^3 - 2 b^2 c^9 d^3 - 11 a^9 b d^4 - \\
   &- 
   30 a^7 b^3 d^4 - 12 a^5 b^5 d^4 + 6 a^3 b^7 d^4 - a b^9 d^4 - 
   32 a^7 b c^2 d^4 - 20 a^5 b^3 c^2 d^4 - 44 a^3 b^5 c^2 d^4 + 
   4 a b^7 c^2 d^4 + \\
   &+14 a^5 b c^4 d^4 
   - 14 a b^5 c^4 d^4 + 
   12 a^3 b c^6 d^4 - 22 a b^3 c^6 d^4 - 3 a b c^8 d^4 + 
   12 a^8 c d^5 + 38 a^6 b^2 c d^5 + 14 a^4 b^4 c d^5 - \\
   &- 
   14 a^2 b^6 c d^5 -2 b^8 c d^5 + 24 a^6 c^3 d^5 - 
   20 a^4 b^2 c^3 d^5 + 44 a^2 b^4 c^3 d^5 + 4 b^6 c^3 d^5 - 
   12 a^4 c^5 d^5 + 12 b^4 c^5 d^5 - \\
   &-4 a^2 c^7 d^5 - 2 b^2 c^7 d^5 +
   4 a^7 b d^6 + 24 a^5 b^3 d^6 - 4 a^3 b^5 d^6 - 4 a b^7 d^6 + 
   38 a^5 b c^2 d^6 - 4 a^3 b^3 c^2 d^6 + 14 a b^5 c^2 d^6 + \\
   &+
   22 a^3 b c^4 d^6 - 12 a b^3 c^4 d^6 +4 a^6 c d^7 - 
   32 a^4 b^2 c d^7 - 4 a^2 b^4 c d^7 + 4 b^6 c d^7 - 
   30 a^4 c^3 d^7 - 4 a^2 b^2 c^3 d^7 - \\
   & 
   - 6 b^4 c^3 d^7 + 
   2 a^2 c^5 d^7 + 4 b^2 c^5 d^7 + 12 a^5 b d^8 - 14 a^3 b^3 d^8 + 
   2 a b^5 d^8 + 6 a^3 b c^2 d^8 + 3 a b^3 c^2 d^8 + 3 a b c^4 d^8 - \\
   &-
   11 a^4 c d^9 
   +6 a^2 b^2 c d^9 + b^4 c d^9 + 2 a^2 c^3 d^9 + 
   b^2 c^3 d^9 - 3 a^3 b d^{10} + 2 a b^3 d^{10} - a^2 c d^{11} - 
   2 b^2 c d^{11} - a b d^{12}) \\
\end{align*}
\begin{align*}
k_{131} \ =& \ 4 (a b - c d)^2 (-2 a^{10} b c + a^8 b^3 c + 4 a^6 b^5 c - 
   2 a^4 b^7 c - 2 a^2 b^9 c + b^{11} c + 9 a^8 b c^3 - 
   22 a^6 b^3 c^3 - 4 a^4 b^5 c^3 + \\
   &+30 a^2 b^7 c^3 + 3 b^9 c^3 + 
   12 a^6 b c^5 + 20 a^4 b^3 c^5 - 64 a^2 b^5 c^5 - 4 b^7 c^5 - 
   10 a^4 b c^7 + 46 a^2 b^3 c^7 - 4 b^5 c^7 - \\
   &-10 a^2 b c^9 +
   3 b^3 c^9 + b c^{11} - a^{11} d + 2 a^9 b^2 d + 2 a^7 b^4 d - 
   4 a^5 b^6 d - a^3 b^8 d + 2 a b^{10} d + 10 a^9 c^2 d - \\
   &- 
   22 a^7 b^2 c^2 d - 6 a^5 b^4 c^2 d 
   +30 a^3 b^6 c^2 d + 
   4 a b^8 c^2 d + 10 a^7 c^4 d + 18 a^5 b^2 c^4 d - 
   66 a^3 b^4 c^4 d - 6 a b^6 c^4 d - \\
   &-12 a^5 c^6 d + 
   46 a^3 b^2 c^6 d 
   -6 a b^4 c^6 d - 9 a^3 c^8 d + 4 a b^2 c^8 d + 
   2 a c^{10} d - 4 a^8 b c d^2 - 30 a^6 b^3 c d^2 + 6 a^4 b^5 c d^2 + \\
   &+
   22 a^2 b^7 c d^2 - 10 b^9 c d^2 
   -46 a^6 b c^3 d^2 + 
   18 a^4 b^3 c^3 d^2 - 26 a^2 b^5 c^3 d^2 + 46 b^7 c^3 d^2 - 
   18 a^4 b c^5 d^2 - \\
   &- 26 a^2 b^3 c^5 d^2 - 64 b^5 c^5 d^2 
   + 
   22 a^2 b c^7 d^2 + 30 b^3 c^7 d^2 - 2 b c^9 d^2 - 3 a^9 d^3 - 
   30 a^7 b^2 d^3 + 4 a^5 b^4 d^3 + \\
   &+22 a^3 b^6 d^3 - 9 a b^8 d^3 - 
   46 a^7 c^2 d^3 + 26 a^5 b^2 c^2 d^3 
   -18 a^3 b^4 c^2 d^3 + 
   46 a b^6 c^2 d^3 - 20 a^5 c^4 d^3 - 18 a^3 b^2 c^4 d^3 - \\
   &-
   66 a b^4 c^4 d^3 + 22 a^3 c^6 d^3 + 30 a b^2 c^6 d^3 - a c^8 d^3 + 
   6 a^6 b c d^4 
   + 66 a^4 b^3 c d^4 - 18 a^2 b^5 c d^4 - 
   10 b^7 c d^4 + \\
   &+66 a^4 b c^3 d^4 + 18 a^2 b^3 c^3 d^4 + 
   20 b^5 c^3 d^4 + 6 a^2 b c^5 d^4 - 4 b^3 c^5 d^4 - 2 b c^7 d^4 
   +
   4 a^7 d^5 + 64 a^5 b^2 d^5 - \\
   &-20 a^3 b^4 d^5 - 12 a b^6 d^5 + 
   64 a^5 c^2 d^5 + 26 a^3 b^2 c^2 d^5 + 18 a b^4 c^2 d^5 + 
   4 a^3 c^4 d^5 - 6 a b^2 c^4 d^5 -4 a c^6 d^5 + \\
   &+ 6 a^4 b c d^6 - 
   46 a^2 b^3 c d^6 + 12 b^5 c d^6 - 30 a^2 b c^3 d^6 - 
   22 b^3 c^3 d^6 + 4 b c^5 d^6 + 4 a^5 d^7 - 46 a^3 b^2 d^7 
   +
   10 a b^4 d^7 - \\
   &
   -30 a^3 c^2 d^7 - 22 a b^2 c^2 d^7 + 2 a c^4 d^7 - 
   4 a^2 b c d^8 + 9 b^3 c d^8 + b c^3 d^8 - 3 a^3 d^9 + 
   10 a b^2 d^9 + 2 a c^2 d^9 - \\
   &-2 b c d^{10} - a d^{11}) 
 \end{align*}
 \begin{align*}
k_{113} \ =& \  (a c - b d) (a b - c d) (a^{12} - 2 a^{10} b^2 - a^8 b^4 + 4 a^6 b^6 - 
   a^4 b^8 - 2 a^2 b^{10} + b^{12} - 2 a^{10} c^2 + 10 a^8 b^2 c^2 
   + 
   12 a^6 b^4 c^2 - \\
   &- 12 a^4 b^6 c^2 - 10 a^2 b^8 c^2 + 2 b^{10} c^2 - 
   a^8 c^4 + 12 a^6 b^2 c^4 + 74 a^4 b^4 c^4 + 12 a^2 b^6 c^4 - 
   b^8 c^4 + 4 a^6 c^6 
   -12 a^4 b^2 c^6 + \\
   &+ 12 a^2 b^4 c^6 - 4 b^6 c^6 -
    a^4 c^8 - 10 a^2 b^2 c^8 - b^4 c^8 - 2 a^2 c^{10} + 2 b^2 c^{10} + 
   c^{12} - 8 a^9 b c d + 16 a^5 b^5 c d 
   -8 a b^9 c d + \\
   &+
   16 a^5 b c^5 d + 16 a b^5 c^5 d - 8 a b c^9 d + 2 a^{10} d^2 - 
   10 a^8 b^2 d^2 - 12 a^6 b^4 d^2 + 12 a^4 b^6 d^2 + 
   10 a^2 b^8 d^2 
   -2 b^{10} d^2 - \\
   &- 10 a^8 c^2 d^2 - 
   72 a^6 b^2 c^2 d^2 - 28 a^4 b^4 c^2 d^2 - 72 a^2 b^6 c^2 d^2 - 
   10 b^8 c^2 d^2 - 12 a^6 c^4 d^2 - 28 a^4 b^2 c^4 d^2 + \\
   &+
   28 a^2 b^4 c^4 d^2 + 12 b^6 c^4 d^2 + 12 a^4 c^6 d^2 - 
   72 a^2 b^2 c^6 d^2 + 12 b^4 c^6 d^2 + 10 a^2 c^8 d^2 - 
   10 b^2 c^8 d^2 - 2 c^{10} d^2 - \\
   &
   -64 a^3 b^3 c^3 d^3 - a^8 d^4 + 
   12 a^6 b^2 d^4 + 74 a^4 b^4 d^4 + 12 a^2 b^6 d^4 - b^8 d^4 + 
   12 a^6 c^2 d^4 + 28 a^4 b^2 c^2 d^4 - \\
   &-28 a^2 b^4 c^2 d^4 - 
   12 b^6 c^2 d^4 + 74 a^4 c^4 d^4 - 28 a^2 b^2 c^4 d^4 + 
   74 b^4 c^4 d^4 + 12 a^2 c^6 d^4 - 12 b^2 c^6 d^4 - c^8 d^4 + \\
   &+
   16 a^5 b c d^5 + 16 a b^5 c d^5 
   +16 a b c^5 d^5 - 4 a^6 d^6 + 
   12 a^4 b^2 d^6 - 12 a^2 b^4 d^6 + 4 b^6 d^6 + 12 a^4 c^2 d^6 - 
   72 a^2 b^2 c^2 d^6 + \\
   &+12 b^4 c^2 d^6 - 12 a^2 c^4 d^6 
   + 
   12 b^2 c^4 d^6 
   +4 c^6 d^6 - a^4 d^8 - 10 a^2 b^2 d^8 - b^4 d^8 - 
   10 a^2 c^2 d^8 + 10 b^2 c^2 d^8 - c^4 d^8 - \\
   &-8 a b c d^9 + 
   2 a^2 d^{10} - 2 b^2 d^{10} - 2 c^2 d^{10} + d^{12}) \\
   \end{align*}

\end{document}